\documentclass[11pt,reqno]{amsart} 

\usepackage[a4paper]{geometry}
\usepackage[utf8]{inputenc}
\usepackage[T1]{fontenc}
\usepackage{lmodern}
\usepackage{amsmath,amssymb}
\usepackage{amsthm}
\usepackage{paralist}
\usepackage{float}
\usepackage{accents}
\usepackage{color}
\usepackage[authoryear]{natbib}
\usepackage{bm}
\usepackage{lscape}
\usepackage{graphicx}
\usepackage{subfig}
\usepackage{longtable}
\usepackage{multibib}
\usepackage{amsaddr}
\newcites{suppl}{References} 

\numberwithin{equation}{section}
\usepackage{dsfont}
\usepackage[colorlinks=true,linkcolor=blue,citecolor=blue,pdfborder={0 0 0}]{hyperref}

\makeatletter
\def\blfootnote{\xdef\@thefnmark{}\@footnotetext}
\makeatother

\newcommand{\id}{\mathds{1}}
\newcommand{\ind}{\id}
\renewcommand{\epsilon}{\varepsilon}
\newcommand{\eps}{{\varepsilon}}
\renewcommand{\phi}{\varphi}
\renewcommand{\theta}{\vartheta}
\DeclareMathOperator{\cum}{cum} 
\DeclareMathOperator{\ident}{id} 
\newcommand{\ubar}[1]{\underaccent{\bar}{#1}}

\newcommand{\reals}{\mathbb{R}}
\newcommand{\R}{\mathbb{R}}
\newcommand{\integers}{\mathbb{Z}}
\newcommand{\Z}{\mathbb{Z}}
\newcommand{\naturals}{\mathbb{N}}
\newcommand{\N}{\mathbb{N}}

\newcommand{\pr}{\mathbb{P}}        
\newcommand{\Prob}{\mathbb{P}}        
\newcommand{\ex}{\mathbb{E}}        
\newcommand{\Exp}{\mathbb{E}}        
\newcommand{\var}{\textnormal{Var}} 
\newcommand{\cov}{\textnormal{Cov}} 
\newcommand{\IMSE}{\textnormal{IMSE}}
\newcommand{\MSE}{\textnormal{MSE}}

\newcommand{\vertiii}[1]{{\left\vert\kern-0.25ex\left\vert\kern-0.25ex\left\vert #1 
    \right\vert\kern-0.25ex\right\vert\kern-0.25ex\right\vert}}

\newcommand{\Ac}{\mathcal{A}}
\newcommand{\Bc}{\mathcal{B}}

\renewcommand{\Mc}{\mathcal{M}}
\newcommand{\Nc}{\mathcal{N}}

\newcommand{\Hc}{\mathcal{H}}

\newcommand{\Lc}{\mathcal{L}}
\newcommand{\Sc}{\mathcal{S}}

\newcommand{\Bb}{\mathbb{B}}

\newcommand{\Gb}{\mathbb{G}}

\newcommand{\Ub}{\mathbb{U}}

\newcommand{\diff}{{\,\mathrm{d}}}

\newcommand{\scs}{\scriptscriptstyle}

\newcommand{\convd}{\rightsquigarrow}              
\newcommand{\convw}{\convd}                           
\newcommand{\weak}{\convd}              

\newtheorem{theorem}{Theorem}[section]
\newtheorem{proposition}[theorem]{Proposition}
\newtheorem{lemma}[theorem]{Lemma}
\newtheorem{corollary}[theorem]{Corollary}

\theoremstyle{definition}

\newtheorem{condition}[theorem]{Condition}

\newtheorem{alg}[theorem]{Algorithm}

\theoremstyle{remark}
\newtheorem{remark}[theorem]{Remark}

\allowdisplaybreaks[4]

\begin{document}
\title[Detecting non-stationarities in functional Time Series]{Detecting deviations from second-order stationarity in locally stationary functional~time~series}

\author{Axel B\"ucher}
\address{Heinrich-Heine-Universit\"at D\"usseldorf, Mathematisches Institut, Universit\"atsstr.~1, 40225 D\"usseldorf, Germany.}
\email{axel.buecher@hhu.de}

\author{Holger Dette and Florian Heinrichs}
\address{Ruhr-Universit\"at Bochum, Fakult\"at f\"ur Mathematik, Universit\"atsstr.\ 150, 44780 Bochum, Germany.}
\email{holger.dette@rub.de}
\email{florian.heinrichs@rub.de}


\maketitle

\begin{abstract}
A time-domain test for the assumption of second order stationarity of a functional time series is proposed. The test is based on combining individual cumulative sum tests which are designed to be sensitive to changes in the mean, variance and autocovariance operators, respectively.  The combination of their dependent $p$-values relies on a joint dependent block multiplier bootstrap of the individual test statistics. Conditions under which the proposed combined testing procedure is asymptotically valid under stationarity are provided. A procedure is proposed to automatically choose the block length parameter needed for the construction of the bootstrap. The finite-sample behavior of the proposed test is investigated in Monte Carlo experiments and an illustration on a real data set is provided. 

\medskip


\end{abstract}

\date{\today}

\maketitle


\section{Introduction} \label{sec:intro}

Within the last decades,  statistical analysis for  functional time series has become a very active area of research
(see the monographs \citealp{Bos00},  \citealp{ferrvieu2006}, \citealp{HorKok12} and \citealp{hsingeubank2015}, among others).
Many authors impose the assumption of stationarity, which allows for developing
advanced statistical theory. For instance, \cite{Bosq2002} and \cite{Dehling2005}
investigate stationary functional processes with a linear representation
and  \cite{Hormann2010} provide a general framework to model  functional observations
from stationary processes.  Frequency domain analysis of stationary functional time series has been considered by \cite{PanTav13}, while \cite{vanDelft2016} propose a new concept of local stationarity  for  functional data. The assumption of second order stationarity is also of particular importance for prediction problems (see \citealp{as03,adh,hyndmann2009} among others) and for dynamic principal component analysis (\citealp{horkidhal2015}).

Ideally, the assumption of stationarity should be checked before applying any statistical methodology.
Several authors have considered this problem, in particular within the context of change point analysis where the null hypothesis of stationarity is tested against the alternative  of  a structural change in certain parameters of the process; see \cite{Aue2009,Berkes2009,Horvath2010,Aston2012} among others.
Tests that are designed to be powerful against more general alternatives are often based on an analysis in the frequency domain.
For example,  \cite{AueVan2017} generalize  the approach of \cite{dwisub2011}  and \cite{jensub2015} to functional time series. More precisely, they begin by showing  that the functional discrete Fourier transform (fDFT) is asymptotically uncorrelated at distinct Fourier frequencies if and only if the process is weakly stationary. The corresponding test is then based on a quadratic form based on a finite dimensional projection  of the empirical covariance operator of the fDFT's. Consequently, the properties of the test depend on the number of lagged fDFT's included. 
As an alternative, \cite{VanBagChaDet17} construct a test using  an estimate of a minimal  distance between the spectral density operator of a non-stationary process  and its best  approximation by a spectral density operator corresponding to a stationary  process
(see also \citealp{DettePreussVetter2011}  for a discussion of this approach in the univariate context). 
The test statistic consists of  sums of Hilbert-Schmidt inner products of periodogram operators (evaluated at different frequencies), and is  asymptotically normal  distributed.

In the present paper, we propose an alternative time domain  test for second order stationarity of a functional time series. More precisely, we suggest to address the problem of detecting non-stationarity by individually checking the hypothesis that the mean and the autocovariance operators at a given lag, say $h$,  of a collection (indexed by time) of approximating stationary functional time series are in fact time independent. As explained in the next paragraph, the individual tests are then combined to yield a joint test including autocovariances up to a given maximal lag $H$.
Thus, the approach investigated here is  similar in spirit to  the classical Portmanteau tests for serial correlation of a univariate time series, where the hypothesis  of white  noise is checked by investigating whether  correlations  up to a given lag vanish (see \citealp{Box1970,Ljung1978}). For the problem of checking stationarity in real-valued time series,  similar approaches have been taken by  \cite{Jin2015} and \cite{BucFerKoj18}. 

To combine the individual tests for stationarity of the mean and the autocovariance operators at a given lag $h$, we use appropriate extensions of well-known $p$-value combination methods dating back to \cite{Fis32}. 
 Each individual test is relying on a block multiplier approach making necessary the choice of a joint block length parameter $m$. Following ideas put forward in \cite{PolWhi2004}, a procedure is proposed to automatically select that parameter data-adaptively in such a way that, asymptotically, a certain MSE-criterion  is optimized.

The remaining part of this article is organized as follows: In Section~\ref{sec:math}, we collect necessary mathematical preliminaries. In Section~\ref{sec:dev1}, 
we first propose individual tests for the hypothesis of second-order stationarity which are particularly sensitive to deviations in the mean, variance and a given lag $h$ autocovariance, respectively. The tests are then combined to a joint test for second order stationarity which is sensitive to deviations in the mean, variance and the first $H$ autocovariances.  
In Section~\ref{sec:examples} we discuss an exemplary locally stationary time series model in great theoretical detail, while finite-sample results and a case study are presented in Section~\ref{sec:finite}. The central proofs are collected in Section~\ref{sec:proofs}, while less central proofs and auxiliary results are provided in a supplementary material.

	\section{Mathematical Preliminaries} \label{sec:math}

	\subsection{Random elements in $L^p$-spaces}
	
For some separable measurable space $(S, \Sc,\nu)$ with a $\sigma$-finite measure $\nu$ and $p>1$, let $\Lc^p(S, \nu)$ denote the set of measurable functions $f:S \to \R$ such that $\|f\|_p = (\int |f|^p \diff \nu)^{1/p}<\infty$. For $f\in \Lc^p(S, \nu)$, let $[f]$ be the set of all functions $g$ such that $f=g$, $\nu$-almost surely. The space $L^p(S, \nu)$ of all equivalence classes $[f]$ then becomes a separable Banach space, and standard weak convergence theory  is applicable. If $S$ is a subset of $\R^d$ and $\nu$ is the Lebesgue measure, we occasionally write $\Lc^p(S)$ and $L^p(S)$. 
		
Let $(\Omega, \Ac, \Prob)$ denote a probability space and let $X: S\times \Omega \to \R$ be $(\Sc \otimes \Ac)$-measurable such that $X(\cdot, \omega) \in \Lc^p(S, \mu)$ for $\Prob$-almost every $\omega$. It follows from Lemma 6.1 in \cite{JanKai15} that $\omega \mapsto [X(\cdot, \omega)]$ is a random variable in $L^p(S, \mu)$ (equipped with the Borel-$\sigma$-field). Conversely, note that for any random variable $[Y]$ in $L^p(S, \mu)$, we can choose a $(\mu \otimes \Prob)$-a.s.\ unique  $(\Sc \otimes \Ac)$-measurable mapping $Y': S \times \Omega \to \R$ such that $Y'(\cdot, \omega) \in [Y](\omega)$ for $\Prob$-almost every $\omega$. We can hence (a.s.) identify random variables in $L^p(S, \mu)$ with measurable functions on $S\times \Omega$ which are $p$-integrable in the first argument ($\Prob$-a.s.); slightly abusing notation we also write $X$ for the equivalence class $[X]$. 
		
%
	
A random variable $X$  in $L^2([0,1]^d)$ is called integrable  if $\Exp \|X\|_2 < \infty$. In that case, it follows from the Riesz representation theorem that there exists a unique element $\mu_X =\Exp X \in L^2([0,1]^d)$ such that $\Exp\langle X,f\rangle = \langle \mu_X,f\rangle$ for all $f\in L^2([0,1]^d)$, where $\langle f,g\rangle=\int_{\scs [0,1]^d} fg \diff \lambda_d$.  If $X$ is even square integrable, that is, $\Exp \|X\|_2^2 < \infty$, the covariance operator of $X$ is defined as the operator $C_X: L^2([0,1]^d) \to L^2([0,1]^d)$ given by $C_X(f)=\Exp[\langle f,X-\mu_X \rangle (X-\mu_X)]$. $C_X$ is nuclear and hence a Hilbert-Schmidt operator (\citealp{Bos00}, Section~1.5), whence, by Theorem 6.11 in \cite{Wei80}, there exists a kernel $c_X \in L^2([0,1]^d \times [0,1]^d)$ such that 
\[
C_X(f)(\tau) = \int_{[0,1]^d} c_X(\tau,\sigma) f(\sigma) \diff \sigma
\]
for almost every $\tau \in[0,1]^d$ and every $f\in L^2([0,1]^d)$.  Similarly, for square integrable random elements $X,Y\in L^2([0,1]^d)$ we define the cross-covariance operator $C_{X,Y}: L^2([0,1]^d) \to L^2([0,1]^d)$ by  $C_{X,Y}(f) = \Exp[\langle X- \mu_X, f\rangle (Y-\mu_Y)]$. By the same reasoning as above, there exists a kernel $c_{X,Y}\in L^2([0,1]^d \times [0,1]^d)$ such that 
\[
C_{X,Y}(f)(\tau) = \int_{[0,1]^d} c_{X,Y}(\tau,\sigma) f(\sigma) \diff \sigma
\]
If $X$ is in fact a $(\Bc([0,1]^d) \otimes \Ac)$-measurable function from $[0,1]^d \times \Omega$ to $\R$ with $X(\cdot, \omega) \in \Lc^2([0,1]^d)$ a.s., then it can be shown that, in the respective $L^2$-spaces,
\[
\mu(\tau)=\Exp[X(\tau)], \quad c_{X}(\tau,\sigma)= \cov\{ X(\tau), X(\sigma) \},  \quad c_{X,Y}(\tau,\sigma)= \cov\{ X(\tau), Y(\sigma) \}.
\]
By the preceding paragraph, this notation also makes sense for equivalence classes $X,Y\in L^{2}([0,1]^d)$.

	\subsection{Functional time series in $L^2([0,1])$}\label{subsec:fts}
For each $t\in\Z$, let $X_{t}: [0,1] \times \Omega \to \R$ denote a $(\Bc|_{[0,1]}  \otimes \Ac)$-measurable function with $X_{t}(\cdot, \omega) \in \Lc^2([0,1])$. By the preceding section, we can regard $[X_{t}]$ as a random variable in $L^2([0,1])$, which we also write as $X_{t}$.  The sequence
	$
	(X_{t})_{t\in\Z} 
	$
	will be referred to as a \textit{functional time series} in $L^2([0,1])$. 
	The functional time series will be called \textit{stationary} if, for all $q\in\Z$ and all $h,t_1,\dots,t_q\in\Z$ 
	\[ 
	(X_{t_1+h},\dots,X_{t_q+h})\overset{d}{=} (X_{t_1},\dots,X_{t_q})
	\]
	in $L^2([0,1])^{q}$.
	
	
	 Let $\rho>0$. A sequence of functional time series $(X_{t,T})_{t\in\Z}$, indexed by $T\in\N$,  is called \textit{locally stationary (of order $\rho$)} if, for any $u\in[0,1]$, there exists a strictly stationary functional time series $\{X_t^{\scs (u)} \mid t\in\Z\}$ in $L^2([0,1])$ and an array of real-valued random variables $\{P_{\scs t,T}^{\scs (u)} \mid t=1,\dots ,T\}_{T\in\naturals}$ with $\ex |P_{t,T}^{\scs (u)}|^\rho<\infty$, uniformly in $1\leq t\leq T, T\in\naturals$ and $u\in[0,1]$, such that 
	\begin{align} \label{eq:ls}
	\|X_{t,T}-X_t^{(u)}\|_2 \leq \bigg(\bigg|\frac{t}{T}-u\bigg|+\frac{1}{T}\bigg)P_{t,T}^{(u)} 
	\end{align} 
	for all $t=1,\dots ,T,T\in\naturals$ and $u\in[0,1]$. 
	This concept of local stationarity was first introduced by \cite{Vogt2012} for $p$-dimensional time series ($p\in \mathbb{N}$).
	By the arguments in the preceding section, we may assume that $X_t^{\scs (u)}$ is in fact a $(\Bc([0,1])\times \Ac)$-measurable function from $[0,1]\times \Omega$ to $\R$ such that $X_t^{\scs (u)}(\cdot , \omega) \in \Lc^2([0,1])$ for $\Prob$-almost every $\omega$.
	In the subsequent sections, we will usually assume that $\rho\geq 2$ and that $\Exp[ \|X_t^{(u)}\|_2^2]< \infty$ for all $u\in[0,1]$. Despite the fact that $\{(X_{t,T})_{t\in\Z}: T\in\N\}$ is a sequence of time series, we will occasionally simply call $(X_{t,T})_{t\in\Z}$ a locally stationary time series.
	
%

\color{black}{}
	\subsection{Further Notation} 

In the following, we will deal with different norms on the spaces $L^p([0,1]^d)$, for $p\geq 1,d\in\naturals$. To avoid confusion, we denote the corresponding norms by $\|\cdot\|_{p,d}$. As a special case, we will write $\|\cdot\|_p$ instead of $\|\cdot\|_{p,1}$. Further, we introduce the notation $\|\cdot\|_{p,\Omega\times[0,1]^d}$ for the $p$-norm on the space $L^p(\Omega\times[0,1]^d,\Prob \otimes \lambda^d)$. 
Finally, we define $(f\otimes g)(x,y)=f(x)g(y)$ for functions $f,g\in L^p([0,1])$. 

	\section{Detecting deviations from second-order stationarity}  \label{sec:dev1}
	
	\subsection{Second-order stationarity in locally stationary time series} \label{subsec:so}
	
Before we can propose suitable test statistics for detecting deviations from second-order stationarity in a locally stationary functional time series, we need to clarify what is meant by second-order stationarity. Loosely speaking, we want to test the null hypothesis that the mean and/or the (auto)covariances do not vary too much over time.   Meaningful asymptotic results will be obtained by formulating these null hypotheses in terms of the approximating sequences $\{X_{t}^{\scs (u)}: t \in \Z\}$ defined in Section~\ref{subsec:fts}. More precisely, we will subsequently assume that $\Exp[ \|X_t^{(u)}\|_2^2]< \infty$ for all $u\in[0,1]$ and consider the hypotheses
\begin{align} \label{eq:h0m}
	H_0^{(m)}: \|\ex[X_0^{(u)}]-\ex[X_0^{(v)}]\|_2=0 \quad \text{ for all } u,v \in [0,1]
\end{align}
and, for some lag $h\ge 0$,
\begin{align} \label{eq:h0c}
	H_0^{(c,h)}:\|\ex[X_{0}^{(u)} \otimes X_h^{(u)}]-\ex[X_{0}^{(v)} \otimes X_h^{(v)}]\|_{2,2}=0 
	\quad \text{ for all } u,v \in [0,1].
\end{align}
Note that the intersection 
\[
	H_0=H_0^{(m)} \cap H_0^{(c,0)} \cap H_0^{(c,1)} \cap \dots
\] 
corresponds to the case where the approximating sequences $\{X_{t}^{\scs (u)}: t \in \Z\}$, indexed by $u\in[0,1]$, all share the same second-order characteristics. We will therefore call the sequence of time series $(X_{t,T})_{t\in\Z}$, indexed by $T\in\N$, second-order stationary if the global hypothesis $H_0$ is met. The test statistics we are going to propose will be particularly sensitive to deviations from (weak) stationarity in the mean, the variance, and the first $H$ autocovariances, which leads us to define
\begin{equation}\label{eq:h0H}
	H_0^{(H)} = H_0^{(m)} \cap H_0^{(c,0)} \cap H_0^{(c,1)} \cap \dots \cap H_0^{(c,H)},
\end{equation}
where $H \in \N_0$ is fixed and denotes the maximum number of lags under consideration.
	
	\begin{remark}
	The hypotheses $H_0^{\scs (m)}$ and $H_0^{\scs (c,h)}$ are independent of the choice of the approximating family $\{X_t^{\scs (u)}: t\in\Z \}_{u\in[0,1]}$. Indeed, suppose there were two approximating families $\{X_t^{\scs (u)}: t\in\Z\}_{u\in[0,1]}$ and $\{Y_t^{\scs (u)}: t\in\Z\}_{u\in[0,1]}$ satisfying \eqref{eq:ls}. By stationarity and the triangle inequality, we have, for any $t,T\in\N$,
	\begin{align*}
	\ex\|X_t^{(u)}-Y_t^{(u)}\|_2
	&=  
	\ex\|X_{\lfloor uT\rfloor}^{(u)}-Y_{\lfloor uT\rfloor}^{(u)}\|_2 \\
	&\leq 
	\ex\|X_{\lfloor uT\rfloor}^{(u)}-X_{\lfloor uT\rfloor,T}\|_2 +\ex\|X_{\lfloor uT\rfloor,T}-Y_{\lfloor uT\rfloor}^{(u)}\|_2
	\leq \frac{C}{T}.
	\end{align*}
	This implies $\ex\|X_t^{\scs (u)}-Y_t^{\scs (u)}\|_2=0$ and hence $\|X_t^{\scs (u)}-Y_t^{\scs (u)}\|_2=0$ almost surely.  \qed
\end{remark}


The following lemma provides two interesting equivalent formulations of each of the above hypotheses. Introduce the notations $M:[0,1]^2\to\R, M_h:[0,1]^3\to\R$, where
\begin{align}
M(u,\tau)&=\int_0^u \ex[X_0^{(w)}(\tau)]\diff w-u\int_0^1 \ex[X_0^{(w)}(\tau)]\diff w,
\label{eq:mut} \\
M_h(u,\tau_1,\tau_2)&=\int_0^u \ex[X_0^{(w)}(\tau_1)X_h^{(w)}(\tau_2)]\diff w-u\int_0^1 \ex[X_0^{(w)}(\tau_1)X_h^{(w)}(\tau_2)]\diff w. 
\label{eq:mhutt}
\end{align}

\begin{lemma}\label{lem:equivh}
	Let $\{(X_{t,T})_{t\in\Z}: T\in\N\}$ denote a locally stationary functional time series of order $\rho\ge 4$ with approximating sequences $(X_{t}^{\scs (u)})_{t \in \Z}$ satisfying $\ex[\| X_{0}^{\scs (u)} \|_2^4]<\infty$ for all $u\in[0,1]$. 
	Then, the hypothesis $H_0^{\scs (m)}$ in \eqref{eq:h0m} is met if and only if
	\begin{equation}\label{eq:h0m1}
	\| M \|_{2,2} =0.
	\end{equation}
	Likewise, for any $h\in \N_0$, $H_0^{\scs (c,h)}$ in \eqref{eq:h0c} is met if and only if 
	\begin{equation}\label{eq:h0c1}
	\| M_h \|_{2,3} =0.
	\end{equation}
	Moreover, the hypothesis $H_0^{\scs (m)}$ is equivalent to
	\begin{align} \label{eq:h0m2}
	& \exists \ C>0: \quad \| \ex[X_{\lfloor uT\rfloor,T}]-\ex[X_{0,T}]\|_{2}\leq \frac{C}{T}  \quad \text{for all } u\in[0,1], T \in \N,
	\end{align}
	and $H_0^{\scs (c,h)}$ is equivalent to
	\begin{multline} \label{eq:h0c2}
	\exists\ C>0: \quad \| \ex[X_{\lfloor uT\rfloor,T}\otimes X_{\lfloor uT\rfloor+h,T} - X_{0,T}\otimes X_{h,T}]\|_{2,2} 
	\leq \frac{C}{T}   \\
	\text{for all } u\in[0,1], T \in \N.
	\end{multline}
\end{lemma}

The lemma is proven in Section~\ref{subsec:proofs3}. We will heavily rely on the conditions \eqref{eq:h0m1} and \eqref{eq:h0c1} when constructing the test statistics in the next section. Assertions \eqref{eq:h0m2} and \eqref{eq:h0c2} are interesting in their own rights, as they provide a sub-asymptotic formulation of the hypothesis of second-order stationarity. They are used in the next section for showing that the tests are consistent, and will also be crucial when extending consistency results to the case of piecewise locally stationary processes in Section~\ref{subsec:amoc}.

	\subsection{Test statistics} \label{subsec:test} 
	In the subsequent sections, we assume to observe, for some $T\in\N$, an excerpt  $X_{1,T}, \dots, X_{T,T}$ from a locally stationary time series $\{(X_{t,T})_{t\in\Z}:T\in\N\}$.  We are interested in testing the hypotheses $H_0^{\scs (m)}$ and $H_0^{\scs (c,h)}$ formulated in the preceding section, which can be done  individually by a CUSUM-type procedure. More precisely, for $u,\tau \in[0,1]$, let
\begin{align} \label{eq:cusumm}
U_T(u,\tau)&= \frac{1}{\sqrt{T}}\bigg( \sum_{t=1}^{\lfloor uT\rfloor}X_{t,T}(\tau)-u\sum_{t=1}^{T}X_{t,T}(\tau) \bigg)
\end{align}
denote the \textit{CUSUM-process for the mean}, and, for  $u,\tau_1, \tau_2\in[0,1]$ and $h \in \N_0$, let
\begin{align} \label{eq:cusumc}
{U}_{T,h}(u,\tau_1,\tau_2)&= \frac{1}{\sqrt{T}}\bigg( \sum_{t=1}^{\lfloor uT\rfloor\wedge (T-h)}X_{t,T}(\tau_1)X_{t+h,T}(\tau_2)-u\sum_{t=1}^{T-h}X_{t,T}(\tau_1)X_{t+h,T}(\tau_2) \bigg)
\end{align} 	
denote the \textit{CUSUM-process for the (auto)cross-moments at lag $h$}. Under the null hypothesis $H_0^{\scs (m)}$, $T^{-1/2} {U}_T(u,\tau)$ can be regarded as an estimator of 
the quantity $M(u,\tau)$ defined in \eqref{eq:mut}, and a similar statement holds for $T^{-1/2} {U}_{T,h}$, which estimates the $M_h$ in \eqref{eq:mhutt}.
Hence, by Lemma~\ref{lem:equivh}, it seems reasonable to reject $H_0^{\scs (m)}$  or $H_0^{\scs (c,h)}$ for large values of 
\begin{align} \label{eq:testst}
\Sc_T^{(m)} = \| {U}_{T} \|_{2,2} \quad  \text{ or } \quad \Sc_T^{(c,h)} = \| {U}_{T,h} \|_{2,3},
\end{align}
respectively.\footnote{Alternatively, one could use the $L^2$-norm in $\tau$ and $(\tau_1,\tau_2)$, respectively, and the supremum in $u$, as proposed in \cite{sharipov2016}. However, preliminary simulation results suggested that a test based on the $L^2$-norm in $u$ performs better in applications with small sample sizes.} 


In Section~\ref{subsec:boot} below we will propose a procedure that allows to combine the previous tests statistics to obtain a joint test for the combined hypothesis $H_0^{\scs (H)}$, with maximal lag $H\in \N_0$ fixed.
For that purpose, we will first need (asymptotic) critical values for the individual test statistics $\Sc_T^{(m)}$ and $\Sc_T^{(c,h)}$, which in turn can be deduced from the joint asymptotic distribution of the CUSUM processes in \eqref{eq:cusumm} and \eqref{eq:cusumc}. The basic tools are the following partial sum processes
\begin{align*}
\tilde{B}_{T}(u,\tau) &=\frac{1}{\sqrt{T}}\sum_{t=1}^{\lfloor uT\rfloor} X_{t,T}(\tau)-\ex[X_{t,T}(\tau)], \\
\tilde{B}_{T,h}(u,\tau_1,\tau_2) &=\frac{1}{\sqrt{T}}\sum_{t=1}^{\lfloor uT\rfloor \wedge (T-h)} X_{t,T}(\tau_1)X_{t+h,T}(\tau_2)-\ex[X_{t,T}(\tau_1)X_{t+h,T}(\tau_2)],
\end{align*}
where $u,\tau,\tau_1, \tau_2 \in [0,1]$ and $h\in \N_0$. The expected values within the sums will  be denoted by 
\[
\mu_{t,T}(\tau) = \Exp[X_{t,T}(\tau)] 
\quad \text{and} \quad 
\mu_{t,T,h}(\tau_1, \tau_2) = \ex[X_{t,T}(\tau_1)X_{t+h,T}(\tau_2)].
\]
The following assumptions are sufficient to guarantee weak convergence of these processes.
\begin{condition}[Assumptions on the functional time series] \label{cond:fts}~
	\begin{enumerate}
	\renewcommand{\theenumi}{(A1)}
	\renewcommand{\labelenumi}{\theenumi}
\item\label{cond:ls} \textbf{Local Stationarity.} The observations $X_{1,T}, \dots X_{T,T}$ are an excerpt from a  locally stationary functional time series $\{(X_{t,T})_{t\in\Z}:T\in\N\}$ of order $\rho=4$ in $L^2([0,1],\reals)$. 

	\renewcommand{\theenumi}{(A2)}
		\renewcommand{\labelenumi}{\theenumi}
		\item\label{cond:mom}  \textbf{Moment Condition.}  
		For any $k\in\N$, there exists a constant $C_k<\infty$ such that $\ex \|X_{t,T}\|_2^k\leq C_k$ and $\ex\|X_0^{\scs (u)}\|_2^k\leq C_k$ uniformly in $t\in\integers,T\in\naturals$ and $u\in[0,1]$.
%

		\renewcommand{\theenumi}{(A3)}
		\renewcommand{\labelenumi}{\theenumi}
		
		\item\label{cond:cum}  \textbf{Cumulant Condition.} For any $j\in\N$ there is a constant $C_j <\infty$ such that
		\begin{equation} \label{sumcum}
		\sum_{t_1,\dots ,t_{j-1}=-\infty}^{\infty} \big\| \cum(X_{t_1,T},\dots ,X_{t_j,T})\big\|_{2,j} \leq C_j<\infty, 
		\end{equation}
		for any $t_j\in\integers$ (for $j=1$ the condition is to be interpreted as $\|\ex X_{t_1,T}\|_2\leq C_1$ for all $t_1 \in \integers$). Further, for $k\in\{2,3,4\}$, there exist functions $\eta_k:\integers^{k-1}\to\reals$ satisfying
		\[
		\sum_{t_1,\dots,t_{k-1}=-\infty}^{\infty} (1+|t_1|+\dots+|t_{k-1}|)\eta_k(t_1,\dots,t_{k-1})  < \infty
		\]
		such that, for any $T\in\N, 1 \le t_1 , \dots , t_k \le T, v, u_1, \dots , u_k \in[0,1],h_1,h_2\in\Z$, $Z_{t,T}^{\scs (u)}\in\{X_{\scs t, T},X_{t}^{\scs (u)}\}$,  and any $Y_{t,h,T}(\tau_1,\tau_2)\in\{ X_{t,T}(\tau_1), X_{t,T}(\tau_1)X_{t+h,T}(\tau_2) \}$, we have
				\begin{enumerate}
			\item[(i)] {\small$\|\cum(X_{t_1,T}-X_{t_1}^{(t_1/T)},Z_{t_2,T}^{(u_2)},\cdots,Z_{t_k,T}^{(u_k)})\|_{2,k} \leq \frac{1}{T} \eta_k(t_2-t_1,\dots,t_k-t_1)$},
			\item[(ii)] {\small$\|\cum(X_{t_1}^{(u_1)}-X_{t_1}^{(v)},Z_{t_2,T}^{(u_2)},\cdots,Z_{t_k,T}^{(u_k)})\|_{2,k} \leq |u_1-v| \eta_k(t_2-t_1,\dots,t_k-t_1)$}, 
			\item[(iii)] {\small$\|\cum(X_{t_1,T},\dots,X_{t_k,T})\|_{2,k} \leq \eta_k(t_2-t_1,\cdots,t_k-t_1)$}, 	
			\item[(iv)] {\small$\int_{[0,1]^{2}} |\cum\big(Y_{t_1,h_1,T}(\tau),Y_{t_2,h_2,T}(\tau) \big)|\diff\tau
				\leq \eta_2(t_2-t_1)$.}
		\end{enumerate}	

	\end{enumerate}
\end{condition}

	Assumption~\ref{cond:mom} is needed to ensure existence of all cumulants. The cumulant condition \ref{cond:cum} is  a (partially)  weakened version 
	of the assumptions made by \cite{LeeSubbaRao2016} and \cite{AueVan2017} and has its origins in classical multivariate time series analysis, see \cite{Brillinger1965}, Assumption 2.6.2. 
Lemma \ref{mixing} below shows that the cumulant conditions in \ref{cond:cum} hold, provided \ref{cond:ls}, \ref{cond:mom}, a further moment condition and a strong mixing condition are satisfied. In particular, they are met for the models employed in Section~\ref{sec:finite} within our simulation study, see in particular Lemma~\ref{lem:ar1}.

The following theorem, proven in Section~\ref{subsec:proofs3}, shows that $\tilde{B}_{T}$ and $\tilde{B}_{T,h}$ jointly converge weakly with respect to the $L^2$-metric.  For $H \in \N_0$, let the cartesian product 
\[\Hc_{H+2}= L^2([0,1]^2)\times\{ L^2([0,1]^3)\}^{H+1}\] 
be equipped with the sum of the individual scalar products, such that $\Hc_{H+2}$ is a Hilbert space itself.

\begin{theorem}\label{theo:main}
	Suppose that Assumptions~\ref{cond:ls}--\ref{cond:cum} are met. Then, the vector $\mathbb{B}_T=(\tilde{B}_T,\tilde{B}_{T,0},\dots ,\tilde{B}_{T,H})$ converges weakly to a centred Gaussian variable $\Bb=(\tilde B,\tilde {B}_0,\dots ,\tilde {B}_H)$ in $\Hc_{H+2}$  with covariance operator $C_\Bb:\Hc_{H+2} \to \Hc_{H+2}$ defined as
	\begin{multline*}
	C_{\Bb} 
	\left(\begin{array}{c}
	g \\ f_0 \\ \vdots \\  f_{H}
	\end{array} \right)
	\left(\begin{array}{c}
	(u,\tau) \\ (u_0, \tau_{01}, \tau_{02})  \\ \vdots \\  (u_H, \tau_{H1}, \tau_{H2})
	\end{array} \right) \\
	=
	\left(\begin{array}{c}
	\langle r^{(m)}((u,\tau), \cdot), g\rangle + \sum_{h=0}^H \langle r^{(m,c)}_h ((u,\tau ), \cdot), f_h \rangle \\ 
	\langle r^{(m,c)}_0(\cdot, (u_{0}, \tau_{01}, \tau_{02})), g\rangle + \sum_{h=0}^H \langle r^{(c)}_{0,h} ((u_0,\tau_{01}, \tau_{02}) , \cdot), f_h \rangle  \\
	\vdots \\  
	\langle r^{(m,c)}_H(\cdot, (u_{H}, \tau_{H1}, \tau_{H2})), g\rangle + \sum_{h=0}^H \langle r^{(c)}_{H,h} ((u_H,\tau_{H1}, \tau_{H2}) , \cdot), f_h \rangle 
	\end{array} \right).
	\end{multline*}
Here, the kernel functions $r^{(m)}, r^{(c)}_{h,h'}$ and $r^{(m,c)}_h$ are given by
\begin{align*}
r^{(m)}((u,\tau), (v,\phi)) 
&= 
\cov\big(\tilde B(u,\tau),\tilde B(v,\phi)\big)
=
\sum_{k=-\infty}^{\infty}\int_{0}^{u\wedge v}c_{k,1}(w) \diff w,\\
r_{h,h'}^{(c)}((u,\tau_1, \tau_2), (v,\phi_1, \phi_2)) 
&= 
\cov\big(\tilde {B}_h(u,\tau_1,\tau_2),\tilde {B}_{h'}(v,\phi_1,\phi_2)\big)
=
\sum_{k=-\infty}^{\infty}\int_{0}^{u\wedge v}c_{k,2}(w) \diff w,\\
r_{h}^{(m,c)}((u,\tau_1, \tau_2), (v,\phi_1, \phi_2)) 
&=
\cov\big(\tilde B(u,\tau),\tilde {B}_{h}(v,\phi_1,\phi_2)\big)
=
\sum_{k=-\infty}^{\infty}\int_{0}^{u\wedge v}c_{k,3}(w) \diff w,
\end{align*}
with
\begin{align*}
c_{k,1}(w)
&= 
c_{k,1}(w, \tau, \varphi)
=
\cov\big(X_0^{(w)}(\tau),X_k^{(w)}(\phi)\big),\\
c_{k,2}(w)
&=
c_{k,2}(w, h, h',\tau_1, \tau_2, \varphi_1, \varphi_2)
=
\cov\big(X_0^{(w)}(\tau_1)X_h^{(w)}(\tau_2),X_k^{(w)}(\phi_1)X_{k+h'}^{(w)}(\phi_2)\big),\\
c_{k,3}(w)
&=
c_{k,3}(w, h, \tau, \varphi_1, \varphi_2)
=
\cov\big(X_0^{(w)}(\tau),X_k^{(w)}(\phi_1)X_{k+h}^{(w)}(\phi_2)\big),
\end{align*}
for any $0\leq h,h'\leq H$. In particular the infinite sums and integrals converge.
\end{theorem}

The following corollary on joint weak convergence of the CUSUM processes defined in \eqref{eq:cusumm} and \eqref{eq:cusumc} is essentially a mere consequence of the continuous mapping theorem. Let 
\begin{align*}
\tilde {G}_T(u,\tau) &= \tilde {B}_T(u,\tau) - u \tilde {B}_T(1,\tau) \\
\tilde {G}_{T,h}(u,\tau_1, \tau_2) &= \tilde {B}_{T,h}(u,\tau_1, \tau_2) - u \tilde {B}_{T,h}(1,\tau_1, \tau_2) \\
\Gb_T &= (\tilde G_T, \tilde G_{T,1}, \dots, \tilde G_{T,H})
\end{align*}
and, similarly, 
\begin{align}
\label{eq:tg}
\tilde {G}(u,\tau) &= \tilde B(u,\tau) - u \tilde B(1,\tau) \\
\nonumber
\tilde {G}_{h}(u,\tau_1, \tau_2) &= \tilde {B}_{h}(u,\tau_1, \tau_2) - u \tilde {B}_{h}(1,\tau_1, \tau_2) \\
\nonumber
\Gb &= (\tilde G, \tilde G_{1}, \dots, \tilde G_{H}).
\end{align}

\begin{corollary}\label{cor:main1}
Suppose that Assumptions~\ref{cond:ls}--\ref{cond:cum} are satisfied.  If $H_0^{\scs (m)}$ holds, then
\[
\| {U}_T -  \tilde {G}_T\|_{2,2} = o_\Prob(1). 
\]
If $H_0^{\scs (c,h)}$ holds, then
\[
\| {U}_{T,h} -  \tilde {G}_{T,h}\|_{2,3} = o_\Prob(1).
\]
As a consequence, if the hypothesis $H_0^{\scs (H)}$ in \eqref{eq:h0H} holds, then, 
\[
\Ub_T = (U_T, U_{T,1}, \dots, U_{T,H}) = \Gb_T + o_\Prob(1) \weak \Gb.
\]
On the other hand, if $H_0^{\scs (m)}$ or $H_0^{\scs (c,h)}$ does not hold, then $\Sc_T^{(m)}\to \infty$ or $\Sc_T^{(c,h)}  \to \infty$ in probability, respectively. 
\end{corollary}


The corollary suggests to reject $H_0^{\scs (m)}$ or $H_0^{\scs (c,h)}$ for large values of $\Sc_T^{(m)}$ or $\Sc_T^{(c,h)}$, respectively. However,
the corresponding null-limiting distributions $\|\tilde{G} \|_{2,2}$ and $\|\tilde{G}_h \|_{2,3}$ depend in a complicated way on the functions $c_{k,j}$ defined in Theorem~\ref{theo:main}, and cannot be easily transformed into a pivotal distribution. 
We therefore propose to derive critical values by a suitable block multiplier bootstrap approximation worked out in detail in Section \ref{subsec:boot}.

\subsection{Strong mixing and Cumulants} \label{subsec:gen}

In this section we will demonstrate that under the assumption of a \textit{strong mixing} locally stationary functional time series, Assumption \ref{cond:cum} is met.
To be precise,  let  $\mathcal{F}$ and $\mathcal{G}$  be  $\sigma$-fields in $(\Omega, \Ac)$ and define
\[ 
\alpha(\mathcal{F},\mathcal{G})=\sup\{ |\pr(A\cap B)-\pr(A)\pr(B)| :  A\in\mathcal{F}, B\in\mathcal{G}\}.
\]
A functional time series $\{(X_{t,T})_{t\in\Z}: T\in\N\}$ in $L^2([0,1])$ is called $\alpha$- or \textit{strongly mixing} if the mixing coefficients
\begin{align*} 
\alpha'(k)=\sup\limits_{T\in\naturals}\sup\limits_{t\in\integers}\alpha\Big(\sigma\big(\{X_{s,T}(\tau)|\tau\in[0,1]\}_{s=-\infty}^t\big),\sigma\big(\{X_{s,T}(\tau)|\tau\in[0,1]\}_{s=t+k}^\infty\big)\Big) 
\end{align*}
vanish as $k$ tends to infinity. Analogously, we define
\begin{align*} 
\alpha''(k)=\sup_{u\in[0,1]}\sup\limits_{t\in\integers}\alpha\Big(\sigma\big(\{X_{s}^{(u)}(\tau)|u,\tau\in[0,1]\}_{s=-\infty}^t\big),\sigma\big(\{X_{s}^{(u)}(\tau)|u,\tau\in[0,1]\}_{s=t+k}^\infty\big)\Big) 
\end{align*}
as mixing coefficients of the family of approximating stationary processes.	Further, we define $\alpha(k)=\max\{\alpha'(k),\alpha''(k)\}$. A locally stationary, functional time series is called \textit{strongly mixing}, if $\alpha(k)$ vanishes, as $k$ tends to infinity and  {\it exponentially strongly  mixing} if  $\alpha (k) \leq c a^k $ for some constants $c >0 $   and $a \in (0,1)$.
Note that we can define the mixing coefficients in terms of a function in $\mathcal{L}^2([0,1])$ rather than an element of the space $L^2([0,1])$ of equivalence classes by Lemma 6.1 in \cite{JanKai15}.  The main result of this section provides 
sufficient conditions for the theory developed so far for strong mixing processes.

\begin{lemma}\label{mixing}
Let $\{(X_{t,T})_{t\in\Z}: T\in\N\}$ be a strongly mixing locally stationary functional time series in $L^2([0,1],\reals)$ such that 
Assumptions~\ref{cond:ls}, \ref{cond:mom} and the condition
\[\sup_{t,T} \|X_{t,T}\|_{r, \Omega \times [0,1]}<C_r < \infty
\]
 are satisfied for any integer $r> 2$. If  $\{(X_{t,T})_{t\in\Z}: T\in\N\}$  is exponentially strongly mixing, then 
it also  satisfies the summability conditions for the cumulants in 
Assumption \ref{cond:cum}.
\end{lemma}

	\subsection{Bootstrap approximation} \label{subsec:boot}
	The bootstrap approximation will be based on two smoothing parameters: a block length sequence $m=m_T$ needed to asymptotically catch the serial dependence within the time series, and a bandwidth sequence $n=n_T$ needed to estimate expected values locally in time. We will impose the following condition.
	
	\begin{condition}[Assumptions on the bootstrap scheme]~
	\begin{enumerate}
	\renewcommand{\theenumi}{(B1)}
	\renewcommand{\labelenumi}{\theenumi}
	\item\label{eq:b1} 
	Let $m=m(T) \le T$ be an integer-valued sequence, to be understood as the block length within a block bootstrap procedure. Assume that $m$ tends to infinity and $m/T$ vanishes, as $T\to\infty$.  
	\renewcommand{\theenumi}{(B2)}
	\renewcommand{\labelenumi}{\theenumi}
	\item\label{eq:b2} 
	Let $n=n(T) \le T/2$ be an integer-valued sequence such that both $m/n$ and $mn^2/T^2$ converge to zero, as $T$ tends to infinity. 
	\renewcommand{\theenumi}{(B3)}
	\renewcommand{\labelenumi}{\theenumi}
	\item\label{eq:b3}  Let $\{R_i^{\scs (k)}\}_{i,k\in\naturals}$ denote independent standard normally distributed random variables, independent of the stochastic process $\{(X_{t,T})_{t\in\Z}:T\in\N\}$ . 
	\end{enumerate}
	\end{condition}
	
Under this set of notations, we define
\[ 
\hat{B}_{T}^{(k)}(u,\tau)
=
\frac{1}{\sqrt{T}} \sum_{i=1}^{\lfloor uT\rfloor}\frac{R_i^{(k)}}{\sqrt{m}}\sum_{t=i}^{(i+m-1)\wedge T} \big\{ X_{t,T}(\tau)-\hat \mu_{t,T}(\tau) \big\}
\]
as a bootstrap approximation for $\tilde {B}_T(u, \tau)$, where
\begin{align*}  
\hat \mu_{t,T}(\tau) 
=
\frac{1}{\tilde{n}_{t,0}}\sum_{j=\ubar{n}_t}^{\bar{n}_{t,0}}X_{t+j,T}(\tau) 
\end{align*}
denotes an estimator for $\mu_{t,T}(\tau)$ relying on the bandwidth sequence $n$ via
\begin{equation}\label{eq:def:n}
\bar{n}_{t,h}=n\wedge (T-t-h), \quad \ubar{n}_t=-n\vee (1-t), \quad \tilde{n}_{t,h} = \bar{n}_{t,h}-\ubar{n}_t+1, 
\end{equation}
for $0\leq h\leq H$. Similarly, for any $0\leq h\leq H$, bootstrap approximations for $\tilde {B}_{T,h}(u, \tau_1, \tau_2)$ are defined as
\begin{align*}
\hat{B}_{T,h}^{(k)}(u,\tau_1,\tau_2)
=
\frac{1}{\sqrt{T}} \sum_{i=1}^{\lfloor uT\rfloor \wedge (T-h)}\frac{R_i^{(k)}}{\sqrt{m}}
\sum_{t=i}^{(i+m-1)\wedge (T-h)} \big\{ X_{t,T}(\tau_1)X_{t+h,T}(\tau_2) - \hat \mu_{t,T,h}(\tau_1,\tau_2) \big\},
\end{align*}
where
\begin{align*} 
\hat \mu_{t,T,h}(\tau_1,\tau_2) =   \frac{1}{\tilde{n}_{t,h}}\sum_{j=\ubar{n}_t}^{\bar{n}_{t,h}}X_{t+j,T}(\tau_1)X_{t+j+h,T}(\tau_2).
\end{align*}
Finally, for fixed $k\in\N$, collect the bootstrap approximations in the vector 
\[
\hat \Bb_T^{(k)} = (\hat{B}_{T}^{(k)},\hat{B}_{T,0}^{(k)},\dots ,\hat{B}_{T,H}^{(k)}).
\]
The following theorem shows that the bootstrap replicates can be regarded as asymptotically independent copies of the original process $\Bb_T$ from Theorem~\ref{theo:main}.

\begin{theorem}\label{theo:boot}
	Suppose that Assumptions~\ref{cond:ls}--\ref{cond:cum} and \ref{eq:b1}--\ref{eq:b3} are met. Then, for any fixed $K\in \N$ and as $T \to \infty$,
	\[
	\big(\mathbb{B}_T,\hat{\mathbb{B}}_{T}^{(1)},\dots ,\hat{\mathbb{B}}_{T}^{(K)}\big)
	\weak 
	\big(\mathbb{B},\mathbb{B}^{(1)},\dots ,\mathbb{B}^{(K)}\big)
	\]
	in $\{ L^2([0,1]^2)\times(L^2([0,1]^3))^{H+1}\}^{K+1}$, where $\mathbb{B}^{\scs (k)}$ ($k=1,\dots ,K$) are independent copies of the centred Gaussian variable $\mathbb{B}$  from Theorem~\ref{theo:main}. Equivalently (\citealp{BucKoj17}, Lemma 2.2), 
	\[
	d_{\mathrm{BL}} ( \pr^{\hat{\mathbb{B}}_{T}^{(1)} \mid X_{1,T}, \dots, X_{T,T}}, \pr^{\mathbb{B}_T} ) =o_\Prob(1), \quad T \to \infty,
	\]
where $d_{\mathrm{BL}}$ denotes the bounded Lipschitz metric between probability distributions on $L^2([0,1]^2)\times(L^2([0,1]^3))^{H+1}$. 
\end{theorem}

The proof is given in Section~\ref{subsec:proofs3}. The preceding theorem, together with Corollary~\ref{cor:main1}, suggests to define the following bootstrap approximation for the CUSUM processes defined in \eqref{eq:cusumm} and \eqref{eq:cusumc}:
\begin{align*}
\hat{G}_T^{(k)}(u,\tau)&=\hat{B}_{T}^{(k)}(u,\tau) -u\hat{B}_{T}^{(k)}(1,\tau),\\
\hat{G}_{T,h}^{(k)}(u,\tau_1,\tau_2)&=\hat{B}_{T,h}^{(k)}(u,\tau_1,\tau_2) -u\hat{B}_{T,h}^{(k)}(1,\tau_1,\tau_2), \\
\hat{\mathbb{G}}_T^{(k)}&=(\hat{G}_T^{(k)},\hat{G}_{T,0}^{(k)},\dots ,\hat{G}_{T,H}^{(k)}).
\end{align*}
Theorem \ref{theo:boot}, Corollary~\ref{cor:main1} and the continuous mapping theorem then imply that, under the hypothesis $H_0^{\scs (H)}$ in \eqref{eq:h0H}, 
\begin{align*}
(\bm S_T, \bm S_T^{(1)}, \dots, \bm S_T^{(K)}) 
&\equiv 
(\Phi(\mathbb{U}_T), \Phi(\hat{\mathbb{G}}_T^{\scs (1)}),\dots ,\Phi(\hat{\mathbb{G}}_T^{\scs(K)}))  \\
&=
(\Phi(\mathbb{G}_T), \Phi(\hat{\mathbb{G}}_T^{\scs (1)}),\dots ,\Phi(\hat{\mathbb{G}}_T^{\scs(K)}))+o_\Prob(1)   \\
&\weak 
(\Phi(\mathbb{G}), \Phi(\mathbb{G}^{\scs (1)}),\dots ,\Phi(\mathbb{G}^{\scs (K)}))
\equiv (\bm S, \bm S^{(1)}, \dots, \bm S^{(K)}) ,
\end{align*}
where $\Phi(G_{-1},G_0, \dots, G_{H}) = (\|G_{-1} \|_{2,2}, \|G_{0}\|_{2,3}, \dots, \| G_H\|_{2,3})$
 and where $\mathbb{G}^{\scs (1)},\dots ,\mathbb{G}^{\scs (K)}$ are independent copies of $\mathbb{G}$. Individual bootstrap-based tests for, e.g., $H_{0}^{\scs (c,h)}$ are then naturally defined by the $p$-value
 \[
 p_{T,K}(S_{T,h}) = \frac1K \sum_{j=1}^K \bm 1( S_{T,h}^{(k)} \ge S_{T,h}),
 \]
where $S_{T,h}^{\scs (k)}$ and $S_{T,h}$ denote the $(h+2)$nd coordinate of $\bm S_{T}^{\scs (k)}$ and  $\bm S_T$, respectively; in particular, $S_{T,-1}=\Sc_T^{\scs (m)}$ and $S_{T,h}=\Sc_T^{\scs (c,h)}$ as defined in \eqref{eq:testst}. Indeed, we can show the following result for each individual test.

\begin{proposition} \label{prop:test}
Suppose that Assumptions~\ref{cond:ls}--\ref{cond:cum} and \ref{eq:b1}--\ref{eq:b3} are met. Then, for all $h \in \Z_{\ge -1}$, provided $K=K_T \to \infty$, and with $H_0^{(c,-1)}=H_0^{(m)}$, we have
\[
 p_{T,K_T}(S_{T,h}) \weak 
 \begin{cases} \mathrm{Uniform}(0,1) & \text{if }H_0^{\scs (c,h)} \text{ is met} \\
 0  & \text{else}.
\end{cases}
\]
\end{proposition}

Moreover,  we can rely on an extension of Fisher's $p$-value combination method \citep{Fis32} as  described in Section 2 in \cite{BucFerKoj18}
 to obtain a combined test for the joint hypothesis $H_0^{\scs (H)}$ in \eqref{eq:h0H}. More precisely, let $\psi:(0,1)^{H+2} \to \R$ be a continuous function that is decreasing in each argument (throughout the simulations, we employ $\psi(p_{-1},\dots,p_H)=\sum_{i=-1}^{H}w_i\Phi^{-1}(1-p_i)$ with weights $w_{-1}=w_{0}=1/3$ and $w_1=\cdots=w_H= (3H)^{-1}$.)
  The combined test is defined by its $p$-value calculated based on the following  algorithm.

\begin{alg}[Combined Bootstrap test for $H_0^{\scs (H)}$] \label{alg:joint}~
\begin{compactenum}
\item Let $\bm S_T^{\scs (0)} = \bm S_{T}$.

\item Given a large integer $K$, compute the sample of $K$ bootstrap replicates $\bm S_T^{\scs (1)},\dots,\bm S_T^{\scs (K)}$ of the vector $\bm S_T^{\scs (0)}$.

\item Then, for all $i \in \{0,1,\dots,K\}$ and $h \in \{-1,\dots,H\}$, compute
\[
p_{T,K}(S_{T,h}^{(i)}) = \frac{1}{K+1} \bigg\{\frac{1}{2} + \sum_{k=1}^K \bm 1 \left( S_{T,h}^{(k)} \geq S_{T,h}^{(i)} \right) \bigg\}.
\]

\item Next, for all $i \in \{0,1,\dots,K\}$, compute
\begin{equation*}
W_{T,K}^{(i)} = \psi \{ p_{T,K}(S_{T,0}^{(i)}),\dots,p_{T,K}(S_{T,H}^{(i)}) \}.
\end{equation*}
\item The global statistic is $W_{T,K}^{\scs (0)}$ and the corresponding $p$-value is given by
\begin{equation*}
p_{T,K}(W_{T,K}^{(0)}) =\frac{1}{K} \sum_{k=1}^K \bm 1 \left( W_{T,K}^{(k)} \geq W_{T,K}^{(0)} \right).
\end{equation*}
\end{compactenum}
\end{alg}

Consistency of this procedure is a mere consequence of Proposition 2.1 in \cite{BucFerKoj18}; details are omitted for the sake of brevity.

	\subsection{Consistency against AMOC-piecewise locally stationary alternatives}  \label{subsec:amoc}
	In the previous section, the proposed tests were shown to be consistent against locally stationary alternatives. In classical change point settings, the underlying CUSUM-principle is also known to be consistent against piecewise (locally) stationary alternatives, notably against those that involve a single change in the signal of interest (AMOC = at most one change). We are going to derive such results within the present setting.  
		
	For the sake of brevity, we only consider AMOC-alternatives in the mean. More precisely, we assume that $\{(X_{t,T})_{t \in \Z}:T\in\N\}$ follows the data generating process
\begin{equation}\label{defCPprocess}
	X_{t,T}=\left\{ 
	\begin{array}{ll}
	\mu_1 + Y_{t,T}~\text{for}~t\le \lambda T\rfloor\\
	\mu_2 + Y_{t,T}~\text{for}~t\ge\lfloor \lambda T\rfloor+1.
	\end{array}
	\right.
\end{equation}
for some $\lambda \in (0,1), \mu_1, \mu_2 \in \Lc^2([0,1])$ and $\{(Y_{t,T})_{t \in \Z}:T\in\N\}$ a locally stationary time series satisfying Condition~\ref{cond:fts}.
In the literature on classic change point detection, one would be  interested in testing for the null hypothesis that $\|\mu_1 - \mu_2 \|_2 =0$, against the alternative that this $L^2$-norm is positive.

Now, if $\|\mu_1 - \mu_2 \|_2 =0$,  we are back in the situation of the preceding sections. However, one can show (by contradiction) that if $\|\mu_1 - \mu_2 \|_2 > 0$, $\{(X_{t,T})_{t \in \Z}:T\in\N\}$ is not locally stationary, whence additional theory must be developed to show consistency of the test statistic $\Sc_{T}^{\scs (H)}$. Note that even the formulation of $H_0^{\scs (H)}$ relying on \eqref{eq:h0m} and \eqref{eq:h0c} is not possible anymore, so that we need to rely on their equivalent sub-asymptotic counterparts \eqref{eq:h0m2} and \eqref{eq:h0c2} in Lemma~\ref{lem:equivh}. 

\begin{proposition} \label{prop:amoc}
	Let $\{(X_{t,T})_{t \in \Z}:T\in\N\}$ be a sequence of functional time series as defined in \eqref{defCPprocess}, with $\mu_1 \neq \mu_2$ in $L^2([0,1])$ and with $\{(Y_{t,T})_{t \in \Z}:T\in\N\}$ satisfying  Conditions~\ref{cond:ls}--\ref{cond:cum}.
	Then, the test statistic $\Sc_T^{\scs (m)}=S_{T,-1}$ based on observations $X_{1,T}, \dots, X_{T,T}$ diverges to infinity, in probability. If, additionally,   \ref{eq:b1}--\ref{eq:b3} are met, then the bootstrap variables $\hat S_{T,-1}^{\scs (k)}$ are stochastically bounded. As a consequence, the proposed test is consistent.
\end{proposition}

\subsection{Data-driven choice of the block length parameter $m$}\label{subsec:choiceM} 

The bootstrap procedure depends on the choice of the width of the local mean estimator, $n$, and the length of the bootstrap blocks, $m$. Preliminary simulation studies suggested that the performance of the procedure crucially depends on the choice of $m$, while it is less sensitive to the choice of $n$ (which may also be chosen by other standard criteria in specific applications, like adaptations of Silverman's rule of thumb, cross-validation or visual investigation of respective plots). In this section we propose a data-driven procedure for choosing the block length $m$ based on a certain optimality criterion. 

Recall that the limiting null-distributions of the proposed test statistics depend in a complicated way on the covariances $\cov\{\tilde B(u,\tau),\tilde B(v,\phi)\}, 
\cov\big\{ \tilde {B}_h(u,\tau_1,\tau_2),\tilde {B}_{h'}(v,\phi_1,\phi_2)\}$ and 
$\cov\{\tilde B(u,\tau),\tilde {B}_{h}(v,\phi_1,\phi_2)\}$. Following Section 5 in \cite{BucKoj2016}, the procedure  we propose essentially chooses $m$ in such a way that the bootstrap approximation for $\sigma_c(\tau, \phi)=\cov\{\tilde B(1,\tau),\tilde B(1,\phi)\}$ is optimal, with respect to $m$, in a certain asymptotic sense. More precisely, we propose to first minimize the integrated mean squared of the `bootstrap-estimator' 
\begin{align*}
\tilde{\sigma}_T(\tau,\phi) 
&= 
\cov\big(\tilde{B}_T^{(1)}(1,\tau),\tilde{B}_T^{(1)}(1,\phi)|X_{1,T},\cdots, X_{T,T}\big)
\end{align*}
considered as an estimator for $\sigma_c(\tau, \phi)$,
with respect to $m$ theoretically (see Lemma~\ref{lem:IMSE} below), and then use a simple plug-in approach to obtain a formula that solely depends on observable quantities.
Observe that $\tilde{\sigma}_T(\tau,\phi)$  can be rewritten as
\begin{align*}
\tilde{\sigma}_T(\tau,\phi)
&=\ex[\tilde{B}_T^{(1)}(1,\tau),\tilde{B}_T^{(1)}(1,\phi)|X_{1,T},\cdots, X_{T,T}]\\
&=
\frac{1}{T} \sum_{i=1}^{T} \frac{1}{m} \bigg(\sum_{t=i}^{(i+m-1)\wedge T}X_{t,T}(\tau)-\mu_{t,T}(\tau)\bigg) \bigg(\sum_{t=i}^{(i+m-1)\wedge T}X_{t,T}(\phi)-\mu_{t,T}(\phi)\bigg)
\end{align*}
whence $\tilde{\sigma}_T(\tau,\phi)$ is not a proper estimator as it depends on the unknown expectation $\mu_{t,T}$. The asymptotic integrated bias and integrated variance satisfy the following expansions. For simplicity, we replace Condition~\ref{cond:cum} by a strong mixing condition as in Section~\ref{subsec:gen}.

\begin{lemma}\label{lem:IMSE}
	Let $m=m(T)$ be an integer-valued sequence, such that $m$ tends to infinity and $m^2/T$ vanishes, as $T$ tends to infinity. If
	conditions   \ref{cond:ls} and \ref{cond:mom} are met and $\{(X_{t,T})_{t\in\Z}: T\in\N\}$  is exponentially strongly mixing, 	then, as $T\to\infty$,
\begin{align*} 
	\int_{[0,1]^2} \big(\ex[\tilde{\sigma}_T(\tau,\phi)]-\sigma_c(\tau,\phi)\big)^2 \diff(\tau,\phi) 
	&= 
	\frac{1}{m^2} \Delta + o(m^{-2}), \\
	\int_{[0,1]^2}\var\big(\tilde\sigma_T(\tau,\phi)\big) \diff(\tau,\phi)
	&= 
	\frac{m}{T} \Gamma + o(m/T).
\end{align*}
where
\[ 
	\Delta = \bigg\| \sum_{k=-\infty}^{\infty} |k| \int_0^1 \cov(X_0^{(w)},X_k^{(w)}) \diff w \bigg\|_{2,2}^2 
\]
and
\[
\Gamma = \frac{2}{3}\int_0^1 \bigg(\sum_{k=-\infty}^{\infty}  \int_0^1  \cov\big(X_0^{(w)}(\tau),X_{k}^{(w)}(\tau)\big) \diff \tau\bigg)^2
+ \bigg\| \sum_{k=-\infty}^{\infty} \cov(X_0^{(w)},X_k^{(w)}) \bigg\|_{2,2}^2  \diff w .
\]
\end{lemma}

As a consequence of this lemma, we obtain the expansion
\begin{align*}\label{eq:IMSE}
\IMSE_T(m)
&= 
\int_{[0,1]^2} \MSE(\tilde{\sigma}_T(\tau,\phi)) \diff(\tau,\phi)\\
&= 
\int_{[0,1]^2} \var\big(\tilde{\sigma}_T(\tau,\phi)\big) + \big(\ex[\tilde{\sigma}_T(\tau,\phi)]-\sigma_c(\tau,\phi)\big)^2 \diff(\tau,\phi) \\
&=
\frac{m}{T} \Gamma + \frac{1}{m^2} \Delta + o(m^{-2}) + o(m/T),
\end{align*} 
which can next be minimized with respect to $m$ to get a natural choice for the block length. More precisely, the dominating function $\Lambda(m)=\tfrac{m}{T} \Gamma + \tfrac{1}{m^2} \Delta$ is differentiable in $m$ with $\Lambda'(m)=\tfrac{\Gamma}{T} - \tfrac{2\Delta}{m^3}$ and $\Lambda''(m)=\tfrac{6\Delta}{m^4}$, whence  $m=\big(\tfrac{2\Delta T}{\Gamma}\big)^{1/3}$ is the unique minimiser of $\Lambda$. In practice,  both $\Gamma$ and $\Delta$ are unknown and must be estimated in terms of the observed data. This leads us to define 
\[
\hat m = \big({2\hat \Delta_T T}/{\hat \Gamma_T}\big)^{1/3}
\]
where, for some constant  $L\in \N$  specified below, 
\[
\hat{\Delta}_T = \int_{[0,1]^2}\bigg( \frac{1}{T-2L} \sum_{i=L+1}^{T-L} \sum_{k=-L}^{L}|k| \hat{\gamma}_{i,k,T}(\tau,\phi) \bigg)^2 \diff(\tau,\phi)
\]
and
\[
\hat{\Gamma}_T = \frac{2}{3} \frac{1}{T-2L} \sum_{i=L+1}^{T-L} \bigg(\sum_{k=-L}^{L} \int_0^1 \hat{\gamma}_{i,k,T}(\tau,\tau) \diff\tau\bigg)^2 + \int_{[0,1]^2} \bigg( \sum_{k=-L}^{L} \hat{\gamma}_{i,k,T}(\tau,\phi) \bigg)^2 \diff(\tau,\phi).
\]
Here $\hat{\gamma}_{i,k,T}$ is defined by
\begin{multline*}
\hat{\gamma}_{i,k,T}(\tau,\phi)= \frac{1}{\bar{n}_{i+k,0}-\ubar{n}_{i}+1 } \sum_{j=\ubar{n}_{i}}^{\bar{n}_{i+k,0}} \bigg(X_{i+j,T}(\tau)-\frac{1}{\tilde{n}_{i+j,0}}\sum_{t=\ubar{n}_{i+j}}^{\bar{n}_{i+j,0}} X_{i+j+t,T}(\tau)\bigg)\\ \times\bigg(X_{i+j+k,T}(\phi)-\frac{1}{\tilde{n}_{i+j+k,0}}\sum_{t=\ubar{n}_{i+j+k}}^{\bar{n}_{i+j+k,0}} X_{i+j+k+t,T}(\phi)\bigg)
\end{multline*}
and $\bar{n}_{t,h},\ubar{n}_t$ and $\tilde{n}_{t,h}$ are given in \eqref{eq:def:n}. Note that the above estimators depend on the choice of the integer $L$. Following \cite{BucKoj2016} and \cite{PolWhi2004}, we select $L$ to be the smallest integer, such that 
\[ 
\hat{\rho}_{k,T} 
=
\frac{\|\frac{1}{T-k}\sum_{i=1}^{T-k}\hat \gamma_{i,k,T}\|_{2,2}}{\|\frac{1}{T} \sum_{i=1}^{T} \hat \gamma_{i,0,T}\|_{2,2}} 
\]
is negligible for any $k>L$; more precisely, $L$ is chosen as the smallest integer such that $\hat{\rho}_{L+k,T} \leq 2 \sqrt{\log(T)/T}$, for any $k=1,\cdots, K_T$, with $K_T= \max\{5, \sqrt{\log T}\}.$

	\section{Time-varying random operator functional AR processes}  \label{sec:examples}

	We consider an exemplary  class of functional locally stationary processes and specify the approximating family of stationary processes. The results in this section are similar to Theorem 3.1 of \cite{Bos00}.  
	
	Let $\mathcal{L}=\mathcal{L}\big(L^2([0,1]),L^2([0,1])\big)$ be the space of bounded linear operators on $L^2([0,1])$. Further, denote by $\|\cdot\|_\mathcal{L}$ and $\|\cdot\|_\mathcal{S}$ the standard operator norm and the Hilbert-Schmidt norm respectively, i.\,e.,
	\[ 
	\|\ell \|_\mathcal{L} = \sup\limits_{\|x\|_2\leq 1} \|\ell(x)\|_2,
	\qquad 
	\|\ell \|_\mathcal{S} = \bigg(\sum_{j=1}^{\infty} \lambda_j^2\bigg)^{1/2} 
	\]
	for $\ell \in \Lc$ with eigenvalues $\lambda_1\geq \lambda_2\geq \dots$.
	By Equation (1.55) in \cite{Bos00}, we have $\|\cdot\|_\mathcal{L} \leq \|\cdot\|_\mathcal{S}$. For any $T\in\N$, consider the recursive functional equation
	\begin{align} \label{eq:ar}
	X_{t,T} = Y_{t,T} + \mu(t/T), \qquad Y_{t,T} = A_{t/T}(Y_{t-1,T})+ \eps_{t,T}, \qquad t \in \Z,
	\end{align}
	where $(\eps_{t,T})_{t\in \Z}$ is a sequence of independent zero mean innovations in $L^2([0,1])$ and where $A_{t,T}:L^2([0,1]) \to L^2([0,1])$ denotes a possibly random and time-varying bounded linear operator. 
	The equation defines what might be called a (time varying) random operator functional autoregressive process of order one, denoted by $\rm tvrFAR(1)$, see also \cite{VanBagChaDet17}, Section~4.1, for the non-random case with $\eps_{t,T}$ not depending on~$T$.
	
In the following, we will only consider the case where $\mu$ is the null function. In the more general case of $\mu$ being Lipschitz, if there exists a locally stationary solution $Y_{t,T}$ of the equation on the right-hand side of \eqref{eq:ar} with approximating family $\{Y_t^{\scs (u)}|t\in\integers\}_{u\in[0,1]}$, then $X_{t,T} = Y_{t,T} + \mu(t/T)$ is obviously locally stationary with approximating family $X_t^{\scs (u)}=Y_t+\mu(u)$.

To be precise, we restrict ourselves to the following specific parametrization 
	\begin{align*} 
	\mu\equiv 0, \qquad A_{t/T} = a(t/T) \tilde A, \qquad \eps_{t,T} = \sigma(t/T)  \tilde \eps_t,
	\end{align*}
	where $a$ and $\sigma>0$ are measurable functions on $[0,1]$.
	The following lemma provides sufficient conditions for ensuring local stationarity of the model and provides an explicit expression for the approximating family of stationary processes. 
	For a related result in the case where $A_{t/T}$ is non-random and $\eps_{t,T}$ does not depend on $T$ see Theorem 3.1 in \cite{vanDelft2016}.

	For a sequence of operators $(B_i)_i$ in $\Lc$, we will write $\prod_{i=0}^{n} B_i = B_0 \circ \dots \circ B_n$ for $n\in\N$. The empty product will be identified with the identity on $L^2([0,1])$, that is, $\prod_{i=0}^{-1} B_i = \ident_{L^2([0,1]}$ 
	
	\begin{lemma}\label{lem:ar1}
		
		Let $(\tilde{\eps}_t)_{t\in\Z}$ be strong white noise in $L^2([0,1])$.
		Further, let $a$ and $\sigma$ be measurable functions on $(-\infty,1]$ such that $\sigma >0$, $a(u) = a(0)$ and $\sigma(u)=\sigma(0)$ for all $u\le 0$.
		Finally, let $\eps_{t,T}=\sigma(t/T)\tilde \eps_t$, $\eps_t^{(u)} = \sigma(u) \tilde \eps_t$ and $A_u=a(u) \tilde A$, where $\tilde A$ denotes a random operator in $\Lc$ that is independent from $(\tilde \eps_t)_{t\in\Z}$ and satisfies
		$\sup_{u \in [0,1]} \|A_u\|_\mathcal{S}\leq q<1$ with probability one. 
		Then:
		\begin{asparaenum}[(i)]
			\item
			For any $u\in[0,1]$, there exists a unique stationary solution $(Y_t^{\scs (u)})_{t \in \Z}$ of the recursive equation
			\begin{align*} 
			Y_t^{(u)}=A_u(Y_{t-1}^{(u)})+\eps_{t}^{(u)}, \qquad t \in \Z,
			\end{align*}
			namely 
			\[
			Y_t^{(u)}=\sum_{j=0}^{\infty} A_u^j(\eps_{u,t-j}),
			\]
			where the latter series converges in $L^2(\Omega \times [0,1], \Prob \otimes \lambda)$ and almost surely in $L^2([0,1])$.

			\item If $\sigma$ and $a$ are Lipschitz continuous, then there exists a unique locally stationary solution $(Y_{t,T})$ of order $\rho=2$ satisfying $\sup_{t\in\Z, T\in\N} \ex[\|Y_{t,T}\|_2^2]<\infty$
			of the recursive equation
			\begin{align*}
			Y_{t,T}=A_{t/T}(Y_{t-1,T})+\eps_{t,T}, \qquad t \in \Z, T\in\N,
			\end{align*}
			namely
			\[ 
			Y_{t,T}=\sum_{j=0}^{\infty} \prod_{i=0}^{j-1} A_{\tfrac{t-i}{T}}(\eps_{t-j,T}),
			\]
			the series again being convergent in $L^2(\Omega \times [0,1], \Prob \otimes \lambda)$ and almost surely in $L^2([0,1])$. The locally stationary process has approximating family $\{(Y_t^{(u)})_{t\in\Z}: u \in [0,1] \}$.
		\end{asparaenum}
		
	\end{lemma}

	\section{Finite-sample Results}  \label{sec:finite}
	
	\subsection{Monte Carlo Simulations} \label{sec:sims}
	A large scale Monte Carlo simulation study was performed to analyse the finite-sample behavior of the proposed tests. The major goals of the study were to analyse the level approximation and the power of the various tests, with a particular view on investigating various different forms of alternatives, notably models from $H_1^{\scs (m)}, H_1^{\scs (c,0)}$ and $H_1^{\scs (c,1)}$. All stated results related to testing the joint hypothesis  $H_0^{\scs(H)}$ are for the combined test described in Algorithm~\ref{alg:joint},  with $\psi(p_{-1},\dots,p_H)=\sum_{i=-1}^{H}w_i\Phi^{-1}(1-p_i)$ with weights $w_1=w_0=1/2$ for $H=0$ and $w_{-1}=w_{0}=1/3$ and $w_1=\cdots=w_H= (3H)^{-1}$ for $H\ge 1$.
	
	For the data-generating processes, we employed 10 different choices for the parameters in \eqref{eq:ar}, which will be described next.
	Let $(\psi_i)_{i\in\N_0}$ denote the Fourier basis of $L^2([0,1])$, that is, for $n\in\N$,
	\[
	\psi_0 \equiv 1, 
	\quad \psi_{2n-1}(\tau) = \sqrt 2 \sin(2\pi n\tau),
	\quad \psi_{2n}(\tau) = \sqrt 2 \cos(2\pi n\tau).
	\]
	Let $(\tilde \eps_t)_{t\in\Z}$ denote an i.i.d.\ sequence of mean zero random variables in $L^2([0,1])$, defined by $\tilde \eps_t = \sum_{i=0}^{16} u_{t,i} \psi_i$, where $u_{t,i}$ are independent and normally distributed with mean zero and variance $\var(u_{i,t}) = \exp(-i/10)$. Independent of $(\tilde \eps_t)_{t\in\Z}$, let $\bm G=(G_{i,j})_{i,j=0, \dots, 16}$ denote a matrix with independent normally distributed entries with $\var(G_{i,j})=\exp(-i-j)$. Let $\tilde A :L^2([0,1]) \to L^2([0,1])$ denote the (random) integral operator  defined by 
	\[
	\tilde A (f)(t)  = \tfrac1{3 \vertiii{G}_F} \sum_{i,j=0}^{16} G_{i,j} \langle f, \psi_i\rangle \psi_j(t)
	=
 \int_0^1 	\Big( \tfrac1{3 \vertiii{G}_F} \sum_{i,j=0}^{16} G_{i,j} \psi_i(s)  \psi_j(t) \Big) f(s) \diff s,
	\]
	where $\vertiii{\bm G}_F$ denotes the Frobenius norm (note that the Hilbert-Schmidt-norm of $\tilde A$ is equal to $1/3$, see \citealp{HorKok12}, Section 2.2).
	Finally, let
	\begin{align*}
	a_0(u) &= 1, &  	
	a_1(u) &= \tfrac12 + u, \\
	a_2(u) &=1-\tfrac{1}{2}\cos(2\pi u),   & 
	a_3(u) &= \tfrac{1}{2}+\id(u \geq 1/2), 
	\end{align*}
	for $u\in[0,1]$ and let $a_j(u)=a_j(0)$ for $u\le 0$ and $a_j(u)=a_j(1)$ for $u\ge 1$.
	The following ten data-generating processes are considered:
	\begin{compactitem}
	\item \textit{Stationary Case.} Let 
		\begin{align} 
		\mu \equiv 0, \qquad A_{t/T} = \tilde A, \qquad \eps_{t,T} = \tilde \eps_t.
		\label{eq:h0} 
	\end{align}
	\item \textit{Models deviating from $H_0^{\scs (m)}$.}   For $j=1,\dots, 3$, consider the choices
	\begin{align} 
		\mu(\tau) = a_j(\tau), \qquad A_{t/T} = \tilde A, \qquad \eps_{t,T} = \tilde \eps_t.
		\label{eq:m} 
	\end{align}
	\item \textit{Models deviating from $H_0^{\scs (c,0)}$.}  For $j=1,\dots, 3$, consider the choices
	\begin{align} 
		\mu\equiv 0, \qquad A_{t/T} = \tilde A, \qquad \eps_{t,T} = a_j(t/T) \tilde \eps_t.
		\label{eq:c0} 
	\end{align}
	\item \textit{Models deviating from $H_0^{\scs (c,1)}$.}   For $j=1,\dots, 3$, consider the choices
	\begin{align} 
		\mu\equiv 0, \qquad A_{t/T} = a_j(t/T) \tilde A, \qquad \eps_{t,T} =  \tilde \eps_t.
		\label{eq:c1} 
	\end{align}
	\end{compactitem}
	
	{Subsequently, the respective models will be denoted by $(\Mc_0)$ and  $(\Mc_{m,j}), (\Mc_{v,j})$ and $(\Mc_{a,j})$ for $j=1, \dots, 3$.}
	Note that the model descriptions are non-exclusive: for instance, the models in \eqref{eq:m} exhibiting deviations from $H_0^{\scs (m)}$ also deviate from $H_0^{\scs (c,0)}$.

Preliminary  simulation studies showed that the data-driven choice of $m$, as introduced in \ref{subsec:choiceM}, yields similar results as a manual choice of $m$, and should be favoured. Further parameters of the simulation design are as follows: the number of bootstrap replicates is set to $K=200$. 
Two sample sizes were considered, namely $T=256$ and $T=512$. Observe though that, unlike many frequency domain based methods for functional time series,  the proposed testing procedure does not require the sample sizes to be  a power of two to work effectively.  The hyperparameter $n$ for estimating local means is set to $n=45,60,75,90,T$. Finally, the maximum number of lags considered was set to $H=4$. Empirical rejection rates are based on $N=500$ simulation runs each and are summarized in Tables~\ref{tab:t256} and~\ref{tab:t512}.

{\scriptsize
	\begin{table}[!htbp] 
		\centering\begin{tabular}{cc|rrrrrrr|rr}
			Model&n & $H_0^{(m)}$ & $H_0^{(c,0)}$ & $H_0^{(0)}$ & $H_0^{(1)}$ & $H_0^{(2)}$ & $H_0^{(3)}$ & $H_0^{(4)}$ & $\bar{m}$ & $sd(m)$ \\ 
			\hline
$(\mathcal{M}_0)$
&45 & 7.2 & 0.4 & 3.0 & 1.0 & 0.8 & 0.8 & 1.0 & 5.77 & 0.41 \\ 
&60 & 7.0 & 0.2 & 2.6 & 0.4 & 0.4 & 0.4 & 0.4 & 5.78 & 0.41 \\ 
&75 & 5.4 & 0.8 & 2.2 & 0.2 & 0.2 & 0.4 & 0.4 & 5.80 & 0.40 \\ 
&90 & 5.2 & 0.2 & 1.8 & 0.2 & 0.2 & 0.2 & 0.2 & 5.80 & 0.41 \\ 
&256 & 4.0 & 0.2 & 1.8 & 0.0 & 0.0 & 0.0 & 0.0 & 5.81 & 0.38 \\ 
$(\mathcal{M}_{m,1})$
&45 & 92.8 & 64.4 & 89.6 & 89.8 & 90.8 & 90.6 & 91.2 & 5.82 & 0.41 \\ 
&60 & 91.8 & 61.2 & 88.8 & 87.8 & 88.6 & 89.4 & 89.0 & 5.81 & 0.41 \\ 
&75 & 90.2 & 59.2 & 87.8 & 86.4 & 87.2 & 87.6 & 87.6 & 5.82 & 0.41 \\ 
&90 & 89.8 & 58.2 & 86.4 & 85.2 & 86.4 & 86.2 & 86.8 & 5.83 & 0.40 \\ 
&256 & 88.2 & 50.4 & 81.8 & 79.4 & 81.8 & 82.8 & 83.2 & 5.82 & 0.42 \\ 
$(\mathcal{M}_{m,2})$
&45 & 57.6 & 31.2 & 55.8 & 53.8 & 55.0 & 55.0 & 56.0 & 5.79 & 0.41 \\ 
&60 & 53.0 & 26.2 & 49.2 & 48.2 & 48.6 & 48.4 & 49.4 & 5.78 & 0.42 \\ 
&75 & 49.0 & 20.8 & 45.0 & 42.8 & 43.4 & 42.8 & 43.0 & 5.78 & 0.42 \\ 
&90 & 41.8 & 16.0 & 38.8 & 33.8 & 35.8 & 35.8 & 36.4 & 5.78 & 0.42 \\ 
&256 & 31.0 & 9.0 & 24.8 & 21.4 & 23.4 & 22.4 & 22.0 & 5.89 & 0.60 \\ 
$(\mathcal{M}_{m,3})$
&45 & 99.8 & 97.0 & 99.8 & 99.8 & 99.8 & 99.8 & 100.0 & 5.78 & 0.43 \\ 
&60 & 99.6 & 96.6 & 99.8 & 99.8 & 99.8 & 99.8 & 99.8 & 5.78 & 0.44 \\ 
&75 & 99.6 & 96.2 & 99.8 & 99.8 & 99.8 & 99.8 & 99.8 & 5.78 & 0.44 \\ 
&90 & 99.8 & 95.8 & 99.8 & 99.8 & 99.8 & 99.8 & 99.8 & 5.80 & 0.44 \\ 
&256 & 99.6 & 94.8 & 99.4 & 99.0 & 99.2 & 99.2 & 99.2 & 6.15 & 1.42 \\ 
$(\mathcal{M}_{v,1})$
&45 & 8.0 & 100.0 & 99.6 & 84.8 & 81.2 & 79.4 & 79.0 & 5.50& 0.59 \\ 
&60 & 8.0 & 100.0 & 99.8 & 82.4 & 77.6 & 76.8 & 75.4 & 5.53 & 0.59 \\ 
&75 & 7.4 & 100.0 & 99.8 & 78.8 & 74.4 & 72.8 & 74.2 & 5.60 & 0.58 \\ 
&90 & 7.2 & 100.0 & 99.6 & 75.4 & 73.2 & 72.4 & 73.0 & 5.69 & 0.53 \\ 
&256 & 3.6 & 100.0 & 96.6 & 63.0 & 56.6 & 56.0 & 53.4 & 6.00 & 0.45 \\ 
$(\mathcal{M}_{v,2})$
&45 & 6.2 & 100.0 & 99.4 & 76.2 & 71.4 & 70.4 & 70.4 & 6.34 & 2.84 \\ 
&60 & 5.0 & 100.0 & 98.2 & 64.8 & 60.2 & 59.0 & 60.0 & 6.81 & 3.70 \\ 
&75 & 4.6 & 99.0 & 90.8 & 51.8 & 48.6 & 47.2 & 45.6 & 7.86 & 5.38 \\ 
&90 & 3.4 & 87.4 & 69.8 & 31.4 & 27.2 & 26.2 & 26.8 & 9.75 & 7.40 \\ 
&256 & 3.6 & 96.6 & 65.6 & 22.4 & 20.2 & 21.0 & 20.6 & 6.60 & 2.28 \\ 
$(\mathcal{M}_{v,3})$
&45 & 20.2 & 100.0 & 100.0 & 95.8 & 93.6 & 92.4 & 91.8 & 14.15 & 10.14 \\ 
&60 & 13.6 & 100.0 & 100.0 & 93.8 & 89.2 & 88.4 & 87.8 & 13.85 & 9.65 \\ 
&75 & 10.2 & 100.0 & 100.0 & 90.6 & 88.2 & 85.0 & 84.6 & 14.05 & 9.47 \\ 
&90 & 8.6 & 100.0 & 100.0 & 87.6 & 86.2 & 82.0 & 80.8 & 13.63 & 8.96 \\ 
&256 & 3.4 & 100.0 & 96.4 & 78.0 & 71.2 & 67.6 & 67.0 & 11.43 & 7.25 \\ 
$(\mathcal{M}_{a,1})$
&45 & 6.6 & 4.8 & 5.6 & 5.0 & 4.6 & 4.2 & 3.6 & 5.82 & 0.41 \\ 
&60 & 5.8 & 3.8 & 5.2 & 4.0 & 3.2 & 3.4 & 2.6 & 5.83 & 0.40 \\ 
&75 & 5.2 & 2.6 & 4.2 & 3.2 & 3.0 & 2.8 & 2.0 & 5.85 & 0.39 \\ 
&90 & 4.8 & 2.2 & 3.8 & 2.8 & 2.4 & 2.0 & 1.6 & 5.85 & 0.39 \\ 
&256 & 3.0 & 1.0 & 1.2 & 0.6 & 0.6 & 0.4 & 0.6 & 5.83 & 0.43 \\ 
$(\mathcal{M}_{a,2})$
&45 & 7.8 & 2.4 & 4.8 & 3.4 & 2.8 & 2.4 & 2.8 & 5.82 & 0.42 \\ 
&60 & 7.0 & 1.6 & 3.6 & 2.0 & 1.2 & 1.2 & 1.4 & 5.84 & 0.42 \\ 
&75 & 6.4 & 1.0 & 3.2 & 1.2 & 1.0 & 0.8 & 0.8 & 5.87 & 0.44 \\ 
&90 & 6.4 & 0.8 & 2.6 & 0.6 & 0.6 & 0.6 & 0.6 & 5.90 & 0.44 \\ 
&256 & 5.4 & 0.2 & 1.2 & 0.2 & 0.2 & 0.2 & 0.2 & 5.84 & 0.52 \\ 
$(\mathcal{M}_{a,3})$
&45 & 10.2 & 16.4 & 13.2 & 20.6 & 16.8 & 13.0 & 11.6 & 5.93 & 0.57 \\ 
&60 & 8.6 & 14.4 & 11.2 & 17.0 & 13.6 & 10.2 & 9.2 & 5.94 & 0.55 \\ 
&75 & 7.6 & 12.6 & 9.2 & 15.0 & 11.6 & 9.0 & 8.2 & 5.96 & 0.59 \\ 
&90 & 7.2 & 11.8 & 9.2 & 14.6 & 10.8 & 8.2 & 7.6 & 5.96 & 0.61 \\ 
&256 & 5.4 & 5.6 & 5.6 & 6.6 & 5.4 & 4.2 & 4.0 & 5.93 & 0.56
		\end{tabular}  \medskip 
		
		\caption{ \it
		Empirical rejection rates for various combined tests, based on a sample size of $T=256$. and a block length parameter $m$ calculated as proposed in Section~\ref{subsec:choiceM}. The last two columns provide the mean and standard deviation of the selected value of $m$.} 
		\label{tab:t256}
	\end{table}
}

{\scriptsize
	\begin{table}[!htbp] 
		\centering\begin{tabular}{cc|rrrrrrr|rr}
			Model&n & $H_0^{(m)}$ & $H_0^{(c,0)}$ & $H_0^{(0)}$ & $H_0^{(1)}$ & $H_0^{(2)}$ & $H_0^{(3)}$ & $H_0^{(4)}$ & $\bar{m}$ & $sd(m)$ \\ 
			\hline
$(\mathcal{M}_0)$
&45 & 10.6 & 5.2 & 7.8 & 4.4 & 4.0 & 3.8 & 3.4 & 7.24 & 0.43 \\ 
&60 & 10.0 & 3.8 & 6.6 & 3.0 & 2.2 & 2.4 & 2.6 & 7.25 & 0.44 \\ 
&75 & 8.4 & 2.0 & 3.8 & 1.4 & 1.0 & 1.6 & 1.4 & 7.28 & 0.46 \\ 
&90 & 7.6 & 1.8 & 3.4 & 1.2 & 1.2 & 1.0 & 1.0 & 7.28 & 0.46 \\ 
&512 & 4.6 & 1.0 & 2.6 & 0.8 & 0.8 & 0.8 & 0.8 & 7.30 & 0.46 \\ 
$(\mathcal{M}_{m,1})$
&45 & 100.0 & 97.4 & 100.0 & 99.8 & 99.8 & 100.0 & 100.0 & 7.23 & 0.43 \\ 
&60 & 99.8 & 97.0 & 100.0 & 99.6 & 99.6 & 99.8 & 100.0 & 7.23 & 0.42 \\ 
&75 & 100.0 & 96.8 & 99.8 & 99.6 & 99.8 & 99.8 & 99.8 & 7.25 & 0.44 \\ 
&90 & 100.0 & 96.6 & 99.8 & 99.4 & 99.6 & 99.8 & 99.8 & 7.27 & 0.45 \\ 
&512 & 99.0 & 91.2 & 98.8 & 98.4 & 98.4 & 98.6 & 98.6 & 7.35 & 0.54 \\ 
$(\mathcal{M}_{m,2})$
&45 & 95.6 & 82.4 & 95.6 & 95.2 & 95.6 & 96.4 & 96.0 & 7.23 & 0.42 \\ 
&60 & 94.2 & 78.8 & 94.6 & 94.2 & 94.6 & 94.6 & 95.0 & 7.25 & 0.43 \\ 
&75 & 94.0 & 75.2 & 94.0 & 93.2 & 93.4 & 94.2 & 94.0 & 7.26 & 0.45 \\ 
&90 & 93.2 & 73.4 & 93.4 & 93.2 & 93.0 & 93.6 & 93.6 & 7.27 & 0.44 \\ 
&512 & 83.4 & 51.8 & 79.2 & 79.6 & 81.6 & 82.0 & 83.4 & 7.40 & 0.55 \\ 
$(\mathcal{M}_{m,3})$
&45 & 100.0 & 100.0 & 100.0 & 100.0 & 100.0 & 100.0 & 100.0 & 7.24 & 0.43 \\ 
&60 & 100.0 & 100.0 & 100.0 & 100.0 & 100.0 & 100.0 & 100.0 & 7.27 & 0.45 \\ 
&75 & 100.0 & 100.0 & 100.0 & 100.0 & 100.0 & 100.0 & 100.0 & 7.28 & 0.46 \\ 
&90 & 100.0 & 100.0 & 100.0 & 100.0 & 100.0 & 100.0 & 100.0 & 7.29 & 0.46 \\ 
&512 & 100.0 & 99.6 & 100.0 & 100.0 & 100.0 & 100.0 & 100.0 & 8.47 & 3.45 \\ 
$(\mathcal{M}_{v,1})$
&45 & 9.2 & 100.0 & 100.0 & 97.8 & 95.6 & 95.4 & 95.0 & 6.85 & 0.44 \\ 
&60 & 8.4 & 100.0 & 100.0 & 95.8 & 93.2 & 91.0 & 89.2 & 6.88 & 0.46 \\ 
&75 & 7.2 & 100.0 & 100.0 & 94.8 & 92.6 & 89.6 & 87.6 & 6.89 & 0.44 \\ 
&90 & 6.0 & 100.0 & 100.0 & 94.4 & 90.6 & 89.6 & 85.8 & 6.91 & 0.43 \\ 
&512 & 4.0 & 100.0 & 100.0 & 90.4 & 84.0 & 80.8 & 79.2 & 7.43 & 0.54 \\ 
$(\mathcal{M}_{v,2})$
&45 & 7.2 & 100.0 & 100.0 & 96.6 & 94.0 & 93.6 & 92.4 & 6.79 & 0.60 \\ 
&60 & 6.6 & 100.0 & 100.0 & 94.2 & 90.4 & 88.2 & 87.6 & 6.86 & 0.57 \\ 
&75 & 6.2 & 100.0 & 100.0 & 93.2 & 88.6 & 85.4 & 82.4 & 6.96 & 0.63 \\ 
&90 & 5.6 & 100.0 & 100.0 & 90.8 & 86.8 & 84.2 & 81.0 & 7.05 & 0.66 \\ 
&512 & 3.8 & 100.0 & 99.8 & 87.2 & 82.4 & 80.0 & 78.2 & 7.50 & 0.75 \\ 
$(\mathcal{M}_{v,3})$
&45 & 8.2 & 100.0 & 100.0 & 99.4 & 97.8 & 96.0 & 96.4 & 7.89 & 3.41 \\ 
&60 & 7.2 & 100.0 & 100.0 & 98.6 & 96.4 & 95.0 & 93.4 & 7.79 & 3.08 \\ 
&75 & 6.4 & 100.0 & 99.8 & 98.6 & 94.8 & 92.0 & 91.2 & 7.87 & 3.12 \\ 
&90 & 6.0 & 100.0 & 100.0 & 98.4 & 95.4 & 92.0 & 90.4 & 7.98 & 3.14 \\ 
&512 & 4.6 & 100.0 & 100.0 & 98.0 & 95.8 & 93.4 & 90.2 & 8.58 & 2.58 \\  
$(\mathcal{M}_{a,1})$
&45 & 11.0 & 23.8 & 19.8 & 31.0 & 26.2 & 23.0 & 21.4 & 7.26 & 0.46 \\ 
&60 & 10.2 & 20.2 & 16.6 & 24.4 & 20.4 & 18.0 & 15.4 & 7.30 & 0.49 \\ 
&75 & 9.8 & 18.0 & 13.4 & 22.0 & 18.8 & 16.4 & 13.2 & 7.30 & 0.47 \\ 
&90 & 9.6 & 17.2 & 13.2 & 20.0 & 16.8 & 15.2 & 13.0 & 7.31 & 0.48 \\ 
&512 & 5.6 & 7.4 & 5.2 & 7.0 & 4.8 & 4.2 & 4.0 & 7.35 & 0.50 \\ 
$(\mathcal{M}_{a,2})$
&45 & 8.8 & 9.2 & 9.2 & 10.8 & 8.8 & 8.4 & 7.4 & 7.29 & 0.50 \\ 
&60 & 7.4 & 6.0 & 5.8 & 6.4 & 5.2 & 4.2 & 3.4 & 7.31 & 0.51 \\ 
&75 & 6.8 & 4.2 & 5.6 & 5.0 & 4.0 & 3.0 & 2.4 & 7.33 & 0.53 \\ 
&90 & 6.4 & 3.4 & 5.2 & 4.4 & 3.0 & 1.6 & 2.2 & 7.35 & 0.55 \\ 
&512 & 4.8 & 1.6 & 1.6 & 1.0 & 0.6 & 0.6 & 0.4 & 7.39 & 0.59 \\ 
$(\mathcal{M}_{a,3})$
&45 & 10.4 & 42.4 & 32.0 & 59.2 & 47.6 & 40.6 & 37.6 & 7.38 & 0.59 \\ 
&60 & 10.3 & 37.8 & 27.8 & 53.2 & 42.6 & 35.6 & 30.8 & 7.40 & 0.60 \\ 
&75 & 8.2 & 35.2 & 25.4 & 51.2 & 40.2 & 33.2 & 28.2 & 7.43 & 0.64 \\ 
&90 & 7.8 & 33.4 & 24.2 & 47.4 & 37.0 & 31.4 & 25.0 & 7.43 & 0.65 \\ 
&512 & 7.0 & 26.8 & 16.2 & 35.4 & 25.4 & 20.6 & 16.4 & 7.52 & 0.68
		\end{tabular} \medskip
		
		\caption{\it 
		Empirical rejection rates for various combined tests, based on a sample size of $T=512$ and a block length parameter $m$ calculated as proposed in Section~\ref{subsec:choiceM}. The last two columns provide the mean and standard deviation of the selected value of $m$.} \label{tab:t512}
	\end{table}
}

From the previous results, it can be seen that different choices of $n$ do not lead to crucially different results. For $T=256$, the tests for the hypotheses $H_0^{\scs (m)}$ and $H_0^{\scs (c,0)}$ already have good power against the alternatives $(\mathcal{M}_{m,1}),(\mathcal{M}_{m,3})$ and $(\mathcal{M}_{v,1}),(\mathcal{M}_{v,2}),(\mathcal{M}_{v,3})$ respectively. When combining $H_0^{\scs (m)}$ and $H_0^{\scs (c,0)}$ and taking even more autocovariances into account, the test does not loose significant power. For $T=512$, the power further increases such that all tests have good power against the alternatives $(\mathcal{M}_{m,i})$ and $(\mathcal{M}_{v,i})$, $i=1,2,3$.
Detecting non-stationarities in models $(\mathcal{M}_{a,i}),i=1,2,3$ turns out to be more difficult. Even though the power increases with $T$, for small values of $T$, the results are not too convincing. These findings can be explained by the fact that the \textit{measures of non-stationarity} $\|M\|_{2,2}$ and $\|M_h\|_{2,3}$, as introduced in \eqref{eq:mut} and \eqref{eq:mhutt}, are comparably small for models $(\mathcal{M}_{a,i}),i=1,2,3$. This can be deduced from Table~\ref{tab:M}, where these measures of non-stationarity are approximated by their natural estimators $\|M_T\|_{2,2}=\|U_T\|_{2,2}/\sqrt{T}$ and $\|M_{T,h}\|_{2,3}=\|U_{T,h}\|_{2,3}/\sqrt{T}$, based on 2,000 Monte-Carlo repetitions and  for various choices of $T$. 
It is noticeable that the values for models $(\mathcal{M}_{a,1})$ and $(\mathcal{M}_{a,2})$ are close to those for $(\mathcal{M}_{0})$, which perfectly explains the results of the simulation study.




{\scriptsize
	\begin{table}[!htbp] 
		\centering\begin{tabular}{cc|rrrrrr}
			T&Model & $\|M_T\|_{2,2}$ & $\|M_{T,0}\|_{2,3}$ & $\|M_{T,1}\|_{2,3}$ & $\|M_{T,2}\|_{2,3}$ & $\|M_{T,3}\|_{2,3}$ & $\|M_{T,4}\|_{2,3}$\\
			\hline
			256& $(\mathcal{M}_0)$ & 0.0759 & 0.2344 & 0.2273 & 0.2272 & 0.2274 & 0.2277 \\
			&$(\mathcal{M}_{m,1})$ & 0.1180  & 0.3208  & 0.3152  & 0.3145  & 0.3146  & 0.3138\\
			&$(\mathcal{M}_{m,2})$ & 0.0952  & 0.2812  & 0.2754  & 0.2754  & 0.2758  & 0.2760\\
			&$(\mathcal{M}_{v,1})$ & 0.0803  & 0.5394  & 0.3124  & 0.2989  & 0.2979  & 0.2977\\
			&$(\mathcal{M}_{v,2})$ & 0.0698  & 0.3681  & 0.2471  & 0.2412  & 0.2412  & 0.2422\\
			&$(\mathcal{M}_{a,1})$ & 0.0764  & 0.2426  & 0.2389  & 0.2328  & 0.2306  & 0.2301\\
			&$(\mathcal{M}_{a,2})$ & 0.0758  & 0.2358  & 0.2304  & 0.2277  & 0.2273  & 0.2273\\
			\hline
			512 & $(\mathcal{M}_0)$ & 0.0540 & 0.1659 & 0.1607 & 0.1603 & 0.1604 & 0.1604 \\
			&$(\mathcal{M}_{m,1})$ & 0.1049  & 0.2590  & 0.2562  & 0.2556  & 0.2547  & 0.2549\\
			&$(\mathcal{M}_{m,2})$ & 0.0781  & 0.2144  & 0.2097  & 0.2095  & 0.2097  & 0.2100\\
			&$(\mathcal{M}_{v,1})$ & 0.0572  & 0.4939  & 0.2313  & 0.2121  & 0.2105  & 0.2099\\
			&$(\mathcal{M}_{v,2})$ & 0.0494  & 0.3241  & 0.1791  & 0.1702  & 0.1695  & 0.1697\\
			&$(\mathcal{M}_{a,1})$ & 0.0542  & 0.1743  & 0.1743  & 0.1665  & 0.1632  & 0.1624\\
			&$(\mathcal{M}_{a,2})$ & 0.0537  & 0.1683  & 0.1652  & 0.1620  & 0.1606  & 0.1602\\
			\hline
			1024 & $(\mathcal{M}_0)$ & 0.0383 & 0.1172 & 0.1134 & 0.1132 & 0.1131 & 0.1133\\
			&$(\mathcal{M}_{m,1})$ & 0.0987  & 0.2250  & 0.2233  & 0.2229  & 0.2230  & 0.2223\\
			&$(\mathcal{M}_{m,2})$ & 0.0681  & 0.1710  & 0.1685  & 0.1685  & 0.1684  & 0.1684\\
			&$(\mathcal{M}_{v,1})$ & 0.0403  & 0.4696  & 0.1776  & 0.1518  & 0.1489  & 0.1486\\
			&$(\mathcal{M}_{v,2})$ & 0.0349  & 0.3000  & 0.1336  & 0.1210  & 0.1196  & 0.1198\\
			&$(\mathcal{M}_{a,1})$ & 0.0386  & 0.1289  & 0.1321  & 0.1215  & 0.1169  & 0.1152\\
			&$(\mathcal{M}_{a,2})$ & 0.0381  & 0.1212  & 0.1206  & 0.1162  & 0.1141  & 0.1134\\
			\hline
			2048 & $(\mathcal{M}_0)$ & 0.0270 & 0.0831 & 0.0802 & 0.0801 & 0.0800 & 0.0801\\
			&$(\mathcal{M}_{m,1})$ & 0.0949  & 0.2047  & 0.2041  & 0.2038  & 0.2035  & 0.2036\\
			&$(\mathcal{M}_{m,2})$ & 0.0624  & 0.1449  & 0.1431  & 0.1431  & 0.1431  & 0.1431\\
			&$(\mathcal{M}_{v,1})$ & 0.0283  & 0.4568  & 0.1430  & 0.1097  & 0.1055  & 0.1049\\
			&$(\mathcal{M}_{v,2})$ & 0.0245  & 0.2869  & 0.1035  & 0.0866  & 0.0846  & 0.0844\\
			&$(\mathcal{M}_{a,1})$ & 0.0272  & 0.0973  & 0.1040  & 0.0906  & 0.0845  & 0.0821\\
			&$(\mathcal{M}_{a,2})$ & 0.0269  & 0.0884  & 0.0898  & 0.0839  & 0.0814  & 0.0804
		\end{tabular} \medskip
		
		\caption{\it 
		$\|M_{T}\|_{2,2}$ and $\|M_{T,h}\|_{2,3},h=0,\cdots, 4$, calculated by $2,000$ Monte-Carlo repetitions. }\label{tab:M}
	\end{table}
}
	
	\subsection{Case Study}
	
	Functional time series naturally arise in the field of meteorology. For instance, the daily minimal temperature at one place over  time can be naturally divided into yearly functional data. 
	
	To illustrate the proposed methodology, we consider the daily minimum temperature recorded at eight different locations across Australia. Exemplary, the temperature curves of Melbourne and Sydney are displayed in Figure~\ref{fig:data}.
	The results of our testing procedure can be found in Table \ref{tab:data}, where we employed $K=1000$ bootstrap replicates, considered up to $H=4$ lags and chose $n=25$, based on visual exploration of the respective plots. The null hypotheses of stationarity can be rejected, at level $\alpha=0.05$, for all measuring stations except of Gunnedah Pool, for which the $p$-values exceed $\alpha$ by a small amount.
	
	
	\begin{figure}[!htbp] 
		\centering
		\subfloat[Melbourne]{\label{figur:5}\includegraphics[width=60mm]{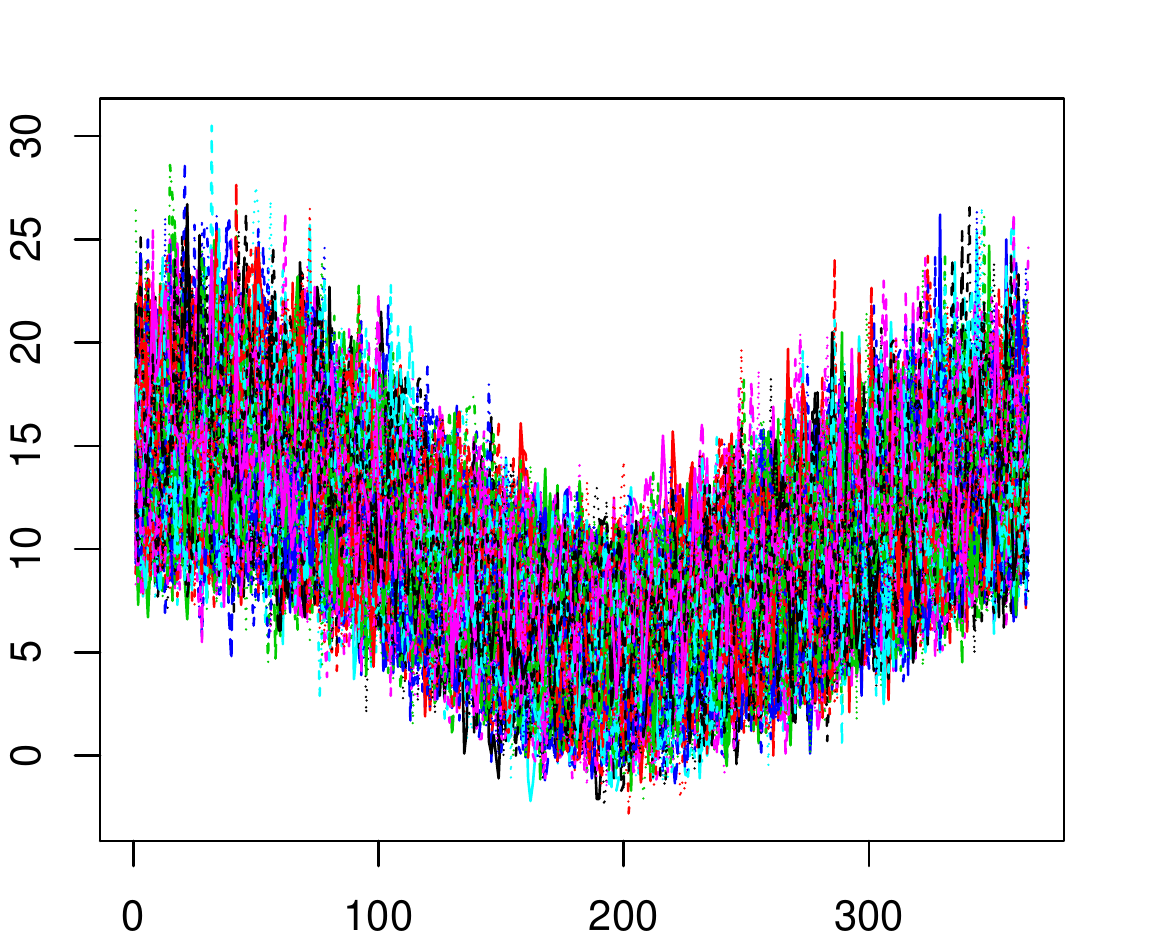}}
		\subfloat[Sydney]{\label{figur:8}\includegraphics[width=60mm]{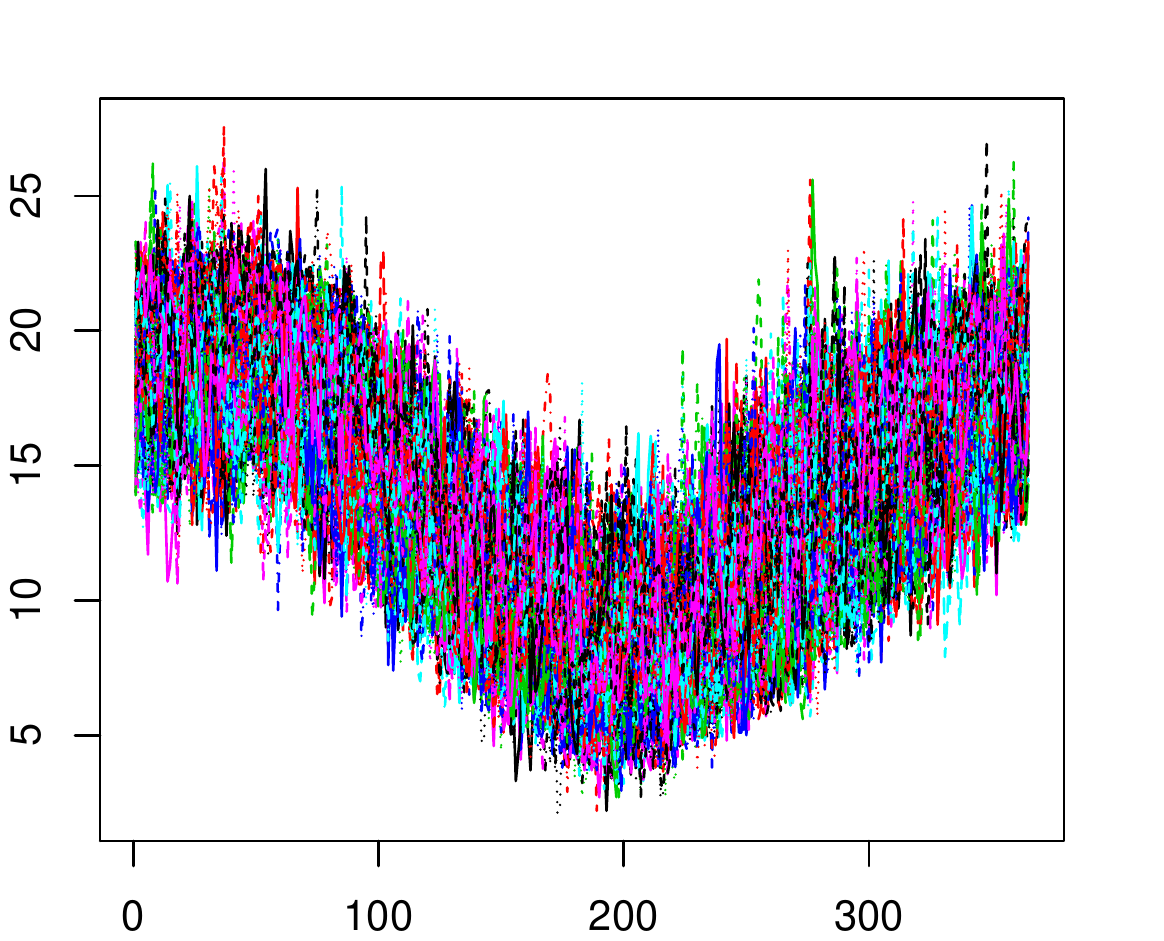}}
		\caption{ \it 
		Temperature curves of Melbourne ($T=161$ years) and Sydney ($T=160$ years), where the x-axis corresponds to a year in rescaled time and the y-axis denotes temperature in degree Celsius.}	\label{fig:data}
	\end{figure}

{\small
	\begin{table}[!htbp] 
	\centering\begin{tabular}{l|r|rrrrrrr|r}
		Location&T & $H_0^{(m)}$ & $H_0^{(c,0)}$ & $H_0^{(0)}$ & $H_0^{(1)}$ & $H_0^{(2)}$ & $H_0^{(3)}$ & $H_0^{(4)}$ & $m$  \\ 
		\hline
	Boulia Airport &131&0.3& 0.0& 0.0& 0.0& 0.0& 0.0& 0.0&5\\
	Gayndah Post Office &117&0.0& 0.0& 0.0& 0.0& 0.0& 0.0& 0.0&4\\
	Gunnedah Pool &136&6.4&  5.8&  5.8&  5.9&  5.8&  5.3&  5.1&6\\
	Hobart &137&0.0& 0.0& 0.0& 0.0& 0.0& 0.0& 0.0&8\\
	Melbourne&161&0.0& 0.0& 0.0& 0.0& 0.0& 0.0& 0.0&5\\	
	Cape Otway Lighthouse&155&0.5&  0.4&  0.4&  0.1&  0.0&  0.0&  0.0& 7\\
	Robe&135&4.5&  1.0&  1.9&  1.0&  1.0&  0.6&  0.5& 6\\	
	Sydney&160&0.0& 0.0& 0.0& 0.0& 0.0& 0.0& 0.0&5\\	
	\end{tabular}  \medskip
	
	\caption{\it  
	$p$-values of the (combined) tests for the respective null hypotheses in percent, and selected value of $m$.}\label{tab:data}
\end{table}
}

	\section{Proofs} \label{sec:proofs}

	Throughout the proofs, $C$ denotes a generic constant whose value may change from line to line.  If not specified otherwise, all convergences are for $T\to\infty$.

\subsection{A fundamental approximation lemma in Hilbert spaces}

\begin{lemma}\label{lem:app}
Fix $p\in\N$. For $i=1, \dots, p$ and $T\in\N$, let $X_{i,T}$ and $X_i$ denote random variables in a separable Hilbert space $(H_i, \langle \cdot, \cdot \rangle_i)$. Futher, let $(\psi_k^{\scs (i)})_{k\in\naturals}$ be an orthonormal basis of $H_i$ and for brevity write $\langle \cdot, \cdot \rangle = \langle \cdot, \cdot \rangle_i$.  Suppose that
\begin{align*}
(1)\quad& Y_T^n:=\big((\langle X_{1,T},\psi_k^{(1)}\rangle)_{k=1}^n,\dots ,(\langle X_{p,T},\psi_k^{(p)}\rangle)_{k=1}^n\big) \\
& \hspace{.6cm} \weak \big((\langle X_{1},\psi_k^{(1)}\rangle)_{k=1}^n,\dots ,(\langle X_{p},\psi_k^{(p)}\rangle)_{k=1}^n\big)=:Y^n \quad \text{as $T\to\infty$, for  any } n\in\N,  \\
(2)\quad& \lim\limits_{n\to\infty} 
\limsup\limits_{T\to\infty} \pr\Big(\sum_{k=n+1}^{\infty}\sum_{i=1}^{p}\langle X_{i,T},\psi_k^{(i)}\rangle^2 > \eps\Big)
=0 \quad \text{ for all } \eps>0.
\end{align*}
	Then, using the notation $\|(x_k)_{k\in\N}\|_2=\sum_{k=1}^\infty x_k^2$,
	\begin{multline*}
	Y_T^\infty:= \big((\langle X_{1,T},\psi_k^{(1)}\rangle)_{k=1}^\infty,\dots ,(\langle X_{p,T},\psi_k^{(p)}\rangle)_{k=1}^\infty\big) \\
	 \weak  
	 \big((\langle X_1,\psi_k^{(1)}\rangle)_{k=1}^\infty,\dots, (\langle X_p,\psi_k^{(p)}\rangle)_{k=1}^\infty\big) =:Y^\infty \quad \text{ in } (\ell^2(\naturals), \| \cdot\|_2)^p
	\end{multline*}
	 and, as a consequence, 
	 \begin{align*} 
	(X_{1,T},\dots ,X_{p,T})
	\weak (X_1,\dots , X_p)  \quad \text{in } H_1 \times  \dots \times H_p.
	\end{align*}
	\end{lemma}

\begin{proof}[Proof of Lemma~\ref{lem:app}]
	To prove the first part, we employ Theorem 2 of \cite{Dehling2009}. Expand the random variables $Y_T^n$ and $Y_n$ in $\reals^{pn}$ to 
	\[
	\tilde{Y}_{T,n}^\infty=\big((a_{T,k}^{(1)})_{k\in\naturals},\dots (a_{T,k}^{(p)})_{k\in\naturals}\big)~\text{and}~\tilde{Y}_{n}^\infty=\big((a_{k}^{(1)})_{k\in\naturals},\dots ,(a_{k}^{(p)})_{k\in\naturals}\big)
	\] 
	in $(\ell^2(\naturals), \| \cdot\|_2)^p$, where $a_{T,k}^{\scs (i)}=\langle X_{i,T},\psi_k^{\scs(i)}\rangle$ and $a_{k}^{\scs(i)}=\langle X_i,\psi_k^{\scs(i)}\rangle$, for any $1\leq k\leq n$, and $a_{T,k}^{\scs(i)}=a_k^{\scs(i)}=0$, for any $k>n$, $i=1,\dots ,p$. By the continuous mapping theorem, $\tilde{Y}_{T,n}^\infty$ converges weakly to $\tilde{Y}_n^\infty$ in $(\ell^2(\naturals), \| \cdot\|_2)^p$, for any $n\in\naturals$ and as $T$ tends to infinity.
		
	By assumption (2) and since the space $(\ell^2(\naturals), \| \cdot\|_2)^p$ is separable and complete, there is a random variable $\tilde{Y}^\infty\in (\ell^2(\naturals), \| \cdot\|_2)^p$ such that $Y_T^\infty\convw \tilde{Y}^\infty$, as $T$ tends to infinity, and $\tilde{Y}_n^\infty \convw \tilde{Y}^\infty$, as $n$ tends to infinity, by Theorem 2 of \cite{Dehling2009}. Due to the latter convergence, the finite dimensional distributions of $\tilde{Y}^\infty$ and $Y^\infty$ are the same. Thus, by Theorem 1.3 of \cite{Billingsley1999} and Lemma 1.5.3 of \cite{VanWel96}, $\tilde{Y}^\infty$ and $Y^\infty$ have the same distribution in $ (\ell^2(\naturals), \| \cdot\|_2)^p$.
	
	Next, observe that, for an arbitrary Hilbert space $H$, the function
	\begin{center}
		$\Phi:=\Bigg\{\begin{array}{llc}\ell^2(\naturals)&\to &H \\ (y_k)_{k\in\naturals}&\mapsto &\sum_{k=1}^\infty y_k \psi_k\end{array}$
	\end{center}
	is continuous, provided $(\psi_k)_{k\in\naturals}$ is an orthonormal basis of $H$. Indeed
	\begin{align*}
	\big\| \Phi\big((y_k)_{k}\big)-\Phi\big((z_k)_{k}\big)\big\|^2
	=\textstyle \sum_{k=1}^{\infty} (y_k-z_k)^2
	= \|y-z\|_2^2.
	\end{align*}
	Thus, the mapping 
		\begin{center}
		$\Phi':=\Bigg\{\begin{array}{llc}(\ell^2(\naturals), \| \cdot\|_2)^p &\to &H_1 \times \dots \times H_p \\ \big((y_{k,1})_{k\in\naturals},\dots ,(y_{k,p})_{k\in\naturals}\big)&\mapsto &\big(\sum_{k=1}^\infty y_{k,1} \psi_k^{(1)},\dots ,\sum_{k=1}^\infty y_{k,p} \psi_k^{(p)}\big)\end{array}$
	\end{center}
	is continuous too, and the continuous mapping theorem implies that
	\begin{align*} 
	\textstyle (X_{1,T},\dots ,X_{p,T})&= \textstyle \Big(\sum_{k=1}^{\infty} \langle X_{1,T},\psi_k^{(1)}\rangle \psi_k^{(1)},\dots  \sum_{k=1}^{\infty} \langle X_{p,T},\psi_k^{(p)}\rangle \psi_k^{(p)}\Big)\\ 
	&\convw 
	\textstyle  \Big(\sum_{k=1}^{\infty} \langle X_1,\psi_k^{(1)}\rangle \psi_k^{(1)},\dots ,\sum_{k=1}^{\infty} \langle X_p,\psi_k^{(p)}\rangle \psi_k^{(p)}\Big)=(X_1,\dots ,X_p), \end{align*}
	as $T$ tends to infinity.
\end{proof}

	\subsection{Proofs for Section~\ref{subsec:so}, \ref{subsec:test}, \ref{subsec:gen} and \ref{subsec:boot}}
	\label{subsec:proofs3}
	
	\begin{proof}[Proof of Lemma~\ref{lem:equivh}] We only prove the equivalence concerning $H_0^{\scs (h)}$; the equivalences regarding $H_0^{\scs (m)}$ follow along similar lines.
	
	\textit{Step 1: Equivalence between \eqref{eq:h0c} and \eqref{eq:h0c1}.}  Suppose that \eqref{eq:h0c1} is met. To prove \eqref{eq:h0c}, it is sufficient to show that 
	\begin{align}  \label{eq:ps}
		\Big\| \ex[X_0^{(u)}\otimes X_h^{(u)}]-\int_0^1 \ex[X_0^{(w)}\otimes X_h^{(w)}]\diff w\Big\|_{2,2}=0
	\end{align}
	for any $u\in[0,1]$. 
	
	Fix $u\in[0,1)$ and let $\delta>0$ be sufficiently small such that $u+\delta<1$. 
	By the reverse triangle inequality, we obtain that
	\begin{align*}
	0 \leq &\  
	\Bigg| \bigg\|\frac{1}{\delta}\int_u^{u+\delta} \ex\big[ X_0^{(w)}\otimes X_h^{(w)} -X_0^{(u)}\otimes X_h^{(u)}\big]\diff w\bigg\|_{2,2} \\
	&\hspace{3cm}
	- \bigg\|\int_0^1 \ex\big[X_0^{(w)}\otimes X_h^{(w)}- X_0^{(u)}\otimes X_h^{(u)} \big]\diff w\bigg\|_{2,2} \Bigg| \\
	\leq&\  
	\frac{1}{\delta} \bigg\| \int_u^{u+\delta} \ex\big[ X_0^{(w)}\otimes X_h^{(w)} \big]\diff w- \delta \int_0^1 \ex\big[ X_0^{(w)}\otimes X_h^{(w)} \big]\diff w\bigg\|_{2,2} \\
	=&\ 
	\frac{1}{\delta} \bigg\|\int_0^{u+\delta} \ex\big[ X_0^{(w)}\otimes X_h^{(w)} \big]\diff w-(u+\delta) \int_0^1 \ex\big[ X_0^{(w)}\otimes X_h^{(w)} \big]\diff w\\
	&\hspace{3cm}-
	\int_0^{u} \ex\big[ X_0^{(w)}\otimes X_h^{(w)} \big]\diff w+ u \int_0^1 \ex\big[ X_0^{(w)}\otimes X_h^{(w)} \big]\diff w\bigg\|_{2,2} \\
	\leq&\ 
	\frac{1}{\delta} \bigg\|\int_0^{u+\delta} \ex\big[ X_0^{(w)}\otimes X_h^{(w)} \big]\diff w-(u+\delta) \int_0^1 \ex\big[ X_0^{(w)}\otimes X_h^{(w)} \big]\diff w\bigg\|_{2,2} \\
	&\hspace{3cm}+ 
	\frac{1}{\delta} \bigg\|\int_0^u \ex\big[ X_0^{(w)}\otimes X_h^{(w)} \big]\diff w-u \int_0^1 \ex\big[ X_0^{(w)}\otimes X_h^{(w)} \big]\diff w\bigg\|_{2,2}.
		\end{align*}
	By continuity of integrals in the upper integration limit, it follows from \eqref{eq:h0c1} that both summands on the right-hand side of this display are equal to zero. As a consequence, 
	\begin{multline}  \label{eq:ps2}
	\bigg\| \int_0^1 \ex[X_0^{(w)}\otimes X_h^{(w)}] \diff w - \ex [X_0^{(u)}\otimes X_h^{(u)} ] \bigg\|_{2,2} \\
	=
	\bigg\|\frac{1}{\delta}\int_u^{u+\delta} \ex[ X_0^{(w)}\otimes X_h^{(w)} -X_0^{(u)}\otimes X_h^{(u)}]\diff w\bigg\|_{2,2}. 
	\end{multline}
	By Jensen's inequality, we can bound the right-hand side of this display from above by
	  \begin{equation}\label{consistencyPart1.1}
	  \bigg( \int_{[0,1]^2} \frac{1}{\delta^2} \int_u^{u+\delta}\big(\ex[X_0^{(w)}(\tau_1)X_h^{(w)}(\tau_2)-X_0^{(u)}(\tau_1)X_h^{(u)}(\tau_2)]\big)^2 \diff w \diff (\tau_1,\tau_2)\bigg)^{1/2}.
	  \end{equation}
	By employing Jensen's inequality again, we can bound the integrand by \[\ex\big[\big(X_0^{(w)}(\tau_1)X_h^{(w)}(\tau_2)-X_0^{(u)}(\tau_1)X_h^{(u)}(\tau_2)\big)^2\big].\] Thus, by Fubini's theorem, \eqref{consistencyPart1.1} is less than or equal to
	\[ \bigg( \frac{1}{\delta^2} \int_u^{u+\delta}\ex[\|X_0^{(w)}\otimes X_h^{(w)}-X_0^{(u)}\otimes X_h^{(u)}\|_{2,2}^2] \diff w\bigg)^{1/2}. \]
	This term is of the order $O(\delta^{1/2})$ due to the inequality
	\begin{align}
	&\ \ex \| X_0^{(w)}\otimes X_h^{(w)}-X_0^{(u)}\otimes X_h^{(u)} \|_{2,2}^2 \nonumber \\
	=&\
	\ex \| X_0^{(w)}\otimes X_h^{(w)}-X_0^{(u)}\otimes X_h^{(w)}+X_0^{(u)}\otimes X_h^{(w)}-X_0^{(u)}\otimes X_h^{(u)} \|_{2,2}^2 \nonumber\\
	\leq&\ 2\big\{  \ex \| (X_0^{(w)}-X_0^{(u)})\otimes X_h^{(w)}\|_{2,2}^2+ \ex \|X_0^{(u)}\otimes (X_h^{(w)}-X_h^{(u)}) \|_{2,2}^2\big \} \nonumber\\
	\le &\ 
	2\big\{  \ex[ \| X_0^{(w)}-X_0^{(u)}\|_2^4]^{1/2} \ \ex[\|X_h^{(w)}\|_{2}^4]^{1/2}+ \ex[ \|X_0^{(u)}\|_2^4]^{1/2} \ \ex[ \|X_h^{(w)}-X_h^{(u)} \|_{2}^4]^{1/2} \big\}  \nonumber\\
	\leq&\
	C |u-w|^2, \label{eq:lip2}
	\end{align}
	where the final bound follows from Lemma~\ref{lem:lip}. Since $\delta$ was chosen arbitrarily, we obtain that the right-hand side of \eqref{eq:ps2} is equal to zero. This proves \eqref{eq:ps} for $u\in[0,1)$, and the case $u=1$ follows from \eqref{eq:lip2}, which is also valid for $u=1$.
		
	Conversely, if \eqref{eq:h0c} holds true, we have by a change of variables, linearity of the integral, Jensen's inequality and Fubini's theorem, 
	\begin{align*}
	&\ \bigg\|\int_0^u \ex\big[ X_0^{(w)}\otimes X_h^{(w)} \big]\diff w-u \int_0^1 \ex\big[ X_0^{(w)}\otimes X_h^{(w)} \big]\diff w\bigg\|_{2,3}^2\\
	=&\
	\bigg\|u \int_0^1 \ex\big[ X_0^{(uw)}\otimes X_h^{(uw)} \big]- \ex\big[ X_0^{(w)}\otimes X_h^{(w)} \big]\diff w\bigg\|_{2,3}^2 \\
	\leq&\ 
	\int_0^1 \int_0^1 u^2 \big\|\ex[X_0^{(uw)}\otimes X_h^{(uw)}]-\ex[X_0^{(w)}\otimes X_h^{(w)}]\big\|_{2,2}^2 \diff w \diff u=0.
	\end{align*}
	
	\textit{Step 2: Equivalence between \eqref{eq:h0c} and \eqref{eq:h0c2}.}
	Note that, irrespective of whether \eqref{eq:h0c} or \eqref{eq:h0c2} is met, local stationarity of $X_{t,T}$ of order $\rho\ge 4$ and stationarity of $(X_t^{\scs (u)})_{t\in\Z}$ with $\Exp\|X_t^{\scs (u)}\|_2^4 < \infty$ implies that
	\begin{align} \label{eq:mdb}
	&\ \| \ex[X_{\lfloor uT\rfloor,T}\otimes X_{\lfloor uT\rfloor+h,T}- X_{\lfloor uT\rfloor}^{(u)}\otimes X_{\lfloor uT\rfloor+h}^{(u)}]\|_{2,2}^2 \nonumber \\
	\leq&\ \ex[\|X_{\lfloor uT\rfloor,T}\otimes X_{\lfloor uT\rfloor+h,T}-X_{\lfloor uT\rfloor}^{(u)}\otimes X_{\lfloor uT\rfloor+h}^{(u)}\|_{2,2}^2]  \nonumber \\
	\leq&\ 2 \big\{ \ex[\|(X_{\lfloor uT\rfloor,T}-X_{\lfloor uT\rfloor}^{(u)})\otimes X_{\lfloor uT\rfloor+h,T}\|_{2,2}^2]+ \ex[\|X_{\lfloor uT\rfloor}^{(u)}\otimes (X_{\lfloor uT\rfloor+h,T}-X_{\lfloor uT\rfloor+h}^{(u)})\|_{2,2}^2]\big\}  \nonumber \\	
	\leq&\ 2 \big\{ \ex[\|X_{\lfloor uT\rfloor,T}-X_{\lfloor uT\rfloor}^{(u)}\|_2^2 \|X_{\lfloor uT\rfloor+h,T}\|_2^2]+ \ex[\|X_{\lfloor uT\rfloor}^{(u)}\|_2^2 \|X_{\lfloor uT\rfloor+h,T}-X_{\lfloor uT\rfloor+h}^{(u)}\|_2^2]\big\}   \nonumber \\
	\leq&\ (C/T^2) \big\{ \ex[ (P_{\lfloor uT\rfloor,T}^{(u)})^4]^{1/2}  \ex \|X_{\lfloor uT\rfloor+h,T}\|_2^4]^{1/2}  + \ex[\|X_{\lfloor uT\rfloor}^{(u)}\|_2^4]^{1/2} \ex[ (P_{\lfloor uT\rfloor+h,T}^{(u)})^4]^{1/2}  \big\}   \nonumber \\
	\leq&\ {C}/{T^2},
	\end{align}
	for any $u\in[0,1]$ and $T\in\N$ and for some universal constant $C>0$.	
	
	Now, suppose that \eqref{eq:h0c} is met. Then, the previous display implies that
	\begin{align*}
	&\ \| \ex[X_{\lfloor uT\rfloor,T}\otimes X_{\lfloor uT\rfloor+h,T}]-\ex[X_{0,T}\otimes X_{h,T}]\|_{2,2}\\
	\leq&\ \| \ex[X_{\lfloor uT\rfloor,T}\otimes X_{\lfloor uT\rfloor+h,T}]-\ex[X_{\lfloor uT\rfloor}^{(u)}\otimes X_{\lfloor uT\rfloor+h}^{(u)}]\|_{2,2} \\
	&\ \quad + \| \ex[X_{0}^{(u)}\otimes X_{h}^{(u)}]-\ex[X_{0}^{(0)}\otimes X_{h}^{(0)}]\|_{2,2} + \| \ex[X_{0}^{(0)}\otimes X_{h}^{(0)}]-\ex[X_{0,T}\otimes X_{h,T}]\|_{2,2}\\
	\leq&\ {C}/{T} + 0 + {C}/{T} = {2C}/{T},
	\end{align*}
	for any $u\in[0,1]$ and $T\in\naturals$, that is, \eqref{eq:h0c2} is met.
		
	Conversely, if \eqref{eq:h0c2} is met, then, by \eqref{eq:mdb} and \eqref{eq:ls}, for any $u,v\in[0,1]$ and $T\in\N$,
	\begin{align*}
	&\ \|\ex[X_{0}^{(u)} \otimes X_h^{(u)}-X_{0}^{(v)} \otimes X_h^{(v)}]\|_2 \\
	=&\
	\|\ex[X_{\lfloor uT\rfloor}^{(u)} \otimes X_{\lfloor uT\rfloor+h}^{(u)}-X_{\lfloor vT\rfloor}^{(v)} \otimes X_{\lfloor vT\rfloor+h}^{(v)}]\|_2 \\
	\le &\
	\|\ex[X_{\lfloor uT\rfloor}^{(u)} \otimes X_{\lfloor uT\rfloor+h}^{(u)}-X_{\lfloor uT\rfloor,T} \otimes X_{\lfloor uT\rfloor+h,T}]\|_2 \\
	&\ \hspace{1cm}+
	\| \ex[X_{\lfloor uT\rfloor,T} \otimes X_{\lfloor uT\rfloor+h,T} - X_{0,T} \otimes X_{h,T}] \|_2 \\
	&\ \hspace{1cm}+
	\| \ex[X_{0,T} \otimes X_{h,T}-X_{\lfloor vT\rfloor,T} \otimes X_{\lfloor vT\rfloor+h,T} \|_2 \\
	&\ \hspace{1cm}+ 
	\|\ex[X_{\lfloor vT\rfloor,T} \otimes X_{\lfloor vT\rfloor+h,T}- X_{\lfloor vT\rfloor}^{(v)} \otimes X_{\lfloor vT\rfloor+h}^{(v)}]\|_2 \\
	 \leq&\ 4 C/T.
	\end{align*}
	Since $T$ was arbitrary, the left-hand side of this display must be zero, whence \eqref{eq:h0c}.
	\end{proof}

\begin{proof}[Proof of Theorem~\ref{theo:main}]
This theorem is an immediate consequence of Theorem~\ref{bootstrapThm}.
\end{proof}	

\begin{proof}[Proof of Corollary~\ref{cor:main1}] 
	 	Suppose that $H_0^{\scs(c,h)}$ is met. Then, by the triangle inequality and a slight abuse of notation (note that $u$ is a variable of integration in the norm $\|\cdot\|_{2,3}$), for $h\leq T$, 
	{\small \begin{align*}
	&\ \| U_{T,h}-G_{T,h}\|_{2,\Omega\times [0,1]^3}  \\
	=&\ \bigg\| \frac{1}{\sqrt{T}}\bigg(\sum_{t=1}^{\lfloor uT\rfloor \wedge (T-h)}\ex[X_{t,T}\otimes X_{t+h,T}]-u\sum_{t'=1}^{T-h}\ex[X_{t',T}\otimes X_{t'+h,T}]\bigg) \bigg\|_{2,3}\\
	=&\ \frac{1}{\sqrt{T}}\bigg\|\frac{1}{T-h}\sum_{t=1}^{\lfloor uT\rfloor \wedge (T-h)}\sum_{t'=1}^{T-h}\ex[X_{t,T}\otimes X_{t+h,T}]-\ex[X_{t',T}\otimes X_{t'+h,T}]\\
	& \hspace{3.5cm} +\bigg(\frac{1}{T-h}-\frac{u}{\lfloor uT\rfloor \wedge (T-h)}\bigg)\sum_{t=1}^{\lfloor uT\rfloor \wedge (T-h)}\sum_{t'=1}^{T-h}\ex[X_{t',T}\otimes X_{t'+h,T}] \bigg\|_{2,3}\\
	\leq&\  \frac{C}{T^{3/2}}\sum_{t,t'=1}^{T-h}\|\ex[X_{t,T}\otimes X_{t+h,T}]-\ex[X_{t',T}\otimes X_{t'+h,T}]\|_{2,2}+\frac{C}{T^{3/2}}\sum_{t=1}^{T-h}\|\ex[X_{t,T}\otimes X_{t+h,T}]\|_{2,2}.
	\end{align*}
	}This expression if of the order $O(T^{-1/2})$ by \eqref{eq:h0c2} and Assumption~\ref{cond:mom}. Hence,  $\| U_{T,h}-\tilde G_{T,h}\|_{2,3}=o_\Prob(1)$, and the assertion for $U_{T}$ follows along similar lines.
	 	
	Now, consider the assertion regarding the alternative 
	$H_1^{\scs (H)} = H_1^{\scs (m)} \cup  H_1^{\scs (c,0)}  \cup \dots \cup  H_1^{\scs (c,H)}$. 
	We only treat the case where $H_1^{\scs (c,h)}$ is met for some $h\in \{0, \dots, H\}$, the case $H_1^{\scs (m)}$ is similar. It is to be shown that $\| U_{T,h}\|_{2,3}\to \infty$ in probability. 
	
By the reverse triangle inequality, we have 
\[ 
	\| U_{T,h}\|_{2,3} = \| \tilde G_{T,h} + \ex U_{T,h}\|_{2,3} \geq \big| \| \tilde G_{T,h}\|_{2,3} - \|\ex U_{T,h}\|_{2,3} \big|.
\]
The term $\|\tilde G_{T,h}\|_{2,3}$ converges weakly to $\|\tilde G_h\|_{2,3}$. Thus, it suffices to show that the second term $\|\ex U_{T,h}\|_{2,3}$ diverges to infinity. For that purpose, note  that another application of the reverse triangle inequality implies that
	\begin{align*}
	\|\ex U_{T,h}\|_{2,3} =&\ \bigg\| \frac{1}{\sqrt{T}}\bigg(\sum_{t=1}^{\lfloor uT\rfloor \wedge (T-h)}\ex[X_{t,T}\otimes X_{t+h,T}]-u\sum_{t=1}^{T-h}\ex[X_{t,T}\otimes X_{t+h,T}]\bigg) \bigg\|_{2,3} \\
	\ge&\ |S_{1,T}-S_{2,T}|,
	\end{align*}
	where 
	\begin{multline*}
	S_{1,T}= \bigg\|\frac{1}{\sqrt{T}}\bigg(\sum_{t=1}^{\lfloor uT\rfloor \wedge (T-h)} \big\{\ex[X_{t,T}\otimes X_{t+h,T}]-\ex[X_t^{(t/T)}\otimes X_{t+h}^{(t/T)}]\big\} \\
	- u\sum_{t=1}^{T-h}\big\{\ex[X_{t,T}\otimes X_{t+h,T}]-\ex[X_t^{(t/T)}\otimes X_{t+h}^{(t/T)}]\big\}\bigg)\bigg\|_{2,3}
	\end{multline*}
	and
	\begin{align*}
	S_{2,T}&=\bigg\|\frac{1}{\sqrt{T}}\bigg(\sum_{t=1}^{\lfloor uT\rfloor \wedge (T-h)} \ex[X_t^{(t/T)}\otimes X_{t+h}^{(t/T)}]-u\sum_{t=1}^{T-h}\ex[X_t^{(t/T)}\otimes X_{t+h}^{(t/T)}]\bigg)\bigg\|_{2,3}.
	\end{align*}
	 In the following, we will show that $S_{1,T}$ vanishes as $T$ increases and that $S_{2,T}$ diverges to infinity. We have
	\begin{multline*}
	S_{1,T} \leq \bigg\{ \int_0^1\bigg(\frac{1}{\sqrt{T}} \sum_{t=1}^{\lfloor uT\rfloor \wedge(T-h)} \big\| \ex[X_{t,T}\otimes X_{t+h,T}]-\ex[X_t^{(t/T)}\otimes X_{t+h}^{(t/T)}] \big\|_{2,2}\\
	+ \frac{u}{\sqrt{T}} \sum_{t=1}^{T-h} \big\| \ex[X_{t,T}\otimes X_{t+h,T}]-\ex[X_t^{(t/T)}\otimes X_{t+h}^{(t/T)}] \big\|_{2,2} \bigg)^2 du \bigg\}^{1/2},
	\end{multline*}
	which is of order $O(T^{-1/2})$ since $\| \ex[X_{t,T}\otimes X_{t+h,T}]-\ex[X_t^{\scs (t/T)}\otimes X_{t+h}^{\scs (t/T)}] \|_{2,2}\leq C/T$ by \eqref{eq:ls}. For the second term $S_{2,T}$, we have, by stationarity
	\begin{align*}
	S_{2,T} &= \sqrt{T}\bigg\|\frac{1}{T}\bigg(\sum_{t=1}^{\lfloor uT\rfloor \wedge (T-h)} \ex[X_0^{(t/T)}\otimes X_{h}^{(t/T)}]-u\sum_{t=1}^{T-h}\ex[X_0^{(t/T)}\otimes X_{h}^{(t/T)}]\bigg)\bigg\|_{2,3},
	\end{align*}
	where the norm converges to 
	\[ \bigg\|\int_0^u \ex[X_0^{(w)}\otimes X_h^{(w)}]\diff w - u \int_0^1 \ex[X_0^{(w)}\otimes X_h^{(w)}]\diff w\bigg\|_{2,3}, \]
	by the dominated convergence theorem and the moment condition \ref{cond:mom}. The expression in the latter display is strictly positive since \eqref{eq:h0c1} is not satisfied and by the continuity of 
	\[ \bigg\|\int_0^u \ex[X_0^{(w)}\otimes X_h^{(w)}]\diff w - u \int_0^1 \ex[X_0^{(w)}\otimes X_h^{(w)}]\diff w\bigg\|_{2,2} \]
	in $u\in[0,1]$. Thus, $S_{2,T}\to \infty$, which implies the assertion.
	\end{proof}

\begin{proof}[Proof of Lemma  \ref{mixing}]
We will only give a proof of \eqref{sumcum}. Parts (i)-(iv) of the cumulant condition   \ref{cond:cum}
follow by similar arguments, which are omitted for the sake of brevity. 
According to Theorem  3 in \cite{StatuleviciusJakimavicius1988}, we have
{\small
\[ |\cum(X_{t_1,T}(\tau_1),\dots ,X_{t_k,T}(\tau_k))|\leq 3(k-1)!2^{k-1} \alpha^{\delta/(1+\delta)}(t_{i+1}-t_i) \prod_{j=1}^{k} \big(\ex|X_{t_j,T}(\tau_j)|^{(1+\delta)k}\big)^{\frac{1}{(1+\delta)k}}, \]
}
for any increasing sequence $t_1\leq t_2\leq \dots \leq t_k$. Straightforward calculations combined with H\"older's and Jensen's inequality lead to
\begin{align*}
\Big\|\prod_{j=1}^k \ex[|X_{t_j,T}|^{(1+\delta)k}]^{\frac{1}{(1+\delta)k}}\Big\|_{2,k}&= \prod_{j=1}^k \big\| \ex[|X_{t_j,T}|^{(1+\delta)k}]^{\frac{1}{(1+\delta)k}}\big\|_{2}\\
&=  \prod_{j=1}^k \bigg(\int_{[0,1]} \ex[|X_{t_j,T}(\tau)|^{(1+\delta)k}]^{\frac{2}{(1+\delta)k}}d\tau\bigg)^{1/2}\\
&\leq  \prod_{j=1}^k \bigg(\int_{[0,1]} \ex[|X_{t_j,T}(\tau)|^{(1+\delta)k}]d\tau\bigg)^{\frac{1}{(1+\delta)k}}\\
&= \prod_{j=1}^k \ex\Big[\big\|X_{t_j,T}\big\|_{(1+\delta)k}^{(1+\delta)k}\Big]^{\frac{1}{(1+\delta)k}}\\
&\leq  \sup_{t,T}\ex\Big[\big\|X_{t,T}\big\|_{(1+\delta)k}^{(1+\delta)k}\Big]^{1/(1+\delta)}\\
&\leq C_{k,1}.
\end{align*}
Thus, combining the previous results, leads to 
\begin{align*}
\|\cum(X_{t_1,T},\dots ,X_{t_k,T})\|_{2,k}
&\leq 3(k-1)!2^{k-1}C_{k,1} \alpha^{\delta/(1+\delta)}(t_{i+1}-t_i) \\
&\leq C_{k,4} \alpha^{\delta/(1+\delta)}(t_{i+1}-t_i),
\end{align*}
for any $i=1,\dots ,k-1$, where the constant $C_{k,4}>0$ depends on $k$ only. 
Hence,
\begin{equation*} 
\|\cum(X_{t_1,T},\dots ,X_{t_k,T})\|_{2,k}\leq C_{k,4} \prod_{i=1}^{k-1} \alpha^{\frac{\delta}{(1+\delta)(k-1)}}(t_{i+1}-t_i).
\end{equation*}
Analogously,  for arbitrary, not necessarily increasing $t_1,\dots ,t_k$, we may obtain that
\[\|\cum(X_{t_1,T},\dots ,X_{t_k,T})\|_{2,k}\leq C_{k,4} \prod_{i=1}^{k-1} \alpha^{\frac{\delta}{(1+\delta)(k-1)}}\big(t_{(i+1)}-t_{(i)}\big),\]
where $\big(t_{(1)},\dots ,t_{(k)}\big)$ denotes the order statistic of $(t_1,\dots ,t_k)$. The latter expression is symmetric in its arguments, thus we have, for any $t_k\in\integers$,
\begin{align*}
&\sum_{t_1,\dots ,t_{k-1}=-\infty}^{\infty} \big\| \cum(X_{t_1,T},\dots ,X_{t_k,T})\big\|_{2,k}\\
&\leq C_{k,4} \sum_{t_1,\dots ,t_{k-1}=-\infty}^{\infty} \prod_{i=1}^{k-1} \alpha^{\frac{\delta}{(1+\delta)(k-1)}}\big(t_{(i+1)}-t_{(i)}\big)\\
&\leq C_{k,4} (k-1)! \sum_{-\infty< t_1\leq \dots \leq t_{k-1}< \infty}  \prod_{i=1}^{k-1} \alpha^{\frac{\delta}{(1+\delta)(k-1)}}(t_{i+1}-t_{i})\\
&\leq C_{k,4} (k-1)! \sum_{-\infty< t_2\leq \dots \leq t_{k-1}< \infty} \sum_{s_1=-\infty}^{\infty} \alpha^{\frac{\delta}{(1+\delta)(k-1)}}(s_1) \prod_{i=2}^{k-1} \alpha^{\frac{\delta}{(1+\delta)(k-1)}}(t_{i+1}-t_{i}).
\end{align*}
By assumption $\{(X_{t,T})_{t\in\Z}: T\in\N\}$  is exponentially strong mixing, and the inner sum is finite and can be bounded by some constant $C_{k,2}$. Thus, 
\begin{multline*}
 \sum_{t_1,\dots ,t_{k-1}=-\infty}^{\infty} \big\| \cum(X_{t_1,T},\dots ,X_{t_k,T})\big\|_{2,k}\\
 \leq C_{k,4} (k-1)! C_{k,2} \sum_{-\infty< t_2\leq \dots \leq t_{k-1}< \infty}\prod_{i=1}^{k-1} \alpha^{\frac{\delta}{(1+\delta)(k-1)}}(t_{i+1}-t_{i}).
 \end{multline*}
Repeating this argument successively, we obtain finally \eqref{sumcum}
as asserted.

\end{proof}

\begin{proof}[Proof of Theorem~\ref{theo:boot}]
	By Slutsky's lemma and Theorem \ref{bootstrapThm}, it is sufficient to prove  that
	\[ 
	\big(\hat{\mathbb{B}}_{T}^{(1)}-\mathbb{B}_{T}^{(1)},\dots ,\hat{\mathbb{B}}_{T}^{(K)}-\mathbb{B}_{T}^{(K)}\big) =o_\Prob(1) 
	\]
	in $\{ L^2([0,1]^2)\times\{L^2([0,1]^3)\}^{H+1}\}^{K}$, as $T$ tends to infinity. This in turn is equivalent to 
	\[ 
	\big(\|\hat{B}_{T}^{(k)}-\tilde{B}_T^{(k)}\|_{2,3},\|\hat{B}_{T,0}^{(k)}-\tilde{B}_{T,0}^{(k)}\|_{2,3}\dots ,\|\hat{B}_{T}^{(k)}-\tilde{B}_{T,h}^{(k)}\|_{2,3}\big)_{k=1,\dots ,K} =o_\Prob(1)
	\]
	in $\reals^{K(H+2)}$. The last convergence holds true if and only if the coordinates converge, i.e., if $\|\hat{B}_{T}^{\scs (k)}-\tilde{B}_{T}^{\scs (k)}\|_{2,3}=o_\Prob(1)$ and $\|\hat{B}_{T,h}^{\scs (k)}-\tilde{B}_{T,h}^{\scs (k)}\|_{2,3}=o_\Prob(1)$, for all $k=1,\dots ,K$ and $h=0,\dots ,H$. We only consider the latter assertion (the former can be treated similarly)  and in fact, we will show convergence in $L^2(\Omega, \Prob)$, which is even stronger.  For this purpose observe that by Fubini's theorem and the independence of the family $(R_i^{\scs (k)})_{i\in\naturals}$
	{\small\begin{align*} 
	&\ex \|\hat{B}_{T,h}^{(k)}-\tilde{B}_{T,h}^{(k)}\|_{2,3}^2\\
	=\,& 
	\ex\Bigg[\int_{[0,1]^3}\frac{1}{mT}\bigg\{\sum_{i=1}^{\lfloor uT\rfloor\wedge(T-h)} R_i^{(k)} \sum_{t=i}^{(i+m-1)\wedge(T-h)}\mu_{t,T,h}(\tau_1, \tau_2)- \hat \mu_{t,T,h}(\tau_1, \tau_2) \bigg\}^2 \diff (u,\tau_1,\tau_2)\Bigg]\\
	=\,& 
	\frac{1}{mT}\int_{[0,1]^3}\sum_{i=1}^{\lfloor uT\rfloor\wedge(T-h)}  \ex\Bigg[\bigg\{\sum_{t=i}^{(i+m-1)\wedge(T-h)} A_{t,1} + A_{t,2}\bigg\}^2\Bigg]\diff (u,\tau_1,\tau_2),
	\end{align*}
	}where 
	\[ 
	A_{t,1}(\tau_1,\tau_2)= \textstyle \frac{1}{\tilde{n}_{t,h}}\sum_{k=\ubar{n}_t}^{\bar{n}_{t,h}}\ex[X_{t,T}(\tau_1)X_{t+h,T}(\tau_2)]-\ex[X_{t+k,T}(\tau_1)X_{t+k+h,T}(\tau_2)] 
	\]
	and
	\[ 
	A_{t,2}(\tau_1,\tau_2)= \textstyle \frac{1}{\tilde{n}_{t,h}}\sum_{k=\ubar{n}_t}^{\bar{n}_{t,h}}X_{t+k,T}(\tau_1)X_{t+k+h,T}(\tau_2)-\ex[X_{t+k,T}(\tau_1)X_{t+k+h,T}(\tau_2)].
	\]
	Since $A_{t,1}$ is deterministic and since $A_{t,2}$ is centred, we can rewrite the expectation in the previous integral as	
	\begin{multline*} 
	\ex\bigg[\bigg\{\sum_{t=i}^{(i+m-1)\wedge(T-h)}A_{t,1}(\tau_1,\tau_2)+A_{t,2}(\tau_1,\tau_2) \bigg\}^2\bigg]\\=\bigg(\sum_{t=i}^{(i+m-1)\wedge(T-h)}A_{t,1}(\tau_1,\tau_2)\bigg)^2+\ex\bigg[\bigg(\sum_{t=i}^{(i+m-1)\wedge(T-h)}A_{t,2}(\tau_1,\tau_2)\bigg)^2\bigg], 
	\end{multline*}
	In the following, we bound both parts separately. For the term $A_{t,1}$, first note that, by stationarity of $(X_t^{\scs (u)})_{t\in\Z}$, 
	\begin{align*}
	&\,\ex[X_{t,T}(\tau_1) X_{t+h,T}(\tau_2)]-\ex[X_{t+k,T}(\tau_1)X_{t+k+h,T}(\tau_2)]\\ 
	&=\,\ex[X_{t,T}(\tau_1) X_{t+h,T}(\tau_2)-X_t^{(t/T)}(\tau_1) X_{t+h}^{(t/T)}(\tau_2)]\\
	&\phantom{=}-\ex[X_{t+k,T}(\tau_1) X_{t+k+h,T}(\tau_2)-X_{t+k}^{(t/T)}(\tau_1) X_{t+k+h}^{(t/T)}(\tau_2)]
	\end{align*}
	in $L^2([0,1]^2)$. Thus, by Jensen's inequality and Fubini's theorem, we have	
	{\small\begin{align*}
	&\ \frac{1}{mT}\int_{[0,1]^3}\sum_{i=1}^{\lfloor uT\rfloor \wedge(T-h)}  \bigg(\sum_{t=i}^{(i+m-1) \wedge(T-h)}A_{t,1}\bigg)^2 \diff (u,\tau_1,\tau_2)\\
	\leq&\ 
	\frac{1}{mT}\int_{[0,1]^3}\sum_{i=1}^{\lfloor uT\rfloor \wedge(T-h)}   \ex\bigg[\bigg\{\sum_{t=i}^{(i+m-1)\wedge(T-h)}\frac{1}{\tilde{n}_{t,h}}\sum_{k=\ubar{n}_t}^{\bar{n}_{t,h}}X_{t,T}(\tau_1)X_{t+h,T}(\tau_2)-X_t^{(t/T)}(\tau_1)X_{t+h}^{(t/T)}(\tau_2)\\
	&\hspace{4cm}-X_{t+k,T}(\tau_1)X_{t+k+h,T}(\tau_2)+X_{t+k}^{(t/T)}(\tau_1)X_{t+k+h}^{(t/T)}(\tau_2) \bigg\}^2\bigg]\diff (u,\tau_1,\tau_2)\\
	\leq&\ \frac{1}{mT}\sum_{i=1}^{T-h}\ex \bigg\| \sum_{t=i}^{(i+m-1)\wedge(T-h)}\frac{1}{\tilde{n}_{t,h}}\sum_{k=\ubar{n}_t}^{\bar{n}_{t,h}} X_{t,T}\otimes X_{t+h,T}-X_t^{(t/T)}\otimes X_{t+h}^{(t/T)}\\
	&\hspace{4cm}-X_{t+k,T}\otimes X_{t+k+h,T}+X_{t+k}^{(t/T)}\otimes X_{t+k+h}^{(t/T)}\bigg\|_{2,2}^2.
	\end{align*}
	}The norm on the right-hand side of the previous inequality be bounded by the triangle inequality by
	\begin{multline*}
	\textstyle \sum_{t=i}^{(i+m-1)\wedge(T-h)}\frac{1}{\tilde{n}_{t,h}}\sum_{k=\ubar{n}_t}^{\bar{n}_{t,h}} \|X_{t,T}\otimes X_{t+h,T}-X_t^{(t/T)}\otimes X_{t+h}^{(t/T)}\|_{2,2}\\
	+\|X_{t+k,T}\otimes X_{t+k+h,T}-X_{t+k}^{(t/T)}\otimes X_{t+k+h}^{(t/T)}\|_{2,2}
	\end{multline*}
	and the inner summands can be bounded due to the local stationarity of $(X_{t,T})$: first,
	\begin{align*}
	\|X_{t,T} \, \otimes &\, X_{t+h,T}-X_t^{(t/T)}\otimes X_{t+h}^{(t/T)}\|_{2,2} \\
	\leq&\ \| X_{t,T}\otimes (X_{t+h,T}-X_{t+h}^{(t/T)})\|_{2,2}+\|X_{t+h}^{(t/T)}\otimes (X_{t,T}-X_t^{(t/T)})\|_{2,2}\\
	=&\ \| X_{t,T}\|_2\|X_{t+h,T}-X_{t+h}^{(t/T)}\|_2+\|X_{t+h}^{(t/T)}\|_2\|X_{t,T}-X_t^{(t/T)}\|_2\\
	\leq&\ T^{-1}\big\{(h+1)\|X_{t,T}\|_2+\|X_{t+h}^{(t/T)}\|_2\big\}P_{t,T}^{(t/T)}
	\end{align*}
	and similarly
	\begin{align*}
	\|X_{t+k,T} \, \otimes &\, X_{t+k+h,T}-X_{t+k}^{(t/T)}\otimes X_{t+k+h}^{(t/T)}\|_{2,2} \\
	\leq&\ T^{-1}\big\{(|k+h|+1)\|X_{t+k,T}\|_2+(|k|+1)\|X_{t+k+h}^{(t/T)}\|_2\big\}P_{t,T}^{(t/T)}.
	\end{align*}
	Assembling bounds, we obtain that
	{\small\begin{align*}
	&\ 
	\frac{1}{mT}\int_{[0,1]^3}\sum_{i=1}^{\lfloor uT\rfloor \wedge(T-h)}  \bigg(\sum_{t=i}^{(i+m-1)\wedge(T-h)}A_{t,1}(\tau_1,\tau_2)\bigg)^2 \diff (u,\tau_1,\tau_2)\\
	\leq&\ 
	\frac{1}{mT}\sum_{i=1}^{T-h}\sum_{t,t'=i}^{(i+m-1)\wedge(T-h)}\frac{1}{\tilde{n}_{t,h} \tilde{n}_{t',h}} \sum_{k=\ubar{n}_t}^{\bar{n}_{t,h}}\sum_{k'=\ubar{n}_{t'}}^{\bar{n}_{t',h}}\frac{1}{T^2}\ex\Big[(|k|+h+1)(|k'|+h+1)P_{t,T}^{(t/T)}P_{t',T}^{(t'/T)}\\
	&\ \hspace{4cm}
	\times\big(\|X_{t,T}\|_2+\|X_{t+h}^{(t/T)}\|_2+\|X_{t+k,T}\|_2+\|X_{t+k+h}^{(t/T)}\|_2\big)\\
	&\ \hspace{4cm}
	\times\big(\|X_{t',T}\|_2+\|X_{t'+h}^{(t'/T)}\|_2+\|X_{t'+k',T}\|_2+\|X_{t'+k'+h}^{(t'/T)}\|_2\big)\Big]\\
	\leq&\
	\frac{C}{mT^3}\sum_{i=1}^{T-h}\sum_{t,t'=i}^{(i+m-1)\wedge(T-h)}\frac{1}{\tilde{n}_{t,h} \tilde{n}_{t',h}} \sum_{k=\ubar{n}_t}^{\bar{n}_{t,h}}\sum_{k'=\ubar{n}_{t'}}^{\bar{n}_{t',h}} (|k|+h+1)(|k'|+h+1)=O\Big(\frac{mn^2}{T^2}\Big),
	\end{align*}
	}which converges to zero by Assumption~\ref{eq:b2}. 
	
	For the term $A_{t,2}$, first observe that, by Jensen's inequality for convex functions, 
	{\small\begin{align*}
	&\ \ex\bigg[\bigg(\sum_{t=i}^{(i+m-1)\wedge(T-h)}A_{t,2}(\tau_1,\tau_2)\bigg)^2\bigg]\\
	\leq&\ 
	m\sum_{t=i}^{(i+m-1)\wedge(T-h)}\frac{1}{\tilde{n}_{t,h}^2}\sum_{k,k'=\ubar{n}_t}^{\bar{n}_{t,h}}\cov\big\{ X_{t+k,T}(\tau_1)X_{t+k+h,T}(\tau_2),X_{t+k',T}(\tau_1)X_{t+k'+h,T}(\tau_2)\big\}.
	\end{align*}
	}By the same arguments as in the proof of Proposition~\ref{prop:cov} and Assumption~\ref{cond:cum}, one can see that the right-hand side of the inequality 
	{\small\begin{multline*}
		\frac{1}{mT}\sum_{i=1}^{T-h} \int_{[0,1]^2} \ex\bigg[\bigg(\sum_{t=i}^{(i+m-1)\wedge(T-h)}A_{t,2}(\tau_1,\tau_2)\bigg)^2\bigg] \diff (\tau_1,\tau_2)\\
	\leq \frac{1}{T}\sum_{i=1}^{T-h} \sum_{t=i}^{(i+m-1)\wedge(T-h)} \frac{1}{\tilde{n}_{t,h}^2} \sum_{k,k'=\ubar{n}_t}^{\bar{n}_{t,h}} \|\cov(X_{t+k,T}\otimes X_{t+k+h,T}, X_{t+k',T} \otimes X_{t+k'+h,T}) \|_{1,2} 
	\end{multline*}	
	}
	is of order $\mathcal{O}(m/n)$. The assertion follows since $m/n=o(1)$ by Assumption~\ref{eq:b2}.
	\end{proof}
	
	\begin{proof}[Proof of Proposition~\ref{prop:test}.]
	The cumulative distribution function of the $(h+2)$nd coordinate of $\bm S$ is continuous by Theorem 7.5 of \cite{DavLif1985}. The assertion under the null hypothesis follows from Lemma~4.1 in \cite{BucKoj17}. 
	Consistency follows from the fact that the bootstrap quantiles are stochastically bounded by Theorem~\ref{theo:boot}, whereas the test statistic diverges by Corollary~\ref{cor:main1}. 
	\end{proof}

\section*{Acknowledgements}
Financial support by the Collaborative Research Center “Statistical modeling of nonlinear
dynamic processes” (SFB 823, Teilprojekt A1, A7 and C1) of the German Research Foundation, by the Ruhr University Research School PLUS, funded by Germany’s Excellence Initiative [DFG GSC 98/3], and by the DAAD (German Academic Exchange Service) is gratefully acknowledged. Parts of this paper were written when Axel Bücher was a postdoctoral researcher at Ruhr-Universität Bochum and while Florian Heinrichs was visiting the Universidad Autónoma de Madrid. The authors would like to thank the institute, and in particular Antonio Cuevas, for its hospitality.

\bibliographystyle{chicago}

\bibliography{bibliography}

\newpage

\thispagestyle{empty}

\begin{center}
%
{\bfseries Supplementary Material on   \\ [0mm] ``DETECTING DEVIATIONS FROM SECOND-ORDER STATIONARITY IN LOCALLY STATIONARY FUNCTIONAL \\ [0mm] TIME SERIES''}
\vspace{.5cm}

{\textsc{Axel Bücher, Holger Dette and Florian Heinrichs}

\blfootnote{\textit{Date:} \today.}

}

\end{center}

\begin{abstract}
This supplementary material contains the additional proofs for the main paper. In Appendix~\ref{app:sectest}, we provide the remaining proofs for the results in Sections~\ref{subsec:amoc} and \ref{subsec:choiceM}. Proofs related to Section~\ref{sec:examples} are provided in Appendix~\ref{app:proofex}. Finally, some auxiliary results are collected in Appendix~\ref{app:aux34}.
\end{abstract}

\appendix

\section{Proofs for Section~\ref{subsec:amoc} and \ref{subsec:choiceM}} \label{app:sectest}

	\begin{proof}[Proof of Proposition~\ref{prop:amoc}.]
	By the definition of $X_{t,T}$, we can rewrite
	\begin{multline}\label{eqOneChangePoint}
	U_T(u,\tau)
	=
	\frac{1}{\sqrt{T}}\bigg(\sum_{t=1}^{\lfloor uT\rfloor}Y_{t,T}(\tau)-u\sum_{t=1}^{T} Y_{t,T}(\tau)\bigg)-\frac{uT-\lfloor uT \rfloor}{\sqrt{T}}\mu_2(\tau) \\
	+\frac{1}{\sqrt{T}}\Big( (1-u)\lfloor\lambda T\rfloor \id(\lambda\leq u)+(\lfloor uT\rfloor - u\lfloor \lambda T\rfloor)\id(\lambda>u)\Big)\big\{ \mu_1(\tau)-\mu_2(\tau)\big\}.
	\end{multline}
	By 
	Corollary~\ref{cor:main1},
	\[ \frac{1}{\sqrt{T}}\bigg(\sum_{t=1}^{\lfloor uT\rfloor}Y_{t,T}(\tau)-u\sum_{t=1}^{T} Y_{t,T}(\tau)\bigg) \]
	converges to a centred Gaussian process $\tilde G$.	
	In particular, the norm $\|\cdot\|_{2,2}$ of the previous display is $\mathcal{O}_p(1)$. The norm of the second summand in \eqref{eqOneChangePoint} is of order $\mathcal{O}(T^{-1/2})$ and the norm of the last summand diverges to infinity as $T$ tends to infinity. Thus, $\|U_T\|_{2,2}\to\infty$ in probability, and therefore the test statistic $\Sc_T^{\scs (m)}$ diverges to infinity in probability.
	
	In the proof of Theorem \ref{theo:boot}, we have seen that $\|\hat{B}_{T}^{\scs (k)}-\tilde{B}_T^{\scs (k)}\|_{2,3}=o_\Prob(1)$ as $T\to\infty$ under the assumption of local stationarity, where $\tilde{B}_T^{\scs(k)}$ is defined in \eqref{eq:tildeb}. The same result can be shown with similar arguments in the setting of a change point. Further,
	\begin{align*}
	&\hspace{-.5cm}\tilde{B}_T^{(k)}(u,\tau)-u\tilde{B}_T^{(k)}(1,\tau)\\
	&= \frac{1}{\sqrt{T}}\bigg\{ \sum_{i=1}^{\lfloor uT\rfloor}\frac{R_i^{(k)}}{\sqrt{m}}\sum_{t=i}^{(i+m-1)\wedge T}[X_{t,T}(\tau)-\ex X_{t,T}(\tau)] \\
	&\hspace{5cm} -u\sum_{i=1}^{T}\frac{R_i^{(k)}}{\sqrt{m}}\sum_{t=i}^{(i+m-1)\wedge T}[X_{t,T}(\tau)-\ex X_{t,T}(\tau)]\bigg \} \\
	&= \frac{1}{\sqrt{T}}\bigg(\sum_{i=1}^{\lfloor uT\rfloor}\frac{R_i^{(k)}}{\sqrt{m}}\sum_{t=i}^{(i+m-1)\wedge T}Y_{t,T}(\tau)-u\sum_{i=1}^{T}\frac{R_i^{(k)}}{\sqrt{m}}\sum_{t=i}^{(i+m-1)\wedge T}Y_{t,T}(\tau)\bigg),
	\end{align*}
	where the right-hand side converges according to Theorem \ref{bootstrapThm} to the process $\tilde B$ as well. Thus, $\|\hat{G}_T^{\scs (k)}\|_{2,2}=\mathcal{O}_p(1)$.
\end{proof}

\begin{proof}[Proof of Lemma~\ref{lem:IMSE}]
	We begin by proving the formula for the bias. We have
	{\small\begin{align}\label{eq:ex:sigma_T}
	\begin{split}\ex[\tilde{\sigma}_T(\tau,\phi)]
	&= \frac{1}{T} \sum_{i=1}^{T} \frac{1}{m} \sum_{t,t'=i}^{(i+m-1)\wedge T} \cov\big(X_{t,T}(\tau),X_{t',T}(\phi)\big)\\
	&= \frac{1}{T} \sum_{i=1}^{T} \frac{1}{m} \sum_{t,t'=i}^{(i+m-1)\wedge T} \cov\big(X_t^{(i/T)}(\tau),X_{t'}^{(i/T)}(\phi)\big)+\mathcal{O}(T^{-1})\\
	&= \frac{1}{T} \sum_{i=1}^{T} \frac{1}{m} \sum_{k=-m+1}^{m-1} (m-|k|)\cov\big(X_0^{(i/T)}(\tau),X_{k}^{(i/T)}(\phi)\big)+\mathcal{O}(T^{-1})\\
	&= \sum_{k=-m+1}^{m-1} \int_0^1 \cov\big(X_0^{(w)}(\tau),X_k^{(w)}(\phi)\big) \diff w\\
	&\phantom{=}  \hspace{2cm} - \frac{1}{T} \sum_{i=1}^{T} \frac{1}{m} \sum_{k=-m+1}^{m-1} |k|\cov\big(X_0^{(i/T)}(\tau),X_{k}^{(i/T)}(\phi)\big) + \mathcal{O}(T^{-1}).
	\end{split}
	\end{align}}Further, by Lemma~3.11 in \citesuppl{DehPhi02}, we can rewrite
	{\small\begin{align*}
	\sigma_c(\tau,\phi)
	&= \sum_{k=-\infty}^{\infty} \int_0^1 \cov(X_0^{(w)}(\tau),X_k^{(w)}(\phi))\diff w\\
	&= \sum_{k=-m+1}^{m-1} \int_0^1 \cov(X_0^{(w)}(\tau),X_k^{(w)}(\phi))\diff w
	+ \sum_{|k|\geq m} \int_0^1 \cov(X_0^{(w)}(\tau),X_k^{(w)}(\phi))\diff w\\
&= \sum_{k=-m+1}^{m-1} \int_0^1 \cov(X_0^{(w)}(\tau),X_k^{(w)}(\phi))\diff w+ \mathcal{O}(a^{rm}),
	\end{align*}}for some $0<r<1$. By the previous display, Equation \eqref{eq:ex:sigma_T} and since $\mathcal{O}(a^{rm})+\mathcal{O}(T^{-1})=o(m^{-2})$, we obtain that
	{\small\begin{align*}
	&\int_{[0,1]^2} \big(\ex[\tilde{\sigma}_T(\tau,\phi)]-\sigma_c(\tau,\phi)\big)^2 \diff(\tau,\phi)\\
	&= \int_{[0,1]^2} \bigg\{\frac{1}{T} \sum_{i=1}^{T} \frac{1}{m} \sum_{k=-m+1}^{m-1} |k| \cov\big(X_0^{(i/T)}(\tau),X_k^{(i/T)}(\phi)\big)+o(m^{-2}) \bigg\}^2 \diff(\tau,\phi)\\
	&= \int_{[0,1]^2} \bigg\{\sum_{k=-m+1}^{m-1} \frac{|k|}{m} \int_0^1 \cov\big(X_0^{(w)}(\tau),X_k^{(w)}(\phi)\big)\diff w \bigg\}^2+o(m^{-2}) \diff(\tau,\phi)\\
	&= \frac{1}{m^2} \bigg\| \sum_{k=-\infty}^{\infty} |k| \int_0^1 \cov(X_0^{(w)},X_k^{(w)}) \diff w \bigg\|_{2,2}^2 + o(m^{-2})
	\end{align*}}as asserted.

Next, consider the formula for the variance.
	Observe that $\var\big(\tilde{\sigma}_T(\tau,\phi)\big) = \ex[\tilde{\sigma}_T^2(\tau,\phi)]-(\ex \tilde{\sigma}_T(\tau,\phi))^2$. By Theorem 2.3.2 of \citesuppl{Brillinger1965}, we can rewrite
	\begin{multline*}
	\ex \tilde{\sigma}_T^2(\tau,\phi)
	= \frac{1}{T^2} \sum_{i,i'=1}^{T} \frac{1}{m^2} \sum_{t_1,t_2=i}^{(i+m-1)\wedge T} \sum_{t_3,t_4=i'}^{(i'+m-1)\wedge T}
	\kappa_{t_1,t_2,t_3,t_4}(\tau,\phi) + \kappa_{t_1,t_2}(\tau,\phi)\kappa_{t_3,t_4}(\tau,\phi)\\
	+ \kappa_{t_1,t_3}(\tau,\tau)\kappa_{t_2,t_4}(\phi,\phi)
	+ \kappa_{t_1,t_4}(\tau,\phi)\kappa_{t_2,t_3}(\phi,\tau),
	\end{multline*}
	where $\kappa_{t_1,t_2,t_3,t_4}(\tau,\phi)=\cum\big(X_{t_1,T}(\tau),X_{t_2,T}(\phi),X_{t_3,T}(\tau),X_{t_4,T}(\phi)\big)$ and $\kappa_{t_1,t_2}(\tau,\phi)=\cov\big(X_{t_1,T}(\tau),X_{t_2,T}(\phi)\big)$ for any $t_1,t_2,t_3,t_4 \in \{1,\cdots, T\}$ and $\tau,\phi \in[0,1]$. With this notation, we can further rewrite
	\[ (\ex \tilde{\sigma}_T(\tau,\phi))^2 = \frac{1}{T^2} \sum_{i,i'=1}^{T} \frac{1}{m^2} \sum_{t_1,t_2=i}^{(i+m-1)\wedge T} \sum_{t_3,t_4=i'}^{(i'+m-1)\wedge T} \kappa_{t_1,t_2}(\tau,\phi)\kappa_{t_3,t_4}(\tau,\phi), \]
	thus, 
	\begin{multline*}
	\var\big(\tilde{\sigma}_T(\tau,\phi)\big)
	= \frac{1}{T^2} \sum_{i,i'=1}^{T} \frac{1}{m^2} \sum_{t_1,t_2=i}^{(i+m-1)\wedge T} \sum_{t_3,t_4=i'}^{(i'+m-1)\wedge T}
	\kappa_{t_1,t_2,t_3,t_4}(\tau,\phi)\\
	+ \kappa_{t_1,t_3}(\tau,\tau)\kappa_{t_2,t_4}(\phi,\phi)
	+ \kappa_{t_1,t_4}(\tau,\phi)\kappa_{t_2,t_3}(\phi,\tau).
	\end{multline*}
	In the following, we investigate the sums $B_1,B_2$ and $B_3$ over the three inner summands separately. 
	
	First, observe that by the strong mixing condition and Theorem  3 in \citesuppl{StatuleviciusJakimavicius1988}, the sum over the cumulants $\kappa_{t_1,t_2,t_3,t_4}(\tau,\phi)$ vanishes with rate $m^2T^{-2}$, i.\,e.,  
	\[ B_1 = \frac{1}{T^2} \sum_{i,i'=1}^{T} \frac{1}{m^2} \sum_{t_1,t_2=i}^{(i+m-1)\wedge T} \sum_{t_3,t_4=i'}^{(i'+m-1)\wedge T}
	\kappa_{t_1,t_2,t_3,t_4}(\tau,\phi) = \mathcal{O}(m^2T^{-2}) \]
To investigate $B_2$, we split the sum into $B_2=B_{2,1}+B_{2,2}+B_{2,3}$, where 
	 \begin{align*} &B_{2,1}(\tau,\phi) = \frac{1}{T^2} \sum_{i=1}^{T} \sum_{|i-i'|\leq m} \frac{1}{m^2} \sum_{t_1,t_2=i}^{(i+m-1)\wedge T} \sum_{t_3,t_4=i'}^{(i'+m-1)\wedge T}
	 \kappa_{t_1,t_3}(\tau,\tau)\kappa_{t_2,t_4}(\phi,\phi),\\ 
	 &B_{2,2}(\tau,\phi) = \frac{1}{T^2} \sum_{i=m+2}^{T} \sum_{i'=1}^{i-m-1} \frac{1}{m^2} \sum_{t_1,t_2=i}^{(i+m-1)\wedge T} \sum_{t_3,t_4=i'}^{(i'+m-1)\wedge T}
	 \kappa_{t_1,t_3}(\tau,\tau)\kappa_{t_2,t_4}(\phi,\phi)\\
	 \text{and}\hspace{0.7cm}&\\
	 &B_{2,3}(\tau,\phi) = \frac{1}{T^2} \sum_{i=1}^{T-m-1} \sum_{i'=i+m+1}^{T} \frac{1}{m^2} \sum_{t_1,t_2=i}^{(i+m-1)\wedge T} \sum_{t_3,t_4=i'}^{(i'+m-1)\wedge T}
	 \kappa_{t_1,t_3}(\tau,\tau)\kappa_{t_2,t_4}(\phi,\phi). \end{align*}
	 
	In the following, we will see that both $B_{2,2}$ and $B_{2,3}$ are negligible, while $B_{2,1}$ contributes to the claimed limit expression. The covariances $\kappa_{t_1,t_3}$ and $\kappa_{t_2,t_4}$ can be bounded by $C \alpha^r(|t_1-t_3|)\leq C a^{r|t_1-t_3|}$ and $C\alpha^r(|t_2-t_4|)\leq Ca^{r|t_2-t_4|}$, for some constants $C>0$ and $0<r<1$, respectively. Therefore,
	{\small\begin{align*}
	B_{2,2}(\tau,\phi)  &\leq \frac{C}{T^2} \sum_{i=m+2}^{T} \sum_{i'=1}^{i-m-1} \frac{1}{m^2} \sum_{t_1,t_2=i}^{i+m-1} \sum_{t_3,t_4=i'}^{i'+m-1} a^{r|t_1-t_3|} a^{r|t_2-t_4|}\\
	&= \frac{C}{T^2} \sum_{i=m+2}^{T} \sum_{i'=1}^{i-m-1} \frac{1}{m^2} \sum_{t_1,\cdots, t_4=1}^{m} a^{r(i-i'+t_1-t_3)} a^{r(i-i'+t_2-t_4)}\\
	&= \frac{C}{T^2} \sum_{i=m+2}^{T} \sum_{i'=m+1}^{i-1} \frac{1}{m^2} \sum_{t_1,\cdots, t_4=1}^{m} a^{r(i-i'+m+t_1-t_3)} a^{r(i-i'+m+t_2-t_4)}\\
	&=\frac{C}{T^2} \sum_{i=m+2}^{T} \sum_{i'=m+1}^{i-1} a^{2r(i-i')} \frac{1}{m^2} \sum_{t_1,\cdots, t_4=1}^{m} a^{r(m+t_1-t_3)} a^{r(m+t_2-t_4)} = \mathcal{O}(T^{-1}).
	\end{align*}}	
	Analogously, $B_{2,3}(\tau,\phi)=\mathcal{O}(T^{-1})$. For $B_{2,1}$ observe that 
	\begin{multline*}
	\int_{[0,1]^2} B_{2,1}(\tau,\phi) \diff(\tau,\phi)
	=\int_{[0,1]^2} \frac{1}{T^2} \sum_{i=1}^{T} \sum_{|i-i'|\leq m} \frac{1}{m^2} \sum_{t_1,\cdots, t_4=1}^{m}
	\cov\big(X_{t_1+i}^{(i/T)}(\tau),X_{t_3+i'}^{(i/T)}(\tau) \big)\\
	\times\cov\big(X_{t_2+i}^{(i/T)}(\phi),X_{t_4+i'}^{(i/T)}(\phi) \big)\diff(\tau,\phi)+\mathcal{O}(m^4 T^{-2}).
	\end{multline*}
	The inner sums of the integrand in the above display, can be rewritten as
	{\small\begin{align}\label{eq:m:cov}
	&\phantom{=\,}\sum_{i'=i-m}^{i+m} \frac{1}{m^2} \sum_{t_1,\cdots, t_4=1}^{m}
	\cov\big(X_{t_1+i}^{(i/T)}(\tau),X_{t_3+i'}^{(i/T)}(\tau) \big)\cov\big(X_{t_2+i}^{(i/T)}(\phi),X_{t_4+i'}^{(i/T)}(\phi) \big) \nonumber\\
	&= \sum_{i'=-m}^{m} \frac{1}{m^2} \sum_{t_1,\cdots, t_4=1}^{m}
	\cov\big(X_{0}^{(i/T)}(\tau),X_{t_3-t_1+i'}^{(i/T)}(\tau) \big)\cov\big(X_{0}^{(i/T)}(\phi),X_{t_4-t_2+i'}^{(i/T)}(\phi) \big) \nonumber \\
	&= \sum_{i',k_1,k_2=-m}^{m} \Big(1-\frac{|k_1|}{m}\Big)\Big(1-\frac{|k_2|}{m}\Big)
	\cov\big(X_{0}^{(i/T)}(\tau),X_{k_1+i'}^{(i/T)}(\tau) \big)\cov\big(X_{0}^{(i/T)}(\phi),X_{k_2+i'}^{(i/T)}(\phi) \big) \nonumber \\
	&= \sum_{i'=-m}^{m} \sum_{k_1,k_2=-m}^{m} \Big(1-\frac{|k_1|}{m}\Big)\Big(1-\frac{|k_2|}{m}\Big)
	\gamma_{k_1+i'}(i/T|\tau)\gamma_{k_2+i'}(i/T|\phi),
	\end{align}}where $\gamma_k(u|x)=\cov\big(X_0^{(u)}(x),X_k^{(u)}(x)\big)$, for any $k\in\N$ and $u,x\in[0,1]$. Let $\ell_m$ be an increasing sequence in $\N$, with $\ell_m\le m$, $\ell_m^2/m\to 0$ and $m^3a^{r\ell_m}\to 0$ as $m\to\infty$, for some $0<r<1$, as $m$ tends to infinity; for instance, $\ell_m=m^{1/3}$. By the strong mixing property and Lemma~3.11 in \citesuppl{DehPhi02}, we can rewrite the right-hand side of \eqref{eq:m:cov} as
	{\small\begin{align}\label{eq:m:cov2}
	&\phantom{ =\, }\sum_{i'=-m}^{m} \sum_{k_1,k_2=-i'-\ell_m}^{-i'+\ell_m} \Big(1-\frac{|k_1|}{m}\Big)\Big(1-\frac{|k_2|}{m}\Big)
	\gamma_{k_1+i'}(i/T|\tau)\gamma_{k_2+i'}(i/T|\phi)+ \mathcal{O}(m^3a^{r\ell_m})\nonumber\\
	&=\sum_{i'=-m}^{m} \sum_{k_1,k_2=-\ell_m}^{\ell_m} \Big(1-\frac{|k_1-i'|}{m}\Big)\Big(1-\frac{|k_2-i'|}{m}\Big)
	\gamma_{k_1}(i/T|\tau)\gamma_{k_2}(i/T|\phi)+ \mathcal{O}(m^3a^{r\ell_m})\nonumber\\
	&=\sum_{k_1,k_2=-\ell_m}^{\ell_m} \gamma_{k_1}(i/T|\tau)\gamma_{k_2}(i/T|\phi) \sum_{i'=-m}^{m}  \Big(1-\frac{|k_1-i'|}{m}\Big)\Big(1-\frac{|k_2-i'|}{m}\Big)
	+ \mathcal{O}(m^3a^{r\ell_m}). 
	\end{align} 
	}A tedious but straight-forward calculation based on splitting the next sum into the three cases $i'< k_1\wedge k_2$, $i' = k_1\wedge k_2 , \dots, k_1\vee k_2$  and $i'>k_1\vee k_2$ implies that 
	\begin{align*}
	\sum_{i'=-m}^{m}  \Big(1-\frac{|k_1\wedge k_2-i'|}{m}\Big)\Big(1-\frac{|k_1\vee k_2-i'|}{m}\Big)=\frac{2}{3}m + \mathcal{O}(\ell_m^2/m).
	\end{align*}
	Plugging this into \eqref{eq:m:cov2} leads, by the dominated convergence theorem and Lipschitz continuity of $\gamma_k(u|\tau)$ in $u$, to 	
	{\small\begin{align*}
	\int_{[0,1]^2} B_{2}(\tau,\phi)\diff(\tau,\phi)
	&= \int_{[0,1]^2} \frac{1}{T^2} \sum_{i=1}^{T} \sum_{k_1,k_2=-\ell_m}^{\ell_m} \frac{2}{3}m \gamma_{k_1}(i/T|\tau)\gamma_{k_2}(i/T|\phi)\diff(\tau,\phi) + o(m/T)\\ 
	&= \frac{m}{T} \frac{2}{3}\int_0^1 \bigg(\sum_{k=-\infty}^{\infty}  \int_0^1  \cov\big(X_0^{(w)}(\tau),X_{k}^{(w)}(\tau)\big) \diff \tau\bigg)^2  \diff w  + o(m/T),
	\end{align*}}since, by the strong mixing property, $\sum_{|k|>\ell_m} \int_0^1 \gamma_k(u|\tau)\diff \tau$ is of order $\mathcal{O}(a^{r\ell_m}).$
	By similar arguments, we have 
	{\small\[\int_{[0,1]^2} B_{3}(\tau,\phi)\diff(\tau,\phi)= \frac{m}{T} \frac{2}{3}\int_0^1 \bigg\| \sum_{k=-\infty}^{\infty} \cov(X_0^{(w)},X_k^{(w)}) \bigg\|_{2,2}^2 \diff w  + o(m/T)\]}
	and the theorem's statement follows.
\end{proof}

\section{Proofs for Section~\ref{sec:examples}}
	\label{app:proofex}

\begin{proof}[Proof of Lemma~\ref{lem:ar1}] Note that $\eps_{t,T}=\sigma(t/T)\tilde{\eps}_t = \sigma(0)\tilde{\eps}_t = \eps_{t}^{\scs (0)}$, for any $t\leq 0$. Further, $\ex \eps_{t}^{\scs (u)}=0$ and $\ex \|\eps_{t}^{\scs (u)}\|_2^2= \sigma^2(u)\ex \|\tilde \eps_{t}\|_2^2$, which is strictly greater than zero and finite.
 
\textit{Proof of (i):} Similar to the proof of Theorem 3.1 of \citesuppl{Bos00}, yet, with a random operator, we have
\begin{align*}
	 \|A_u^j(\eps_{u,t-j})\|_{2,\Omega \times [0,1]}
	 = \ex  \|A_u^j(\eps_{u,t-j})\|_{2}
	 \leq \ex [ \|A_u^j\|_\mathcal{L} \|\eps_{u,t-j}\|_2 ]
	 \leq C q^j \ex \|\eps_{u,t-j}\|_2 \leq C q^j
\end{align*}
since
\begin{align*}
	\|A_u^{j}\|_\mathcal{L}\leq \|A_u^{j}\|_\mathcal{S} \leq \|A_u\|_\mathcal{S}^{j}  \leq q^{j}
\end{align*}
by Equation (1.55) of \citesuppl{Bos00} and submultiplicativity of the Hilbert-Schmidt norm. 
We can now follow the proof of Theorem 3.1 of \citesuppl{Bos00} to deduce the assertions in~(i).


\textit{Proof of (ii):} Similarly as before, we have
\begin{align*}
	\Big\| \prod_{i=0}^{j-1} A_{\tfrac{t-i}{T}}(\eps_{t-j,T})\Big\|_{2,\Omega\times[0,1]}^2
	=\ex \Big\| \prod_{i=0}^{j-1} A_{\tfrac{t-i}{T}}(\eps_{t-j,T})\Big\|_2^2
	&\leq \ex \bigg[ \Big\| \prod_{i=0}^{j-1} A_{\tfrac{t-i}{T}}\Big\|_\mathcal{L}^2 \|\eps_{t-j,T}\|_2^2 \bigg]\\
	&\leq q^{2j} \ex \|\eps_{t-j,T}\|_2^2 \leq C q^{2j}
	\end{align*}
where, by convention, $\prod_{i=0}^{-1} A_{(t-i)/T}=\ident_{L^2([0,1])}$. Therefore, for $1\le m \le m'$,
{\small\begin{align*}
\Delta_m^{m'} = \Big\| \sum_{j=m}^{m'}\prod_{i=0}^{j-1}A_{\tfrac{t-i}{T}}(\eps_{t-j,T})\Big\|_{2,\Omega\times[0,1]}^2
\leq \Big(\sum_{j=m}^{m'}\Big\| \prod_{i=0}^{j-1}A_{\tfrac{t-i}{T}}(\eps_{t-j,T})\Big\|_{2,\Omega\times[0,1]}\Big)^2 \leq C \Big(\sum_{j=m}^{m'}q^j\Big)^2, 
\end{align*}
}and the right-hand side converges to zero as $m$ and $m'$ tend to infinity. As a consequence,   $\tilde Y_{t,T} := \sum_{j=0}^{\infty} \prod_{i=0}^{j-1} A_{(t-i)/{T}} (\eps_{t-j,T})$ converges in $L^2(\Omega \times [0,1],\pr \otimes \lambda)$ by the Cauchy criterion. Further, 
\begin{align}\label{eq:ar:lem}
\ex\|\tilde Y_{t,T}\|_2^2 
&\leq\ex\Big[\Big(\sum_{j=0}^{\infty}\big\|\prod_{i=0}^{j-1}A_{\tfrac{t-i}{T}}\big\|_\mathcal{L}\|\eps_{t-j,T}\|_2\Big)^2\Big] \\
&\leq \ex\Big[\Big(\sum_{j=0}^{\infty}q^j\|\eps_{t-j,T}\|_2\Big)^2\Big] \nonumber\\
&\le \Big(\frac1{1-q}\Big)^2 \Exp[\| \tilde \eps_0\|_2^2] \sup_{u\in[0,1]} \sigma^2(u) < \infty. \nonumber
\end{align} 
Hence $\sum_{j=0}^{\infty}\big\|\prod_{i=0}^{j-1}A_{(t-i)/T}\big\|_\mathcal{L}\|\eps_{t-j,T}\|_2<\infty$ almost surely, which implies almost sure convergence in $L^2([0,1])$ of the series defining $\tilde Y_{t,T}$ by the Riesz-Fisher theorem.

Finally, we have
\begin{align*}
\tilde Y_{t,T}-A_{t/T}(\tilde Y_{t-1,T})=\sum_{j=0}^{\infty}\prod_{i=0}^{j-1}A_{\tfrac{t-i}{T}} (\eps_{t-j,T})- A_{t/T}\Big(\sum_{j=0}^{\infty}\prod_{i=0}^{j-1}A_{\tfrac{t-1-i}{T}} (\eps_{t-1-j,T})\Big)=\eps_{t,T},
\end{align*}
whence $\tilde Y_{t,T}$ is a solution of \eqref{eq:ar} satisfying $\sup_{t\in\Z,T\in\N}\ex \|\tilde Y_{t,T}\|_2^2<\infty$  by \eqref{eq:ar:lem} and, as we will show below, is locally stationary of order $\rho=2$.  

Conversely, let $Z_{t,T}$ be a locally stationary solution of \eqref{eq:ar} of order $\rho=2$ which satisfies $\sup_{t\in\Z,T\in\N}\ex \|Z_{t,T}\|_2^2<\infty$. By induction, we have
\[ Z_{t,T}= \sum_{j=0}^{n} \prod_{i=0}^{j-1} A_{\tfrac{t-i}{T}}(\eps_{t-j,T}) + \prod_{i=0}^{n} A_{\tfrac{t-i}{T}}(Z_{t-n-1,T}). \]
Thus, 
		
\begin{align*}
\ex \Big\|Z_{t,T}- \sum_{j=0}^{n} \prod_{i=0}^{j-1} A_{\tfrac{t-i}{T}}(\eps_{t-j,T})\Big\|_2^2
&= \ex \Big\| \prod_{i=0}^{n} A_{\tfrac{t-i}{T}}(Z_{t-n-1,T}) \Big\|_2^2\\
&\leq \ex \bigg[ \prod_{i=0}^{n} \| A_{\tfrac{t-i}{T}}\|_\mathcal{S}^2 \|Z_{t-n-1,T}\|_2^2 \bigg]\\
&\leq q^{2(n+1)} \ex \|Z_{t-n-1,T}\|_2^2,
\end{align*}
which converges to zero as $n$ tends to infinity.

It remains to show that $\tilde Y_{t,T}$ is locally stationary of order $\rho=2$ with approximating family $\{(Y_t^{\scs (u)})_{t\in\Z}: u \in [0,1] \}$. Note that
\begin{align*}
\prod_{i=1}^n B_i - \prod_{i=1}^n C_i 
&= 
\sum_{m=1}^n \Big(\prod_{k=1}^{m-1} B_k \Big) (B_m - C_m) \Big(\prod_{k=m+1}^{n} C_k \Big) \\
&=
\sum_{m=1}^n \Big(\prod_{k=1}^{n-m} B_k \Big) (B_{n-m+1} - C_{n-m+1}) \Big(\prod_{k=n-m+2}^{n} C_k \Big)
\end{align*}
for all $n \in \N, B_i, C_i \in \Lc$, the empty product being defined as the identity on $L^2([0,1])$.
As a consequence
 {\small
 \begin{align*}
 Y_{t,T}-Y_t^{(u)}
 &= 
 \sum_{j=0}^{\infty} \prod_{i=0}^{j-1} A_{\tfrac{t-i}{T}}(\eps_{t-j,T}) - \sum_{j=0}^{\infty} A_u^j(\eps_{u,t-j})\\
 &= 
 \sum_{j=0}^{\infty} \prod_{i=0}^{j-1} A_{\tfrac{t-i}{T}}(\eps_{t-j,T}) - \prod_{i=0}^{j-1} A_{\tfrac{t-i}{T}}(\eps_{u,t-j}) + \prod_{i=1}^{j} A_{\tfrac{t-i+1}{T}}(\eps_{u,t-j}) - \prod_{i=1}^{j} A_u(\eps_{u,t-j})\\
 &= 
 \sum_{j=0}^{\infty} \prod_{i=0}^{j-1} A_{\tfrac{t-i}{T}}(\eps_{t-j,T}-\eps_{u,t-j}) + \sum_{m=1}^{j}\Big(\prod_{i=0}^{j-1-m} A_{\tfrac{t-i}{T}}\Big) (A_{\tfrac{t-j+m}{T}}-A_u)A_u^{m-1}(\eps_{u,t-j}).
 \end{align*}
}Since $\|A_u\|_\mathcal{S}\leq q$ for any $u\in[0,1]$ (with probability one) and $\|Ax\|_2 \leq \|A\|_\mathcal{S}\|x\|_2$ for any $A\in\mathcal{L},x \in L^2([0,1])$, it follows that
{\small
\begin{align*}
\|Y_{t,T}-Y_t^{(u)}\|_2
&\leq 
\sum_{j=0}^{\infty}\Big( q^j |\sigma\big(\tfrac{t-j}{T}\big)-\sigma(u)| \|\tilde \eps_{t-j}\|_2 + q^j\sum_{m=1}^{j} |a\big(\tfrac{t-j+m}{T}\big)-a(u)|\|\tilde \eps_{t-j} \|_2\Big)\\
&\leq 
\sum_{j=0}^{\infty}\Big( q^j \big(\big|\tfrac{t}{T}-u\big|+\tfrac{j}{T}\big)\|\tilde \eps_{t-j}\|_2 + q^j\sum_{m=1}^{j} \big(\big|\tfrac{t}{T}-u\big|+\tfrac{j-m}{T}\big) \|\tilde \eps_{t-j} \|_2\Big)\\
&\leq 
C \sum_{j=0}^{\infty}q^j\|\tilde \eps_{t-j}\|_2 \big((j+1)\big|\tfrac{t}{T}-u\big|+\tfrac{j^2}{T}\big) \\
&\leq 
C \big(\big|\tfrac{t}{T}-u\big|+\tfrac{1}{T}\big) \sum_{j=0}^{\infty}q^j (j+1)^2\|\tilde \eps_{t-j}\|_2.
\end{align*}
}The assertion finally follows from the fact that $P_{t,T}^{\scs (u)}= \sum_{j=0}^{\infty}q^j (j+1)^2\|\tilde \eps_{t-j}\|_2$ has a finite second moment.
\end{proof}

	\section{Auxiliary results for the proofs in Section~\ref{sec:dev1}} 
	\label{app:aux34}

\begin{lemma} \label{lem:nuc}
$C_\Bb$ is a symmetric, positive trace class operator. 
As a consequence (Theorem 1.2.5 of \citealpsuppl{ManRha2004}),  $\Bb$ is a Gaussian random variable in $\Hc_{H+2}$.
\end{lemma}

\begin{proof}
To ensure readability, we will denote the scalar product of $L^2([0,1]^p)$ by $\langle\cdot,\cdot\rangle_p,p=2,3$ and consider the case $H=0$ only. The arguments for $H\geq 1$ are the same, yet notationally more involved.
	
	\textit{Symmetry:} Let $(g_1,f_1)$ and $(g_2,f_2)$ be elements in $\Hc_{2}$. Then,  
	\begin{align}\label{eq:symOp}
	\langle C_\Bb(g_1,f_1),(g_2,f_2)\rangle 
	&= \big\langle \langle r^{(m)}(\ast,\cdot),g_1(\cdot)\rangle_2,g_2(\ast)\big\rangle_2 \nonumber
	+ \big\langle \langle r_0^{(m,c)}(\ast,\cdot),f_1(\cdot)\rangle_3,g_2(\ast)\big\rangle_2\\ 
	&+ \big\langle \langle r_0^{(m,c)}(\ast,\cdot),g_1(\ast)\rangle_2,f_2(\cdot)\big\rangle_3
	+ \big\langle \langle r_{0,0}^{(c)}(\ast,\cdot),f_1(\cdot)\rangle_3,f_2(\ast)\big\rangle_3. 
	\end{align}
	Now, writing $u=(u_1,u_2)$ and $v=(v_1, v_2)$,
	\begin{align*}
	&\big\langle \langle r^{(m)}(\ast,\cdot),g_1(\cdot)\rangle_2,g_2(\ast)\big\rangle_2\\
	&= \int_{[0,1]^4} \sum_{k=-\infty}^{\infty}\int_0^{u_1\wedge v_1} \cov\big(X_0^{(w)}(u_2),X_k^{(w)}(v_2)\big) \diff w g_1(v)g_2(u)\diff(u,v)\\
	&= \int_{[0,1]^4} \sum_{k=-\infty}^{\infty}\int_0^{u_1\wedge v_1} \cov\big(X_0^{(w)}(u_2),X_k^{(w)}(v_2)\big) \diff w g_1(u)g_2(v)\diff(u,v)\\
	&= \big\langle g_1(\ast), \langle r^{(m)}(\ast,\cdot),g_2(\cdot)\rangle_2\big\rangle_2
	\end{align*}
	and similarly $\big\langle \langle r_{0,0}^{(c)}(\ast,\cdot),f_1(\cdot)\rangle_3,f_2(\ast)\big\rangle_3=\big\langle f_1(\ast), \langle r_{0,0}^{(c)}(\ast,\cdot),f_2(\cdot)\rangle_3\big\rangle_3$. Further, writing $u=(u_1,u_2)$ and $v=(v_1, v_2, v_3)$, we have
	\begin{align*}
	&\big\langle \langle r_0^{(m,c)}(\ast,\cdot),f_1(\cdot)\rangle_3,g_2(\ast)\big\rangle_2+ \big\langle \langle r_0^{(m,c)}(\ast,\cdot),g_1(\ast)\rangle_2,f_2(\cdot)\big\rangle_3\\
	&= \int_{[0,1]^5} \sum_{k=-\infty}^\infty \int_0^{u_1\wedge v_1} \cov\big(X_0^{(w)}(u_2),X_k^{(w)}(v_2)X_{k+h}^{(w)}(v_3)\big)\\
	&\hspace{7cm}\times\big(f_1(v)g_2(u)+f_2(v)g_1(u)\big)\diff w\diff(u,v)\\
	&=\int_{[0,1]^5} \sum_{k=-\infty}^\infty \int_0^{u_1\wedge v_1} \cov\big(X_0^{(w)}(u_2),X_k^{(w)}(v_2)X_{k+h}^{(w)}(v_3)\big)\\
	&\hspace{7cm}\times\big(g_1(u)f_2(v)+g_2(u)f_1(v)\big)\diff w\diff(u,v)\\
	&= \big\langle g_1(\ast), \langle r_0^{(m,c)}(\ast,\cdot),f_2(\cdot)\rangle_3\big\rangle_2+ \big\langle f_1(\cdot), \langle r_0^{(m,c)}(\ast,\cdot),g_2(\ast)\rangle_2\big\rangle_3.
	\end{align*}
	Thus, by \eqref{eq:symOp}, it follows $\langle C_\Bb(g_1,f_1),(g_2,f_2)\rangle = \langle (g_1,f_2),C_\Bb(g_2,f_2)\rangle$.
	
	\textit{Positivity:}
	The positivity of $C_\Bb$ can be seen by similar elementary calculations. Let $(g,f)$ be in $\Hc_{2}$ and observe that
	\begin{align}\label{eq:posOp}
	&\langle C_\Bb(g,f),(g,f)\rangle \nonumber \\
	&= \int_0^1 \sum_{k=-\infty}^{\infty}\bigg\{
		\int_{[0,1]^4} \id(w\leq u_1\wedge v_1) \cov\big(X_0^{(w)}(u_2),X_k^{(w)}(v_2)\big)g(u)g(v)\diff(u,v)\nonumber  \\
		&\phantom{=}+ \int_{[0,1]^5} \id(w\leq u_1\wedge v_1) \cov\big(X_0^{(w)}(u_2),X_k^{(w)}(v_2)X_{k+h}^{(w)}(v_3)\big)\big(f(v)g(u)+g(u)f(v)\big)\diff(u,v)\nonumber  \\
		&\phantom{=}+ \int_{[0,1]^6} \id(w\leq u_1\wedge v_1) \cov\big(X_0^{(w)}(u_2)X_h^{(w)}(u_3),X_k^{(w)}(v_2)X_{k+h}^{(w)}(v_3)\big)f(u)f(v)\diff(u,v) \bigg\} \diff w\nonumber  \\
	&= \int_0^1 \sum_{k=-\infty}^{\infty} \ex[Y_0(w)Y_k(w)+2Y_0(w)Z_k(w)+Z_0(w)Z_k(w)]\diff w \nonumber \\
	&= \int_0^1 \sum_{k=-\infty}^{\infty} \ex[Y_0(w)\big(Y_k(w) + Z_k(w)\big)+\big(Y_0(w)+Z_0(w)\big)Z_k(w)]\diff w,
	\end{align}
	where, for $k\in\Z$ and $w\in[0,1]$, 
	\begin{align*}
	Y_k(w)&=\int_{[0,1]^2} \id(w\leq u_1)g(u)(X_k^{(w)}(u_2)-\ex[X_k^{(w)}(u_2)])\diff u\\
	Z_k(w)&=\int_{[0,1]^3} \id(w\leq u_1)f(u)(X_k^{(w)}(u_2)X_{k+h}^{(w)}(u_3)-\ex[X_k^{(w)}(u_2)X_{k+h}^{(w)}(u_3)])\diff u.
	\end{align*}
	As $Y_k$ and $Z_k$ are defined based on a family of stationary processes, we may write $\ex[(Y_0(w)+Z_0(w))Z_k(w)]=\ex[(Y_{-k}(w)+Z_{-k}(w))Z_0(w)]$. As the summation runs over all $k\in\Z$, we can rewrite the right-hand side of \eqref{eq:posOp} as
	\begin{align*}
	&\int_0^1 \sum_{k=-\infty}^{\infty} \ex[Y_0(w)\big(Y_k(w) + Z_k(w)\big)+\big(Y_k(w)+Z_k(w)\big)Z_0(w)]\diff w\\
	&= \int_0^1 \sum_{k=-\infty}^{\infty} \ex[\big(Y_0(w)+Z_0(w)\big)\big(Y_k(w) + Z_k(w)\big)]\diff w\\
	&= \int_0^1 \lim_{n\to\infty} \sum_{k=-n}^{n} \ex[\big(Y_0(w)+Z_0(w)\big)\big(Y_k(w) + Z_k(w)\big)]\diff w,
	\end{align*}
	which is non-negative since
	\begin{align*}
	&\sum_{k=-n}^{n} \ex[\big(Y_0(w)+Z_0(w)\big)\big(Y_k(w) + Z_k(w)\big)]\\
	&= \ex\bigg[\bigg(\frac{1}{\sqrt{n}}\sum_{k=0}^{n} Y_k(w)+Z_k(w) \bigg)^2 \bigg] + \frac{1}{n} \sum_{k=-n}^{n} |k| \ex[\big(Y_0(w)+Z_0(w)\big)\big(Y_k(w)+ Z_k(w)\big)]\\
	&= \var\bigg(\frac{1}{\sqrt{n}}\sum_{k=0}^{n} Y_k(w)+Z_k(w) \bigg) + \mathcal{O}(n^{-1})
	\end{align*}
	by Assumption~\ref{cond:cum}.
	
	\textit{Trace class:}
	Let $(\psi_\ell^{\scs (1)})_{\ell\in\N}$ and $(\psi_\ell^{\scs (2)})_{\ell\in\N}$ be orthonormal bases of $L^2([0,1]^2)$ and $L^2([0,1]^3)$ respectively. Then the union $\{(\psi_\ell^{(\scs 1)},0)\}_{\ell\in\N} \cup \{(0,\psi_\ell^{\scs (2)})\}_{\ell\in\N}$ is an orthonormal basis of $\Hc_{2}$. By the definition of the trace norm, we have
	\begin{align*}
	\|C_\Bb\|_\Nc &= \sum_{\ell=1}^\infty \langle C_\Bb (\psi_\ell^{(1)},0),(\psi_\ell^{(1)},0)\rangle + \sum_{\ell=1}^\infty \langle C_\Bb (0,\psi_\ell^{(2)}),(0,\psi_\ell^{(2)})\rangle\\
	&= \sum_{\ell=1}^{\infty} \big\langle \langle r^{(m)}(\ast,\cdot),\psi_{\ell}^{(1)}(\cdot)\rangle_2,  \psi_\ell^{(1)}(\ast) \big\rangle_2 + \big\langle \langle r_{0,0}^{(c)}(\ast,\cdot),\psi_{\ell}^{(2)}(\cdot)\rangle_3,  \psi_\ell^{(2)}(\ast) \big\rangle_3\\
	&= \|C^{(m)}\|_\Nc + \|C_0^{(c)}\|_\Nc,
	\end{align*}
	where $C^{(m)}$ and $C_0^{(c)}$ are the operators defined by the kernels $r^{(m)}$ and $r_{0,0}^{(c)}$ respectively. By the proof of (D3) in the proof of Proposition \ref{PropCovConv}, Fatou's lemma and Fubini's theorem,
	\begin{align*}
	\|C^{(m)}\|_\Nc
	&= \sum_{\ell=1}^{\infty} \langle C^{(m)} \psi_\ell^{(1)}, \psi_\ell^{(1)} \rangle\\
	&= \sum_{\ell=1}^{\infty} \int_{[0,1]^4} \sum_{k=-\infty}^{\infty} \int_0^{u\wedge v} \cov\big(X_0^{(w)}(\phi),X_k^{(w)}(\tau)\big)\diff w \psi_\ell^{(1)}(\phi, v) \psi_\ell^{(1)}(\tau, u)\diff(u,v,\tau,\phi)\\
	&= \sum_{\ell=1}^{\infty} \lim\limits_{T\to\infty} \cov\big(\langle \tilde{B}_T,\psi_\ell^{(1)}\rangle , \langle \tilde{B}_T,\psi_\ell^{(1)}\rangle \big)\\
	&= \sum_{\ell=1}^{\infty} \lim\limits_{T\to\infty} \ex[\langle \tilde{B}_T,\psi_\ell^{(1)}\rangle^2]\\
	&\leq \liminf_{T\to\infty}  \sum_{\ell=1}^{\infty} \ex[\langle \tilde{B}_T,\psi_\ell^{(1)}\rangle^2]\\
	&= \liminf_{T\to\infty} \ex\bigg[\sum_{\ell=1}^{\infty}\langle \tilde{B}_T,\psi_\ell^{(1)}\rangle^2\bigg]\\
	&= \liminf_{T\to\infty} \ex\|\tilde{B}_T\|_{2,2}^2,
	\end{align*}
	which is finite since for any $T\in\N$ since
	\begin{align*}
	\ex\|\tilde{B}_T\|_{2,2}^2
	&= \ex\bigg[\int_{[0,1]^2} \tilde{B}_T^2(u,\tau)\diff(u,\tau) \bigg]\\
	&= \int_{[0,1]^2} \ex[\tilde{B}_T^2(u,\tau)]\diff(u,\tau)\\
	&= \int_{[0,1]^2} \ex\bigg[\frac{1}{T}\sum_{t_1,t_2=1}^{\lfloor uT\rfloor}\big(X_{t_1,T}(\tau)-\ex X_{t_1,T}(\tau)\big)\big(X_{t_2,T}(\tau)-\ex X_{t_2,T}(\tau)\big) \bigg] \diff(u,\tau)\\
	&= \int_{[0,1]^2} \frac{1}{T}\sum_{t_1,t_2=1}^{\lfloor uT\rfloor}\cov\big(X_{t_1,T}(\tau),X_{t_2,T}(\tau)\big) \diff(u,\tau)\\
	&\leq \frac{1}{T}\sum_{t_1,t_2=1}^{T}\int_{[0,1]} |\cov\big(X_{t_1,T}(\tau),X_{t_2,T}(\tau)\big)|\diff\tau \\
	&\leq \frac{1}{T}\sum_{t_1,t_2=1}^{T} \nu_2(t_2-t_1)\leq C < \infty.
	\end{align*}
	By similar arguments, it follows that $\|C_0^{(c)}\|_\Nc\leq C$, thus $\|C_\Bb\|_\Nc<\infty$.
\end{proof}

\begin{lemma} \label{lem:lip}
Suppose that $\{X_{t,T}: t=1, \dots, T\}_{T\in\N}$ is a locally stationary time series of order $\rho\ge 1$. Then, for any $1\le p \le \rho$, 
	\begin{equation*} 
	\ex\big[\|X_t^{(u)}-X_t^{(v)}\|_2^p \big]\leq C_p|u-v|^p \qquad \forall\ u,v \in [0,1],
	\end{equation*}	 
	where $C_p=2^{p-1} \sup_{t=1, \dots, T, T\in \N, u \in [0,1]} \ex |P_{t,T}^{\scs (u)}|^p$.
\end{lemma} 

\begin{proof} By the triangle  inequality and convexity of $x\mapsto |x|^p$, 
	\begin{align*} 
	\ex\Big[\|X_t^{(u)}-X_t^{(v)}\|_2^p  \Big]
	&\leq 
	2^{p-1} 
	\ex\Big [\|X_t^{(u)}-X_{\lfloor uT\rfloor,T}\|_2^p+\|X_{\lfloor uT\rfloor,T}-X_t^{(v)}\|_2^p \Big]  \\
	&\leq  
	C_p \big(|u-v|+\tfrac{4}{T}\big)^p ,
	\end{align*}
	for any $T\in\naturals$.
	\end{proof}

	Recall the notations introduced in Section~\ref{subsec:boot}. For $k\in\naturals$, define 
\begin{align} \label{eq:tildeb}
\tilde{B}_{T}^{(k)}(u,\tau)
=
\frac{1}{\sqrt{T}} \sum_{i=1}^{\lfloor uT\rfloor}\frac{R_i^{(k)}}{\sqrt{m}}\sum_{t=i}^{(i+m-1)\wedge T} \big\{ X_{t,T}(\tau)- \mu_{t,T}(\tau) \big\},
\end{align}
where $\mu_{t,T}(\tau) = \Exp[X_{t,T}(\tau)]$
and, for any $0\leq h\leq H$, let
\begin{align} \label{eq:tildebh}
\tilde{B}_{T,h}^{(k)}(u,\tau_1,\tau_2)
=
\frac{1}{\sqrt{T}} \sum_{i=1}^{\lfloor uT\rfloor \wedge (T-h)}\frac{R_i^{(k)}}{\sqrt{m}}\sum_{t=i}^{(i+m-1)\wedge (T-h)} \big\{X_{t,T}(\tau_1)X_{t+h,T}(\tau_2)- \mu_{t,T,h}(\tau_1, \tau_2)\big\},
\end{align}
where $\mu_{t,T,h}(\tau_1, \tau_2) = \ex[X_{t,T}(\tau_1)X_{t+h,T}(\tau_2)]$. Finally, let
\[
\Bb_T^{(k)} = (\tilde{B}_{T}^{(k)},\tilde{B}_{T,0}^{(k)},\dots ,\tilde{B}_{T,H}^{(k)}).
\]
We then have the following joint asymptotic behaviour of the primary process $\Bb_T$ and the non-observable multiplier versions $\Bb_T^{\scs (k)}$. Note that Theorem~\ref{theo:main} is an immediate consequence.

\begin{theorem}\label{bootstrapThm}
	Suppose that Assumptions~\ref{cond:ls}--\ref{cond:cum} and \ref{eq:b1} and \ref{eq:b3} are met. Then, for any fixed $K\in\N$,
	\[
	\big(\mathbb{B}_T,\mathbb{B}_{T}^{(1)},\dots ,\mathbb{B}_{T}^{(K)}\big)
	\convw 
	\big(\mathbb{B},\mathbb{B}^{(1)},\dots ,\mathbb{B}^{(K)}\big)
	\]
	in $\{ L^2([0,1]^2)\times(L^2([0,1]^3))^{H+1}\}^{K+1}$, where $\mathbb{B}^{ (1)}, \dots, \mathbb{B}^{ (K)}$ are independent copies of the centred Gaussian variable $\mathbb{B}$ from Theorem~\ref{theo:main} (see also Lemma~\ref{lem:nuc}).
\end{theorem}

\begin{proof}[Proof of Theorem~\ref{bootstrapThm}] 
We only prove the assertion for $K=1$; the general case follows by the same arguments but is notationally more involved.   The theorem is then an immediate consequence of the fundamental approximation Lemma~\ref{lem:app}, together with  Lemma~\ref{lem:fidis} and \ref{ThmTightness}.
\end{proof}

Let $\{\psi'_n\}_{n\in\naturals}$ and $\{\psi_n\}_{n\in\naturals}$ be orthonormal bases of $L^2([0,1]^2)$ and $L^2([0,1]^3)$ with 
\[\sup_{n\in\naturals} \sup_{x\in[0,1]^2} |\psi'_n(x)|\leq C<\infty~ \text{and} ~\sup_{n\in\naturals} \sup_{x\in[0,1]^3} |\psi_n(x)|\leq C<\infty.\]
Note that such bases can be constructed as tensor products of the orthonormal basis 
\[\mathcal{B}=\{ \sqrt{2}\cos(2\pi nx), \sqrt{2}\sin(2\pi nx):n\in\naturals \}\cup\{1\}~\text{in}~L^2([0,1]) \]
(c.f.\ \citealpsuppl{Kadison1983}, Example 2.6.11).

\begin{lemma}\label{lem:fidis}
	Let assumptions~\ref{cond:ls}--\ref{cond:cum} and \ref{eq:b1} and \ref{eq:b3}  be satisfied. 
	Then, for any $p\in\naturals$,
	\begin{multline*}
	\Big( (\langle \tilde{B}_T,\psi'_n\rangle)_{n=1}^p,\big\{(\langle \tilde{B}_{T,h},\psi_n\rangle)_{n=1}^p\big\}_{h=0}^H,(\langle \tilde{B}_T^{(1)},\psi'_n\rangle)_{n=1}^p,\big\{(\langle \tilde{B}_{T,h}^{(1)},\psi_n\rangle)_{n=1}^p\big\}_{h=0}^H\Big)\\
	\convw \Big( (\langle \tilde B,\psi'_n\rangle)_{n=1}^p,\big\{(\langle \tilde B_h,\psi_n\rangle)_{n=1}^p\big\}_{h=0}^H,(\langle \tilde B^{(1)},\psi'_n\rangle)_{n=1}^p,\big\{(\langle \tilde B_h^{(1)},\psi_n\rangle)_{n=1}^p\big\}_{h=0}^H\Big),
	\end{multline*}
	in $\reals^{2(H+2)p}$.
\end{lemma}

\begin{proof}[Proof of Lemma~\ref{lem:fidis}]
	Fix some $p\in\naturals$. By the Cramér-Wold device, it is sufficient to show that	
	\begin{align} \label{eq:zt} 
	Z_T:=\sum_{n=1}^{p}\bigg(c_n \langle \tilde{B}_T,\psi'_n\rangle + d_n \langle \tilde{B}_T^{(1)},\psi'_n\rangle+ \sum_{h=0}^{H} c_{n,h} \langle \tilde{B}_{T,h},\psi_n\rangle + d_{n,h} \langle \tilde{B}_{T,h}^{(1)},\psi_n\rangle \bigg) 
	\end{align}
	converges weakly to
	\begin{align} \label{eq:z} 
	Z:= \sum_{n=1}^{p}\bigg(c_n \langle \tilde B,\psi'_n\rangle + d_n \langle \tilde B^{(1)},\psi'_n\rangle+ \sum_{h=0}^{H} c_{n,h} \langle \tilde B_h,\psi_n\rangle + d_{n,h} \langle \tilde B_h^{(1)},\psi_n\rangle \bigg),   
	\end{align}
	for any real numbers $c_n,d_n,c_{n,h},d_{n,h}\in\reals,1\leq n\leq p,0\leq h\leq H$. By Theorem 30.1 and Example 30.1 of \citesuppl{Billingsley1995}, the normal distribution is determined uniquely by its moments. Since there is a one-to-one correspondence between moments and cumulants, this also holds true for the latter ones. The only non-zero cumulants of a normal distribution are the first two, which equal the mean and the variance  \citepsuppl{Holmquist1988}.
	
	It is easy to see that $\ex Z_T=0$ since $\tilde{B}_T,\tilde{B}_{T,h},\tilde{B}_T^{(1)}$ and $\tilde{B}_{T,h}^{(1)}$ are centred, for any $0\leq h\leq H$. For example, we have by the Fubini-Tonelli theorem
	\begin{align*}
	\ex[\langle \tilde{B}_T,\psi_n\rangle] 
	&= \ex \bigg[\int_{[0,1]^3} \tilde{B}_T(u,\tau_1,\tau_2) \psi_n(u,\tau_1,\tau_2) \diff (u,\tau_1,\tau_2) \bigg]\\
	&= \int_{[0,1]^3} \ex\big[\tilde{B}_T(u,\tau_1,\tau_2)\big] \psi_n(u,\tau_1,\tau_2) \diff (u,\tau_1,\tau_2)=0.
	\end{align*}
	The theorem is applicable since, by the moment condition \ref{cond:mom}, $\ex\big[|\tilde{B}_T(u,\tau_1,\tau_2)|\big]<\infty$. From Proposition \ref{PropCovConv} follows the convergence of the second moments and by Proposition~\ref{PropCumConv}, the higher-order cumulants vanish. Thus, we can conclude the convergence of $Z_T$ to $Z$ by Theorem 2.22 of \citesuppl{vanderVaart1998}.
\end{proof}

\begin{proposition}\label{PropCovConv}
	Let assumptions \ref{cond:ls}--\ref{cond:cum} and \ref{eq:b1} and \ref{eq:b3} be satisfied. Then, 
	\[
	\lim_{T \to\infty}\var(Z_T)=\var(Z), 
	\]
	with $Z_T$  and $Z$ as defined in \eqref{eq:zt} and \eqref{eq:z}.
\end{proposition}
\begin{proof}[Proof of Proposition~\ref{PropCovConv}]
	Since $Z_T$ is a linear combination of $\langle\tilde{B}_T,\psi'_n\rangle, \langle\tilde{B}_{T,h},\psi_n\rangle,\langle\tilde{B}_T^{\scs (1)},\psi'_n\rangle$ and $\langle \tilde{B}_{T,h}^{\scs (1)}, \psi_n \rangle$, for $0\leq h\leq H$ and $1\leq n\leq p$, it is sufficient to prove that
	\begin{align*}
	&\lim_{T \to \infty}\cov(\langle \tilde{B}_T,\psi'_n\rangle, \langle \tilde{B}_T,\psi'_\ell\rangle) 
	= 
	\cov(\langle \tilde B,\psi'_n\rangle, \langle \tilde B,\psi'_\ell\rangle),\tag{D1}\\
	&\lim_{T \to \infty}\cov(\langle \tilde{B}_T,\psi'_n\rangle, \langle \tilde{B}_{T,h},\psi_\ell\rangle) 
	= 
	\cov(\langle \tilde B,\psi'_n\rangle, \langle \tilde B_h,\psi_\ell\rangle),\tag{D2}\\
	&\lim_{T \to \infty}\cov(\langle \tilde{B}_{T,h},\psi_n\rangle, \langle \tilde{B}_{T,h'},\psi_\ell\rangle) 
	= 
	\cov(\langle \tilde B_h,\psi_n\rangle, \langle \tilde B_{h'},\psi_\ell\rangle),\tag{D3}\\
	&\lim_{T \to \infty}\cov(\langle \tilde{B}_T^{(1)},\psi'_n\rangle, \langle \tilde{B}_T^{(1)},\psi'_\ell\rangle) 
	=
	\cov(\langle \tilde B^{(1)},\psi'_n\rangle, \langle \tilde B^{(1)},\psi'_\ell\rangle),\tag{D4}\\
	&\lim_{T \to \infty}\cov(\langle \tilde{B}_T^{(1)},\psi'_n\rangle, \langle \tilde{B}_{T,h}^{(1)},\psi_\ell\rangle) 
	= 
	\cov(\langle \tilde B^{(1)},\psi'_n\rangle, \langle \tilde B_h^{(1)},\psi_\ell\rangle),\tag{D5}\\
	&\lim_{T \to \infty}\cov(\langle \tilde{B}_{T,h}^{(1)},\psi_n\rangle, \langle \tilde{B}_{T,h'},\psi_\ell\rangle) 
	= 
	\cov(\langle \tilde B_h^{(1)},\psi_n\rangle, \langle \tilde B_{h'}^{(1)},\psi_\ell\rangle),\tag{D6}\\
	&\lim_{T \to \infty}\cov(\langle \tilde{B}_T,\psi'_n\rangle, \langle \tilde{B}_T^{(1)},\psi'_\ell\rangle) 
	= 
	\cov(\langle \tilde B,\psi'_n\rangle, \langle \tilde B^{(1)},\psi'_\ell\rangle)=0,\tag{D7}\\
	&\lim_{T \to \infty}\cov(\langle \tilde{B}_T,\psi'_n\rangle, \langle \tilde{B}_{T,h}^{(1)},\psi_\ell\rangle) 
	= 
	\cov(\langle \tilde B,\psi'_n\rangle, \langle \tilde B_h^{(1)},\psi_\ell\rangle)=0,\tag{D8}\\
	&\lim_{T \to \infty}\cov(\langle \tilde{B}_{T,h},\psi_n\rangle, \langle \tilde{B}_T^{(1)},\psi'_\ell\rangle) 
	= 
	\cov(\langle \tilde B_h,\psi_n\rangle, \langle \tilde B^{(1)},\psi'_\ell\rangle)=0, \tag{D9}\\
	&\lim_{T \to \infty}\cov(\langle \tilde{B}_{T,h},\psi_n\rangle, \langle \tilde{B}_{T,h'},\psi_\ell\rangle) 
	= 
	\cov(\langle \tilde B_h,\psi_n\rangle, \langle \tilde B_{h'}^{(1)},\psi_\ell\rangle)=0,\tag{D10}
	\end{align*}
	for any $h,h'\in\{0,\dots ,H\}$ and $n,\ell\in\{1,\dots ,p\}$. For that purpose, observe that all scalar products have mean zero.

	\noindent \textit{Proof of (D3).} Fix $h,h'\in\{0,\dots ,H\}$. We have
	{\small\begin{align*} 
	\begin{split}
		S_{T,3}=&\ \cov\big(\langle\tilde{B}_{T,h},\psi_{n}\rangle, \langle\tilde{B}_{T,h'},\psi_{\ell}\rangle\big)\\
		=&\ \ex\bigg[ \int_{[0,1]^6} \tilde{B}_{T,h}(u,\tau_1,\tau_2)\psi_{n}(u,\tau_1,\tau_2) \tilde{B}_{T,h'}(u',\tau_1',\tau_2')\psi_{\ell}(u',\tau_1',\tau_2') \diff (u,u',\tau_1,\tau_1',\tau_2,\tau_2')\bigg]\\
		=&\ \frac{1}{T} \sum_{t=1}^{T-h}\sum_{t'=1}^{T-h'} \ex\bigg[\int_{[0,1]^6} \psi_{n}(u,\tau_1,\tau_2)\psi_{\ell}(u',\tau_1',\tau_2')\big\{X_{t,T}(\tau_1)X_{t+h,T}(\tau_2)-\mu_{t,T,h}(\tau_1, \tau_2) \big\}\\ 
		&\phantom{=}\times\big\{X_{t',T}(\tau_1')X_{t'+h',T}(\tau_2')-\mu_{t',T,h'}(\tau_1', \tau_2')\big\} \id(t\leq \lfloor uT \rfloor, t'\leq \lfloor u'T\rfloor)\diff (u,u',\tau_1,\tau_1',\tau_2,\tau_2') \bigg] \\
		=&\  \frac{1}{T} \sum_{t=1}^{T-h}\sum_{t'=1}^{T-h'} \int_{[0,1]^6} \cov\big(X_{t,T}(\tau_1)X_{t+h,T}(\tau_2),X_{t',T}(\tau_1')X_{t'+h',T}(\tau_2')\big)\\
		&\hspace{2.5cm} \times \psi_{n}(u,\tau_1,\tau_2)\psi_{\ell}(u',\tau_1',\tau_2') \id(t\leq \lfloor uT \rfloor, t'\leq \lfloor u'T\rfloor)\diff (u,u',\tau_1,\tau_1',\tau_2,\tau_2'),
		\end{split}
		\end{align*}
	}where we applied Fubini's theorem in the last equality. Further, we can rewrite
	{\small\begin{align*}
	\cov(X_{t,T}\otimes X_{t+h,T},X_{t',T}\otimes X_{t'+h',T})
	&=
	\cov\big((X_{t,T}-X_t^{(t/T)})\otimes X_{t+h,T},X_{t',T}\otimes X_{t'+h',T}\big)\\
	&\hspace{0.5cm}+ \cov\big(X_t^{(t/T)}\otimes (X_{t+h,T}-X_{t+h}^{(t/T)}),X_{t',T}\otimes X_{t'+h',T}\big)\\
	&\hspace{0.5cm}+ \cov\big(X_t^{(t/T)}\otimes X_{t+h}^{(t/T)},(X_{t',T}-X_{t'}^{(t/T)})\otimes X_{t'+h',T}\big)\\
	&\hspace{0.5cm}+ \cov\big(X_t^{(t/T)}\otimes X_{t+h}^{(t/T)},X_{t'}^{(t/T)}\otimes(X_{t'+h',T}-X_{t'+h'}^{(t/T)}) \\
	&\hspace{0.5cm}+  \cov(X_t^{(t/T)}\otimes X_{t+h}^{(t/T)},X_{t'}^{(t/T)}\otimes X_{t'+h'}^{(t/T)})\big).
	\end{align*}
	}Invoking this decomposition, we can split each integral appearing in $S_{T,3}$ into five summands. 
	By \ref{cond:cum}, Proposition \ref{prop:cov} and the Cauchy-Schwarz inequality, the sums over all of this summands are of the order $\mathcal O(T^{-1})$, except for the last one. Thus, 
we obtain that
{\small\begin{align}\label{eq:S2}
	S_{T,3}&= \int_{[0,1]^6} \bigg\{\frac{1}{T}\sum_{t=1}^{\lfloor (u\wedge u')T \rfloor}\cov\big(X_{t}^{(t/T)}(\tau_1)X_{t+h}^{(t/T)}(\tau_2),X_{t}^{(t/T)}(\tau_1')X_{t+h'}^{(t/T)}(\tau_2')\big)\\
	&\hspace{1.8cm}+ \frac{1}{T}\sum_{t=1}^{\lfloor uT \rfloor}\sum_{t'=t+1}^{\lfloor u'T\rfloor} \cov\big(X_{t}^{(t/T)}(\tau_1)X_{t+h}^{(t/T)}(\tau_2),X_{t'}^{(t/T)}(\tau_1')X_{t'+h'}^{(t/T)}(\tau_2')\big) \nonumber\\
	&\hspace{1.8cm}+ \frac{1}{T}\sum_{t'=1}^{\lfloor u'T\rfloor}\sum_{t=t'+1}^{\lfloor uT \rfloor} \cov\big(X_{t}^{(t/T)}(\tau_1)X_{t+h}^{(t/T)}(\tau_2),X_{t'}^{(t/T)}(\tau_1')X_{t'+h'}^{(t/T)}(\tau_2')\big)\bigg\} \nonumber \\
	&\hspace{1.3cm}\times \psi_{n}(u,\tau_1,\tau_2)\psi_{\ell}(u',\tau_1',\tau_2') \diff (u,u',\tau_1,\tau_1',\tau_2,\tau_2') + \mathcal{O}(T^{-1}). \nonumber
	\end{align}
}The convergence of the integral over the first sum is straightforward:
	{\small\begin{align*}
	\lim\limits_{T\to\infty} & \int_{[0,1]^6} \frac{1}{T} \sum_{t=1}^{\lfloor (u\wedge u')T\rfloor}\cov\big(X_{0}^{(t/T)}(\tau_1)X_{h}^{(t/T)}(\tau_2),X_{0}^{(t/T)}(\tau_1')X_{h'}^{(t/T)}(\tau_2')\big)\\
	&\hspace{1.3cm} \times \psi_{n}(u,\tau_1,\tau_2)\psi_{\ell}(u',\tau_1',\tau_2')\diff (u,u',\tau_1,\tau_1',\tau_2,\tau_2')\\
	=&\int_{[0,1]^6} \int_{0}^{u\wedge u'}\cov\big(X_{0}^{(w)}(\tau_1)X_{h}^{(w)}(\tau_2),X_{0}^{(w)}(\tau_1')X_{h'}^{(w)}(\tau_2')\big)\diff w\\
	&\hspace{1.3cm}\times \psi_{n}(u,\tau_1,\tau_2)\psi_{\ell}(u',\tau_1',\tau_2')\diff (u,u',\tau_1,\tau_1',\tau_2,\tau_2'),
	\end{align*}
	}where the limit and the integral can be interchanged, by \ref{cond:mom} and Lebesgue's dominated convergence theorem. The convergence of the remaining two sums is technically more involved, and we only present details for the case $t<t'$. 
	By stationarity of $(X_t^{\scs (u)})_{t\in\Z}$,
{\small\begin{align*}
	&\int_{[0,1]^6} 
	\bigg\{\frac{1}{T} \sum_{t=1}^{\lfloor uT \rfloor}\sum_{t'=t+1}^{\lfloor u'T\rfloor}\cov\big(X_{t}^{(t/T)}(\tau_1)X_{t+h}^{(t/T)}(\tau_2),X_{t'}^{(t/T)}(\tau_1')X_{t'+h'}^{(t/T)}(\tau_2')\big)\bigg\}\\
	&\hspace{1cm}  
	\times \psi_{n}(u,\tau_1,\tau_2)\psi_{\ell}(u',\tau_1',\tau_2')\diff (u,u',\tau_1,\tau_1',\tau_2,\tau_2')\\
	=&\int_{[0,1]^6} 
	\bigg\{\frac{1}{T} \sum_{t=1}^{\lfloor (u\wedge u')T \rfloor}\sum_{k=1}^{\lfloor u'T\rfloor-t}\cov\big(X_{0}^{(t/T)}(\tau_1)X_{h}^{(t/T)}(\tau_2),X_{k}^{(t/T)}(\tau_1')X_{k+h'}^{(t/T)}(\tau_2')\big)\bigg\}\\
	&\hspace{1cm}   
	\times \psi_{n}(u,\tau_1,\tau_2)\psi_{\ell}(u',\tau_1',\tau_2')\diff (u,u',\tau_1,\tau_1',\tau_2,\tau_2')\\
	=&\int_{[0,1]^6} 
	\bigg\{\frac{1}{T} \sum_{k=1}^{{\lfloor u'T\rfloor -1}}\sum_{t=1}^{\lfloor uT\rfloor\wedge (\lfloor u'T\rfloor-k)}\cov\big(X_{0}^{(t/T)}(\tau_1)X_{h}^{(t/T)}(\tau_2),X_{k}^{(t/T)}(\tau_1')X_{k+h'}^{(t/T)}(\tau_2')\big)\bigg\}\\
	&\hspace{1cm}  
	\times \psi_{n}(u,\tau_1,\tau_2)\psi_{\ell}(u',\tau_1',\tau_2') \diff (u,u',\tau_1,\tau_1',\tau_2,\tau_2').
\end{align*}
}By Lebesgue's dominated convergence theorem and Lemma \ref{LemmaLimitIntegral}, the right-hand side of the latter display converges to
{\small\begin{multline*} 
	\int_{[0,1]^6} 
	\bigg\{\sum_{k=1}^{\infty}\int_{0}^{u\wedge u'}\cov\big(X_{0}^{(w)}(\tau_1)X_{h}^{(w)}(\tau_2),X_{k}^{(w)}(\tau_1')X_{k+h'}^{(w)}(\tau_2')\big)\diff w\bigg\}\\
	\times \psi_{n}(u,\tau_1,\tau_2)\psi_{\ell}(u',\tau_1',\tau_2') \diff (u,u',\tau_1,\tau_1',\tau_2,\tau_2'). 
\end{multline*} 
}Thus, we have 
{\small\begin{align*}
\lim_{T\to\infty} S_{T,3} 
&= \int_{[0,1]^6} \psi_{n}(u,\tau_1,\tau_2)\psi_{\ell}(u',\tau_1',\tau_2')\cov\big(\tilde B_h(u,\tau_1,\tau_2),\tilde B_{h'}(u',\tau_1',\tau_2')\big)\diff (u,u',\tau_1,\tau_1',\tau_2,\tau_2')\\
&= \cov(\langle \tilde B_h,\psi_{n}\rangle,\langle \tilde B_{h'},\psi_{\ell}\rangle),
\end{align*}
}which proves (D3).

\medskip

\noindent \textit{Proof of (D6).}  By the independence of the standard normally distributed random variables $(R_i)_{i\in\naturals} = (R_i	^{\scs (1)})_{i\in\naturals}$, we have 
{\small\begin{align*}
	S_{T,6} 
	&= 
	\cov(\langle \tilde{B}_{T,h}^{(1)},\psi_n\rangle, \langle \tilde{B}_{T,h'},\psi_\ell\rangle)\\
	&= 
\frac{1}{T}\sum_{i=1}^{T-h}\sum_{i'=1}^{T-h'}\frac{1}{m}\ex\bigg[ R_iR_{i'} \int_{[0,1]^6}\psi_n(u,\tau_1,\tau_2)\psi_\ell(u',\tau_1',\tau_2') \id(i\leq\lfloor uT\rfloor, i'\leq \lfloor u'T\rfloor)\\
	&\phantom{====}
\times\bigg( \sum_{t=i}^{(i+m-1)\wedge (T-h)}X_{t,T}(\tau_1)X_{t+h,T}(\tau_2)-\mu_{t,T,h}(\tau_1, \tau_2) \bigg)\\
	&\phantom{====}
\times\bigg( \sum_{t'=i'}^{(i'+m-1)\wedge (T-h')}X_{t',T}(\tau_1')X_{t'+h',T}(\tau_2')-\mu_{t',T,h'}(\tau_1', \tau_2') \bigg)\diff (u,u',\tau_1,\tau_1',\tau_2,\tau_2') \bigg]\\
	&= 
\frac{1}{T}\sum_{i=1}^{T-(h\vee h')}\frac{1}{m}\ex\bigg[ \int_{[0,1]^6}\psi_n(u,\tau_1,\tau_2)\psi_\ell(u',\tau_1',\tau_2') \id(i\leq\lfloor uT\rfloor\wedge \lfloor u'T\rfloor)\\
	&\phantom{====}
\times\bigg( \sum_{t=i}^{(i+m-1)\wedge (T-h)}X_{t,T}(\tau_1)X_{t+h,T}(\tau_2)-\mu_{t,T,h}(\tau_1, \tau_2) \bigg)\\
	&\phantom{====}
\times\bigg( \sum_{t'=i}^{(i+m-1)\wedge (T-h')}X_{t',T}(\tau_1')X_{t'+h',T}(\tau_2')-\mu_{t',T,h'}(\tau_1', \tau_2')  \bigg)\diff (u,u',\tau_1,\tau_1',\tau_2,\tau_2') \bigg]\\
	&= 
	\frac{1}{T}\sum_{i=1}^{T-(h\vee h')}\frac{1}{m} \int_{[0,1]^6}  \psi_n(u,\tau_1,\tau_2)\psi_\ell(u',\tau_1',\tau_2') \id(i\leq\lfloor (u\wedge u')T\rfloor ) 
 \\
	&\phantom{====}  \times \bigg( \sum_{t=i}^{(i+m-1)\wedge(T-h)}\sum_{t'=i}^{(i+m-1)\wedge(T-h')} \cov\big(X_{t,T}(\tau_1)X_{t+h,T}(\tau_2),X_{t',T}(\tau_1')X_{t'+h',T}(\tau_2')\big)  \bigg)  \\
	&\hspace{10cm}
	\diff (u,u',\tau_1,\tau_1',\tau_2,\tau_2'), 
\end{align*}
}by Fubini's theorem. 
By the same arguments that led to \eqref{eq:S2}, we further  have
{\small\begin{align*}
	S_{T,6} 
	&= 
	\frac{1}{T}\sum_{i=1}^{T}\frac{1}{m} \int_{[0,1]^6}  \psi_n(u,\tau_1,\tau_2)\psi_\ell(u',\tau_1',\tau_2') \id(i\leq\lfloor (u\wedge u')T\rfloor ) 
 \\
	&\phantom{====}  \times \bigg(\sum_{t=i}^{i+m-1}\sum_{t'=i}^{i+m-1}\cov\big(X_t^{(i/T)}(\tau_1)X_{t+h}^{(i/T)}(\tau_2),X_{t'}^{(i/T)}(\tau_1')X_{t'+h'}^{(i/T)}(\tau_2')\big)\bigg)  \\
	&\hspace{9cm}
	\diff (u,u',\tau_1,\tau_1',\tau_2,\tau_2') +\mathcal{O}(m^{-1}). 
\end{align*}
}As before, we split the above sum into three sums $L_1, L_2, L_3$, for $t=t', t<t'$ and $t>t'$, respectively. For $L_1$, we have
{\small\begin{align*}
L_1
&=
\frac{1}{T}\sum_{i=1}^{T}\frac{1}{m} \int_{[0,1]^6}\bigg(\sum_{t=i}^{i+m-1 }\cov\big(X_t^{(i/T)}(\tau_1)X_{t+h}^{(i/T)}(\tau_2),X_t^{(i/T)}(\tau_1')X_{t+h'}^{(i/T)}(\tau_2')\big)\bigg)\\
&\hspace{3cm}
\times \psi_n(u,\tau_1,\tau_2)\psi_\ell(u,\tau_1,\tau_2) \id(i\leq\lfloor uT\rfloor\wedge \lfloor u'T\rfloor)\diff (u,u',\tau_1,\tau_1',\tau_2,\tau_2')\\
&=
\int_{[0,1]^6} \frac{1}{T}\bigg(\sum_{i=1}^{\lfloor (u\wedge u')T\rfloor }\cov\big(X_0^{(i/T)}(\tau_1)X_{h}^{(i/T)}(\tau_2),X_0^{(i/T)}(\tau_1')X_{h'}^{(i/T)}(\tau_2')\big)\bigg) \\
&\hspace{3cm}
\times\psi_n(u,\tau_1,\tau_2)\psi_\ell(u,\tau_1,\tau_2)\diff (u,u',\tau_1,\tau_1',\tau_2,\tau_2'),
\end{align*}
}by stationarity of $(X_t^{\scs (u)})_{t\in\Z}$. The right-hand side converges to 
{\small
\begin{multline*}
\int_{[0,1]^6}\psi_n(u,\tau_1,\tau_2)\psi_\ell(u,\tau_1,\tau_2) \\
\times \int_0^{u\wedge u'}\cov\big(X_0^{(w)}(\tau_1)X_{h}^{(w)}(\tau_2),X_0^{(w)}(\tau_1')X_{h}^{(w)}(\tau_2')\big)\diff w \diff (u,u',\tau_1,\tau_1',\tau_2,\tau_2'), 
\end{multline*}
}as $T$ tends to infinity, by Lebegue's dominated convergence theorem. 

The sums $L_2$ and $L_3$ can be treated in a similar manner, and we only provide details for $L_2$. By the same arguments as before and the stationarity of $(X_t^{\scs (u)})_{t\in\Z}$, it follows
{\small\begin{align*}
L_2 
&=\int_{[0,1]^6}\frac{1}{T}\sum_{i=1}^{\lfloor(u\wedge u')T\rfloor}\frac{1}{m} \sum_{t=i}^{i+m-2}\sum_{t'=t+1}^{i+m-1}\cov\big(X_t^{(i/T)}(\tau_1)X_{t+h}^{(i/T)}(\tau_2),X_{t'}^{(i/T)}(\tau_1')X_{t'+h'}^{(i/T)}(\tau_2')\big)\\
&\phantom{==}\times\psi_n(u,\tau_1,\tau_2)\psi_\ell(u',\tau'_1,\tau'_2) \diff (u,u',\tau_1,\tau_1',\tau_2,\tau_2')\\
&=\int_{[0,1]^6}\frac{1}{T}\sum_{i=1}^{\lfloor(u\wedge u')T\rfloor}\frac{1}{m} \sum_{t=i}^{i+m-2}\sum_{k=1}^{i+m-1-t}\cov\big(X_0^{(i/T)}(\tau_1)X_h^{(i/T)}(\tau_2),X_{k}^{(i/T)}(\tau_1')X_{k+h'}^{(i/T)}(\tau_2')\big)\\
&\phantom{==}\times\psi_n(u,\tau_1,\tau_2)\psi_\ell(u',\tau'_1,\tau'_2) \diff (u,u',\tau_1,\tau_1',\tau_2,\tau_2')\\
&=\int_{[0,1]^6}\sum_{k=1}^{m-1} \frac{m-k}{m} \frac{1}{T}\sum_{i=1}^{\lfloor(u\wedge u')T\rfloor}\cov\big(X_0^{(i/T)}(\tau_1)X_h^{(i/T)}(\tau_2),X_{k}^{(i/T)}(\tau_1')X_{k+h'}^{(i/T)}(\tau_2')\big)\\
&\phantom{==}\times\psi_n(u,\tau_1,\tau_2)\psi_\ell(u',\tau'_1,\tau'_2) \diff (u,u',\tau_1,\tau_1',\tau_2,\tau_2').
\end{align*} 
}The right-hand side of the previous display converges to
{\small\begin{multline*}
 \int_{[0,1]^6}\psi_n(u,\tau_1,\tau_2)\psi_\ell(u',\tau_1',\tau_2') \sum_{k=1}^{\infty}\int_0^{u\wedge u'}\cov\big(X_0^{(w)}(\tau_1)X_{h}^{(w)}(\tau_2),X_{k}^{(w)}(\tau_1')X_{k+h'}^{(w)}(\tau_2')\big)\diff w \\
 \diff (u,u',\tau_1,\tau_1',\tau_2,\tau_2'),
 \end{multline*}
}as $T$ tends to infinity, by Lebegue's dominated convergence theorem. Thus, (D6) follows by Fubini's theorem, since
{\small\begin{align*}
\lim\limits_{T\to\infty} S_{T,6}
&= \int_{[0,1]^6} \sum_{k=-\infty}^{\infty}\int_0^{u\wedge u'}\cov\big(X_0^{(w)}(\tau_1)X_{h}^{(w)}(\tau_2),X_{k}^{(w)}(\tau_1')X_{k+h'}^{(w)}(\tau_2')\big)\diff w\\
&\phantom{==}\times \psi_n(u,\tau_1,\tau_2)\psi_\ell(u,\tau_1,\tau_2) \diff (u,u',\tau_1,\tau_1',\tau_2,\tau_2')\\
&= \ex[\langle \psi_n,\tilde B_h^{(1)}\rangle \langle \psi_\ell,\tilde B_{h'}^{(1)}\rangle]= \cov(\langle \psi_n,\tilde B_h^{(1)}\rangle, \langle \psi_\ell,\tilde B_{h'}^{(1)}\rangle).
\end{align*}
}

\noindent \textit{Proof of (D7)--(D10).}  The convergences (D7) to (D10) follows from the fact that the multipliers $R_i=R_i^{\scs (1)}$ are independent from the data and centred. For example, we have 
{\small\begin{align*}
	&\, \cov(\langle \tilde{B}_{T,h},\psi_n\rangle, \langle \tilde{B}_{T,h}^{(1)},\psi_\ell\rangle)\\
	=&\, \ex[\langle \tilde{B}_{T,h},\psi_n\rangle \langle \tilde{B}_{T,h}^{(1)},\psi_\ell\rangle]\\
	=&\, \ex\bigg[\langle \tilde{B}_{T,h},\psi_n\rangle \sum_{i=1}^{T-h}\frac{R_i}{\sqrt{mT}}\sum_{t=i}^{(i+m-1)\wedge (T-h)}\int_{[0,1]^3}\big(X_{t,T}(\tau_1)X_{t+h,T}(\tau_2)-\mu_{t,T,h}(\tau_1, \tau_2) \big)\\
	&\phantom{====}\times\psi_\ell(u,\tau_1,\tau_2)\id(i\leq \lfloor uT\rfloor)\diff (u,\tau_1,\tau_2)\bigg] = 0,
\end{align*}
}which implies (D10).
	
\noindent \textit{Proof of (D1)-(D2), (D4)-(D5).}  Convergences (D1)-(D2) and (D4)-(D5) can be shown with the same arguments as (D3) and (D6), respectively, but they are technically less involved.
%
\end{proof}

\begin{proposition}\label{PropCumConv}
	Let assumptions \ref{cond:ls}--\ref{cond:cum} and \ref{eq:b1} and \ref{eq:b3} be satisfied. Then, 
	\[
	\lim_{T\to\infty}\cum_j(Z_T)=0, 
	\]
	for any $j\geq 3$, where $Z_T$ is defined in \eqref{eq:zt}.
\end{proposition}

\begin{proof}
	By linearity of cumulants, we have
	\begin{align*}
	&\, \cum_j(Z_T) \\
	=&\, \cum_j\Bigg(\sum_{n=1}^{p}\bigg(c_n \langle \tilde{B}_T,\psi'_n\rangle + d_n \langle \tilde{B}_T^{(1)},\psi'_n\rangle+ \sum_{h=0}^{H} c_{n,h} \langle \tilde{B}_{T,h},\psi_n\rangle + d_{n,h} \langle \tilde{B}_{T,h}^{(1)},\psi_n\rangle \bigg)\Bigg)\\
	=&\, \sum_{n_1,\dots ,n_j=1}^{p} \sum_{\substack{w=(w_1,\dots ,w_j)\\w_i\in\{0,1\}\times\{-1,\dots ,H\}, 1\leq i\leq j}} \bigg(\prod_{i=1}^j a_i^{(w_i)}\bigg) \cum(\langle A_1^{(w_1)},\psi_{n_1}\rangle,\dots ,\langle A_j^{(w_j)},\psi_{n_j}\rangle),
	\end{align*}
	where 
	\[a_i^{(0,-1)}=c_i,A_i^{(0,-1)}=\tilde{B}_T, a_i^{(0,h)}=c_{i,h}~\text{and}~A_i^{(0,h)}=\tilde{B}_{T,h},\]
	and further,
	\[a_i^{(1,-1)}=d_i,A_i^{(1,-1)}=\tilde{B}_T^{(1)}, a_i^{(1,h)}=d_{i,h}~\text{and}~A_i^{(1,h)}=\tilde{B}_{T,h}^{(1)},\]
	for $h=0,\dots ,H$ and $i=1,\dots ,j$. Fix some integers $n_1,\dots ,n_j\in\{1,\dots ,p\}$. Further, denote the cumulants in the above sum by
	\[ \overline{\cum}(w):=\cum(\langle A_1^{(w_1)},\psi_{n_1}\rangle,\dots ,\langle A_j^{(w_j)},\psi_{n_j}\rangle). \]
	In the following, we restrict our attention to the subset $(\{0,1\}\times\{0,\dots ,H\})^j$ of the set $(\{0,1\}\times\{-1,\dots ,H\})^j$ since the proof for the latter follows the same arguments but is notationally more involved. 
	
	First, fix $w=(w_1,\dots ,w_j)$ with $w_i\in \{0\}\times \{0,\dots ,H\}$. Thus, for any $w_i$ there is a $h_i\in\{0,\dots H\}$ such that $A_i^{\scs (w_i)}=\tilde{B}_{T,h_i}$.
	By the definition of cumulants and Fubini's theorem, we obtain that
	{\small\begin{align*}	
	&\, \cum(\langle \tilde{B}_{T,h_1},\psi_{n_1}\rangle,\dots ,\langle \tilde{B}_{T,h_j},\psi_{n_j}\rangle)\\
	=&\, 
	\sum_{\{\nu_1,\dots ,\nu_R\}} (-1)^{R-1} (R-1)! \prod_{r=1}^{R} \ex\bigg[ \prod_{i\in\nu_r} \int_{[0,1]^3} \psi_{n_i}\big(u^{(i)},\tau_1^{(i)},\tau_2^{(i)}\big)\\
	&\phantom{===} 
	\times \frac{1}{\sqrt{T}} \sum_{t=1}^{\lfloor u^{(i)}T\rfloor\wedge (T-h_i)} \big(X_{t,T}(\tau_1^{(i)})X_{t+h_i,T}(\tau_2^{(i)})-\mu_{t,T,h_i}(\tau_1^{(i)},\tau_2^{(i)})\big)\diff (u^{(i)},\tau_1^{(i)},\tau_2^{(i)}) \bigg]\\
	=&\, 
	\sum_{\{\nu_1,\dots ,\nu_R\}} (-1)^{R-1} (R-1)! \prod_{r=1}^{R} \ex\bigg[\int_{[0,1]^{3|\nu_r|}}  \prod_{i\in\nu_r} \psi_{n_i}\big(u^{(i)},\tau_1^{(i)},\tau_2^{(i)}\big) \\
	&\phantom{===} 
	\times \frac{1}{\sqrt{T}} \sum_{t=1}^{\lfloor u^{(i)}T\rfloor\wedge (T-h_i)} \big(X_{t,T}(\tau_1^{(i)})X_{t+h_i,T}(\tau_2^{(i)})-\mu_{t,T,h_i}(\tau_1^{(i)},\tau_2^{(i)}) \big)\diff (u^{(i)},\tau_1^{(i)},\tau_2^{(i)}| i\in\nu_r) \bigg]\\
	=&\, 
	\sum_{\{\nu_1,\dots ,\nu_R\}} \int_{[0,1]^{3j}} (-1)^{R-1} (R-1)! \prod_{r=1}^{R} \ex\bigg[ \prod_{i\in\nu_r} \psi_{n_i}\big(u^{(i)},\tau_1^{(i)},\tau_2^{(i)}\big)\\
	&\phantom{=====} 
	\times\frac{1}{\sqrt{T}} \sum_{t=1}^{\lfloor u^{(i)}T\rfloor\wedge (T-h_i)} \big(X_{t,T}(\tau_1^{(i)})X_{t+h_i,T}(\tau_2^{(i)})-\mu_{t,T,h_i}(\tau_1^{(i)},\tau_2^{(i)})\big)\bigg] \\
	& \hspace{9cm}
	\diff (u^{(i)},\tau_1^{(i)},\tau_2^{(i)}| 1\leq i\leq j) \\
	=&\, 
	\int_{[0,1]^{3j}} \cum \bigg( \frac{1}{\sqrt{T}}\sum_{t=1}^{\lfloor u^{(1)}T\rfloor\wedge (T-h_1)}X_{t,T}(\tau_1^{(1)})X_{t+h_1,T}(\tau_2^{(1)})\psi_n\big(u^{(1)},\tau_1^{(1)},\tau_2^{(1)}\big),\dots\\
	&\phantom{=====} 
	\dots, \frac{1}{\sqrt{T}}\sum_{t=1}^{\lfloor u^{(j)}T\rfloor\wedge (T-h_j)}X_{t,T}(\tau_1^{(j)})X_{t+h_j,T}(\tau_2^{(j)})\psi_n(u^{(j)},\tau_1^{(j)},\tau_2^{(j)})\bigg) \\
	& \hspace{9cm}
	\diff (u^{(i)},\tau_1^{(i)},\tau_2^{(i)}| 1\leq i\leq j)\\
	=&\, 
	\int_{[0,1]^{3j}} \frac{1}{T^{j/2}}\sum_{t_1=1}^{T-h_1}\dots\sum_{t_j=1}^{T-h_j}\cum \big(X_{t_1,T}(\tau_1^{(1)})X_{t_1+h_1,T}(\tau_2^{(1)}),\dots, X_{t_j,T}(\tau_1^{(j)})X_{t_j+h_j,T}(\tau_2^{(j)})\big) \\
	&\phantom{=====}
	\times\prod_{i=1}^j \psi_{n_i}\big(u^{(i)},\tau_1^{(i)},\tau_2^{(i)}\big)\id(t_i\leq \lfloor u^{(i)}T\rfloor)  
	\diff (u^{(i)},\tau_1^{(i)},\tau_2^{(i)}| 1\leq i\leq j)
	\end{align*}
	}where the summation extends over all partitions $\{\nu_1,\dots ,\nu_R\}$ of the set $\{1, \dots , j\}$. 	
	The absolute value of this expression is bounded by
	{\small\begin{align*}
	&\frac{1}{T^{j/2}}\sum_{t_1=1}^{T-h_1}\dots\sum_{t_j=1}^{T-h_j} \bigg(\int_{[0,1]^{3j}}\prod_{i=1}^j \psi_{n_i}^2(u^{(i)},\tau_1^{(i)},\tau_2^{(i)})\id(t_i\leq \lfloor u^{(i)}T\rfloor) \diff (u^{(i)},\tau_1^{(i)},\tau_2^{(i)}| 1\leq i\leq j)\bigg)^{1/2} \\
	&\phantom{===}\times\bigg(\int_{[0,1]^{2j}} \cum^2 \big(X_{t_1,T}(\tau_1^{(1)})X_{t_1+h_1,T}(\tau_2^{(1)}),\dots , X_{t_j,T}(\tau_1^{(j)})X_{t_j+h_j,T}(\tau_2^{(j)})\big) \\
	&\hspace{10cm} \diff(\tau_1^{(i)},\tau_2^{(i)}| 1\leq i\leq j)\bigg)^{1/2}\\
	\leq&\,  
	\frac{C}{T^{j/2}}\sum_{t_1=1}^{T-h_1}\dots\sum_{t_j=1}^{T-h_j} \|\cum \big(X_{t_1,T}\otimes X_{t_1+h_1,T},\dots , X_{t_j,T}\otimes X_{t_j+h,T}\big) \|_{2,2j}
	\end{align*}}since, by assumption, $\psi_n(x)\leq C$ uniformly in $x$ and $n$.
	
	In the following, we will bound the expression 
	\[ \sum_{t_1=1}^{T-h_1}\dots\sum_{t_j=1}^{T-h_j} \big\| \cum(X_{t_1,T}\otimes X_{t_1+h_1,T},\dots ,X_{t_j,T}\otimes X_{t_j+h_j,T})\big\|_{2,2j}. \]
	For that purpose, consider the table 
	\begin{equation*}
	S:=\begin{array}{ccc}
	(1,0)&&(1,1)\\
	\vdots&&\vdots\\
	(j,0)&&(j,1),
	\end{array}
	\end{equation*}
	where $j\geq 3$. In the following, the tuple $(i,0)$ will be identified with the index $t_i$ and $(i,1)$ will be identified with $t_i+h_i$. Let $\{\nu_1,\dots ,\nu_q\}$ be a partition of $S$.
	Two sets $\nu_i$ and $\nu_\ell$ of the partition \textit{hook} if there is an index $k$ such that $(k,0)\in\nu_i$ and $(k,1)\in\nu_\ell$ or vice versa. The sets $\nu_i$ and $\nu_\ell$ \textit{communicate} if there is a sequence $\nu_i=\tilde{\nu}_1,\dots ,\tilde{\nu_k}=\nu_\ell$ such that $\tilde{\nu}_{i'}$ and $\tilde{\nu}_{i'+1}$ hook, for any $1\leq i'\leq k-1$. The partition $\{\nu_1,\dots ,\nu_q\}$ is \textit{indecomposable} if all pairs of sets communicate. By Theorem 2.3.2 of \citesuppl{Brillinger1965}, we can rewrite
	\begin{align}\label{cumulantsProductThmApplied}
	\begin{split}
	\cum(X_{t_1,T}X_{t_1+h_1,T},\dots ,X_{t_j,T}X_{t_j+h,T})
	= \sum_{\{\nu_1,\dots ,\nu_q\}} \prod_{k=1}^{q} \cum(X_{t_{s_1},T}^{1-s_2}X_{t_{s_1}+h_{s_1},T}^{s_2}, s\in\nu_k),
	\end{split}
	\end{align}
	where $s=(s_1,s_2)$, where the summation extends over all indecomposable partitions $\{\nu_1,\dots ,\nu_q\}$ of $S$ and where we omit the arguments $\tau_{i}^{\scs (k)}$ for the ease of notation. Observe that $X_{t_{s_1},T}^{1-s_2}X_{t_{s_1}+h_{s_1},T}^{s_2}=X_{t_{s_1},T}$ if $s_2=0$ and $X_{t_{s_1},T}^{1-s_2}X_{t_{s_1}+h_{s_1},T}^{s_2}=X_{t_{s_1}+h_{s_1},T}$ if $s_2=1$.
	
	Clearly, Equation \eqref{cumulantsProductThmApplied} leads to the bound
	\begin{multline*}
	\sum_{t_1=1}^{T-h_1}\dots\sum_{t_j=1}^{T-h_j} \big\| \cum(X_{t_1,T}\otimes X_{t_1+h_1,T},\dots ,X_{t_j,T}\otimes X_{t_j+h_j,T})\big\|_{2,2j}\\
	\leq \sum_{\{\nu_1,\dots ,\nu_q\}} \sum_{t_1=1}^{T-h_1}\dots\sum_{t_j=1}^{T-h_j} \prod_{\ell=1}^{q} \Big\|\cum\Big(X_{t_{s_1},T}^{1-s_2}X_{t_{s_1}+h_{s_1},T}^{s_2}, s\in\nu_\ell\Big)\Big\|_{2,|\nu_\ell|}
\end{multline*}
	Fix an indecomposable partition $\{\nu_1,\dots ,\nu_q\}$ of $S$. If $q=1$, the sum \[\sum_{t_1=1}^{T-h_1}\dots\sum_{t_j=1}^{T-h_j} \Big\|\cum\Big(X_{t_1,T},X_{t_1+h_1,T},\cdots, X_{t_j,T},X_{t_j+h_j,T} \Big)\Big\|_{2,2j}\] is of order $\mathcal{O}(T)$ by \ref{cond:cum}. For $q\ge 2$, there exist $\mu_1,\dots ,\mu_{q-1}$ such that
	\[ \nu_i\cap \big\{(\mu_1,0),\dots ,(\mu_{q-1},0),(\mu_1,1),\dots ,(\mu_{q-1},1) \big\}\neq\varnothing, \]
	for any $i=1,\dots ,q$. Informally speaking, the indices $\mu_1,\dots ,\mu_{q-1}$ 'connect' the sets of the partition. Without loss of generality, we assume $\mu_1=1,\dots ,\mu_{q-1}=q-1$, $(\mu_i,0)=(i,0)\in\nu_i$ and $(\mu_i,1)=(i,1)\in\nu_{i+1}$ for $i=1, \dots, q-1$.  Then,
	{\small\begin{align*}
&\,  \sum_{t_1=1}^{T-h_1}\dots\sum_{t_j=1}^{T-h_j} \prod_{\ell=1}^{q} \Big\|\cum\Big(X_{t_{s_1},T}^{1-s_2}X_{t_{s_1}+h_{s_1},T}^{s_2}, s\in\nu_\ell\Big)\Big\|_{2,|\nu_\ell|} \\
	=&\,  \sum_{t_1=1}^{T-h_1}\dots\sum_{t_j=1}^{T-h_j} \Big\|\cum\Big(X_{t_1,T},X_{t_{s_1},T}^{1-s_2}X_{t_{s_1}+h_{s_1},T}^{s_2}, s\in\nu_1\setminus\{(1,0)\}\Big)\Big\|_{2,|\nu_1|}\\
	&\phantom{=} \times  \Big\|\cum\Big(X_{t_2,T},X_{t_1+h_1,T},X_{t_{s_1},T}^{1-s_2}X_{t_{s_1}+h_{s_1},T}^{s_2}, s\in\nu_2\setminus\big\{(2,0),(1,1)\big\}\Big)\Big\|_{2,|\nu_2|}\\
	&\phantom{=}\cdots\\
	&\phantom{=} \times  \Big\|\cum\Big(X_{t_{q-1},T},X_{t_{q-2}+h_{q-2},T},X_{t_{s_1},T}^{1-s_2}X_{t_{s_1}+h_{s_1},T}^{s_2}, s\in\nu_{q-1}\setminus\big\{(q-1,0),(q-2,1)\big\}\Big)\Big\|_{2,|\nu_{q-1}|}\\
	&\phantom{=} \times  \Big\|\cum\Big(X_{t_{q-1}+h_{q-1},T},X_{t_{s_1},T}^{1-s_2}X_{t_{s_1}+h_{s_1},T}^{s_2}, s\in\nu_q\setminus\{(q-1,1)\}\Big)\Big\|_{2,|\nu_q|}.
	\end{align*}
	}Consider the sets $\tilde{\nu}_1:=\nu_1\setminus\{(1,0)\}, \tilde{\nu}_2:=\nu_2\setminus\{(2,0),(1,1)\},\dots ,\tilde{\nu}_{q-1}:=\nu_{q-1}\setminus\{(q-1,0),(q-2,1)\},\tilde{\nu}_q=\nu_q\setminus\{(q-1,1)\}$, and observe that these sets form a partition of the set $\{(q,0),\dots ,(j,0),(q,1),\dots ,(j,1)\}$. Let $m_i$ be the cardinality of $\tilde{\nu}_i$, for $i=1,\dots ,q$. 
	By adding summands, we can bound the above sum by
	{\small \begin{align*}
	&\sum_{t_1,\dots ,t_{q-1}=1}^{T}\Bigg( \sum_{t_1^{(1)}, \dots, t_{m_1}^{(1)} =-\infty}^{\infty} \Big\|\cum\Big(X_{t_1,T},X_{t_1^{(1)},T},\dots ,X_{t_{m_1}^{(1)},T}\Big)\Big\|_{2,m_1+1}\\
	&\hspace{1.5cm}
	\times \sum_{t_1^{(2)},\dots ,t_{m_2}^{(2)}=-\infty}^{\infty} \Big\|\cum\Big(X_{t_2,T},X_{t_1+h_1,T},X_{t_1^{(2)},T},\dots ,X_{t_{m_2}^{(2)},T}\Big)\Big\|_{2,m_2+2}\\
	&\hspace{1.6cm}\vdots\\
	&\hspace{1.5cm}\times \sum_{t_1^{(q-1)},\dots ,t_{m_{q-1}}^{(q-1)}=-\infty}^{\infty} \Big\|\cum\Big(X_{t_{q-1},T},X_{t_{q-2}+h_{q-2},T},X_{t_1^{(q-1)},T},\dots ,X_{t_{m_{q-1}}^{(q-1)},T}\Big)\Big\|_{2,m_{q-1}+2}\\
	&\hspace{1.5cm} \times\sum_{t_1^{(q)},\dots ,t_{m_q}^{(q)}=-\infty}^{\infty} \Big\|\cum\Big(X_{t_{q-1}+h_{q-1},T},X_{t_1^{(q)},T},\dots ,X_{t_{m_q}^{(q)},T}\Big)\Big\|_{2,m_q+1}\Bigg).
	\end{align*}
	}The last inner sum is bounded by some constant $C_{m_{q+1}}$ by Assumption \ref{cond:cum}. The outer sum over the index $t_{q-1}$ can be pulled in front of the last inner sum and we obtain
	{\small \[ \sum_{t_{q-1}=1}^{T}\sum_{t_1^{(q-1)},\dots ,t_{m_{q-1}}^{(q-1)}=-\infty}^{\infty} \Big\|\cum\Big(X_{t_{q-1},T},X_{t_{q-2}+h,T},X_{t_1^{(q-1)},T},\dots ,X_{t_{_{q-1}}^{(q-1)},T}\Big)\Big\|_{2,m_{q-1}+2} \leq C_{m_{q-1}+2} \]
	}Doing this successively, we have the bound
	\begin{align*}
	&\,C_{m_q+1}\prod_{i=2}^{q-1} C_{m_i+2} \sum_{t_1=1}^{T} \sum_{t_1^{(1)},\dots ,t_{m_1}^{(1)}=-\infty}^{\infty} \Big\|\cum\Big(X_{t_1,T},X_{t_1^{(1)},T},\dots ,X_{t_{m_1}^{(1)},T}\Big)\Big\|_{2,m_1+1} \\
	\leq&\, C_{m_1+1}C_{m_q+1}\bigg(\prod_{i=2}^{q-1} C_{m_i+2}\bigg) T = \mathcal{O}(T). \end{align*}
	We finally obtain that $\cum(\langle \tilde{B}_T,\psi_{n_1}\rangle,\dots ,\langle \tilde{B}_T,\psi_{n_j}\rangle)=\mathcal{O}(T^{1-j/2})$, which vanishes as $T$ tends to infinity since $j\geq 3$. Thus, we have proven the statement for any $w=(w_1,\dots ,w_j)$ with $w_i\in\{0\}\times\{0,\dots ,H\}$. 
	
	In the following, we investigate the cumulant $\overline{\cum}(w)$, for $w=(w_1,\dots ,w_j)$ with $w_i\in\{1\}\times\{0,\dots ,H\}$. The cumulants corresponding to arbitrary $w\in (\{0,1\}\times\{0,\dots ,H\})^j$ can be bounded by using the same arguments.
	By similar arguments as for the case $w_i\in\{0\}\times\{0,\dots ,H\}$, we obtain that
{\small\begin{align}\label{bootstrapCum_j} 
	&\, |\cum(\langle \tilde{B}_{T,h_1}^{(1)},\psi_{n_1}\rangle,\dots ,\langle \tilde{B}_{T,h_j}^{(1)},\psi_{n_j}\rangle)|\nonumber\\
	=&\, 
	\bigg|\int_{[0,1]^{3j}} \cum\big(\tilde{B}_{T,h_1}^{(1)}(u^{(1)},\tau_1^{(1)},\tau_2^{(1)})\psi_{n_1}(u^{(1)},\tau_1^{(1)},\tau_2^{(1)}), \dots \nonumber\\
	&\hspace{2cm} 
	\dots ,\tilde{B}_{T,h_j}^{(1)}(u^{(j)},\tau_1^{(j)},\tau_2^{(j)})\psi_{n_j}(u^{(j)},\tau_1^{(j)},\tau_2^{(j)})\big) \diff (u^{(i)},\tau_1^{(i)},\tau_2^{(i)}|1\leq i\leq j)\bigg|\nonumber\\
	=&\,
	\bigg|\int_{[0,1]^{3j}} \bigg(\prod_{i=1}^{j}\psi_{n_i}(u^{(i)},\tau_1^{(i)},\tau_2^{(i)})\bigg) \frac{1}{(mT)^{j/2}}\sum_{i_1=1}^{\lfloor u^{(1)}T\rfloor\wedge(T-h_1)} \dots \sum_{i_j=1}^{\lfloor u^{(j)}T\rfloor\wedge(T-h_j)} \nonumber\\ 
	&\phantom{=}
	\times\cum\bigg(R_{i_1} \sum_{t={i_1}}^{({i_1}+m-1)\wedge(T-h_1)}X_{t,T}(\tau_1^{(1)})X_{t+h_1,T}(\tau_2^{(1)}),  \dots \nonumber\\
	&\hspace{2cm} \cdots,R_{i_j} \sum_{t={i_j}}^{({i_j}+m-1)\wedge(T-h_j)}X_{t,T}(\tau_1^{(j)})X_{t+h_j,T}(\tau_2^{(j)})\bigg) \diff (u^{(i)},\tau_1^{(i)},\tau_2^{(i)}|1\leq i\leq j)\bigg|\nonumber\\
	&\leq 
	\frac{C}{(mT)^{j/2}}\sum_{i_1=1}^{T-h_1}\dots\sum_{i_j=1}^{T-h_j} \bigg\|\cum\bigg(R_{i_1} \sum_{t={i_1}}^{({i_1}+m-1)\wedge(T-h_1)}X_{t,T}\otimes X_{t+h_1,T},\dots\nonumber\\
	&\phantom{====================}\dots ,R_{i_j} \sum_{t={i_j}}^{({i_j}+m-1)\wedge(T-h_j)}X_{t,T}\otimes X_{t+h_j,T}\bigg)\bigg\|_{2,2j}.
	\end{align}
}

Most of the cumulants on the right-hand side of the above equation are zero. More specific, if there is an index $i_\ell$ with $i_\ell\neq i_{\ell'}$ for any $\ell'\neq \ell$, then by Theorem 2.3.1 (iii) and Theorem 2.3.2 of \citesuppl{Brillinger1965}, the corresponding cumulant in the above sum equals zero. Thus, we can bound the right-hand side of \eqref{bootstrapCum_j} by
\[ 
\frac{C_j}{(mT)^{j/2}}
\sum_{k=1}^{\lfloor j/2\rfloor}
\sum_{\substack{n_1,\cdots,n_k\geq 2\\ \sum_{i=1}^{k}n_i=j}}
\sum_{i_1,\cdots,i_k=1}^{T} 
\big\|\cum(R_{i_1}Y_{i_1},\cdots, R_{i_1}Y_{i_1},\cdots, R_{i_k}Y_{i_k},\cdots, R_{i_k}Y_{i_k})\big\|_{2,2j}, 
\]	
where $n_\ell$ determines how often the product $R_{i_\ell} Y_{i_\ell}$ occurs in the cumulants and 
\[
Y_{i_\ell}=\sum_{t=i_\ell}^{(i_\ell+m-1)\wedge (T-h_\ell)} X_{t,T}\otimes X_{t+h_\ell,T},
\] 
for any $\ell\in\{1,\dots,k\}$. By Theorem 2.3.2 of \citesuppl{Brillinger1965}, we can again rewrite each cumulant in the above sum as a sum over products of cumulants of single random variables, where the sum ranges over all indecomposable partitions of the table 
	\begin{equation*}
\begin{array}{ccc}
R_{i_1}&&Y_{i_1}\\
\vdots&&\vdots\\
R_{i_1}&&Y_{i_1}\\
\vdots&&\vdots\\
R_{i_k}&&Y_{i_k}\\
\vdots&&\vdots\\
R_{i_k}&&Y_{i_k}.
\end{array}
\end{equation*}
By making use of the same technique as before, we can use the indecomposability to prove that 
\[\sum_{\substack{n_1,\cdots,n_k\geq 2\\ \sum_{i=1}^{k}n_i=j}}\sum_{i_1,\cdots,i_k=1}^{T} \big\|\cum(R_{i_1}Y_{i_1},\cdots, R_{i_1}Y_{i_1},\cdots, R_{i_k}Y_{i_k},\cdots, R_{i_k}Y_{i_k})\big\|_{2,2j}\]
is of order $\mathcal{O}(m^{j-(k-1)}T)$. Now we can see that the right-hand side of \eqref{bootstrapCum_j}, and thus, $\cum(\langle \tilde{B}_{T,h_1},\psi_{n_1}\rangle,\dots ,\langle \tilde{B}_{T,h_j},\psi_{n_j}\rangle)$ are of order $\mathcal{O}(m^{j/2}T^{1-j/2})$, which vanishes as $T$ tends to infinity. Similar, $\overline{\cum}(w)$ vanishes for any $w\in(\{0,1\}\times\{0,\dots ,H\})^j$, as $T$ tends to infinity and, by this, $\cum_j(Z_T)$ does so as well.
\end{proof}

\begin{proposition}\label{prop:cov}
	Let Assumptions \ref{cond:ls}-\ref{cond:cum} be satisfied. Then, for any $h,h'\in\N_0,$
	\begin{align*}\frac{1}{T} \sum_{t=1}^{T-h}\sum_{t'=1}^{T-h'} &\|\cov\big((X_{t,T}-X_t^{(t/T)})\otimes X_{t+h,T},X_{t',T}\otimes X_{t'+h',T}\big)\|_{2,4}\\
	&+ \|\cov\big(X_t^{(t/T)}\otimes (X_{t+h,T}-X_{t+h}^{(t/T)}),X_{t',T}\otimes X_{t'+h',T}\big)\|_{2,4}\\
	&+ \|\cov\big(X_t^{(t/T)}\otimes X_{t+h}^{(t/T)},(X_{t',T}-X_{t'}^{(t/T)})\otimes X_{t'+h',T}\big)\|_{2,4}\\
	&+ \|\cov\big(X_t^{(t/T)}\otimes X_{t+h}^{(t/T)},X_{t'}^{(t/T)}\otimes(X_{t'+h',T}-X_{t'+h'}^{(t/T)})\big)\|_{2,4}=\mathcal{O}(T^{-1}).
	\end{align*}
\end{proposition}
\begin{proof}
	To ensure readability, we focus on the sum over the first summand. The other summands can be treated with similar arguments. First, define $Y_{t,T}=X_{t,T}-X_t^{\scs (t/T)}$. From the definition of cumulants, Theorem 2.3.2 of \citesuppl{Brillinger1965} and the triangular inequality, we get the bound
	{\small\begin{align*}
	&\|\cov\big(Y_{t,T}\otimes X_{t+h,T},X_{t',T}\otimes X_{t'+h',T}\big) \|_{2,4}
	\leq \|\cum\big(Y_{t,T}, X_{t+h,T},X_{t',T}, X_{t'+h',T}\big)\|_{2,4}\\
&\phantom{=}+ \|\cum(Y_{t,T})\|_2 \|\cum(X_{t+h,T},X_{t',T}, X_{t'+h',T})\|_{2,3}\\
&\phantom{=}+ \|\cum(X_{t+h,T})\|_2 \|\cum(Y_{t,T},X_{t',T}, X_{t'+h',T})\|_{2,3}\\
&\phantom{=}+ \|\cum(X_{t',T})\|_2 \|\cum(Y_{t,T},X_{t+h,T}, X_{t'+h',T})\|_{2,3}\\
&\phantom{=}+ \|\cum(X_{t'+h',T})\|_2 \|\cum(Y_{t,T},X_{t+h,T}, X_{t',T})\|_{2,3}\\
&\phantom{=}+ \|\cum(Y_{t,T})\|_2 \|\cum(X_{t',T})\|_2 \|\cum(X_{t+h,T}, X_{t'+h',T})\|_{2,2}\\
&\phantom{=}+ \|\cum(X_{t+h,T})\|_2 \|\cum(X_{t',T})\|_2 \|\cum(Y_{t,T}, X_{t'+h',T})\|_{2,2}\\
&\phantom{=}+ \|\cum(Y_{t,T})\|_2 \|\cum(X_{t'+h',T})\|_2 \|\cum(X_{t+h,T}, X_{t',T})\|_{2,2}\\
&\phantom{=}+ \|\cum(X_{t+h,T})\|_2 \|\cum(X_{t'+h',T})\|_2 \cum(Y_{t,T}, X_{t',T})\|_{2,2}\\
&\phantom{=}+ \|\cum(Y_{t,T}, X_{t',T})\|_{2,2} \|\cum(X_{t+h,T}, X_{t'+h',T})\|_{2,2}\\
&\phantom{=}+ \|\cum(Y_{t,T}, X_{t'+h',T})\|_{2,2} \|\cum(X_{t+h,T},X_{t',T})\|_{2,2}\\
&	\leq C \Big\{\|\cum\big(Y_{t,T}, X_{t+h,T},X_{t',T}, X_{t'+h',T}\big)\|_{2,4}\\
&\phantom{=}+ \frac{1}{T} \|\cum(X_{t+h,T},X_{t',T}, X_{t'+h',T})\|_{2,3}
+  \|\cum(Y_{t,T},X_{t',T}, X_{t'+h',T})\|_{2,3}\\
&\phantom{=}+  \|\cum(Y_{t,T},X_{t+h,T}, X_{t'+h',T})\|_{2,3}
+  \|\cum(Y_{t,T},X_{t+h,T}, X_{t',T})\|_{2,3}\\
&\phantom{=}+  \frac{1}{T} \|\cum(X_{t+h,T}, X_{t'+h',T})\|_{2,2}
+  \|\cum(Y_{t,T}, X_{t'+h',T})\|_{2,2}\\
&\phantom{=}+  \frac{1}{T}  \|\cum(X_{t+h,T}, X_{t',T})\|_{2,2}
+  \cum(Y_{t,T}, X_{t',T})\|_{2,2}\\
&\phantom{=}+ \|\cum(Y_{t,T}, X_{t',T})\|_{2,2} \|\cum(X_{t+h,T}, X_{t'+h',T})\|_{2,2}\\
&\phantom{=}+ \|\cum(Y_{t,T}, X_{t'+h',T})\|_{2,2} \|\cum(X_{t+h,T},X_{t',T})\|_{2,2}\Big\},
	\end{align*}
	}where me made use of \eqref{eq:ls} in the second inequality.
	Now, we can investigate the sums over all summands separately. We focus exemplary on three summands, as the remaining summands can be treated with the same arguments. By \ref{cond:cum}, we have
{\small\begin{align*}
	&\hspace{-1cm}\frac{1}{T}\sum_{t=1}^{T-h}\sum_{t'=1}^{T-h'} \|\cum\big(Y_{t,T}, X_{t+h,T},X_{t',T}, X_{t'+h',T}\big)\|_{2,4}\\
	&\leq \frac{1}{T}\sum_{t_1,\cdots,t_4=1}^{T} \|\cum\big(Y_{t_1,T}, X_{t_2,T},X_{t_3,T}, X_{t_4,T}\big)\|_{2,4}\\
	&\leq \frac{1}{T} \sum_{t_1,\cdots,t_4=1}^{T} \frac{1}{T}\eta_4(t_2-t_1,t_3-t_1,t_4-t_1) = \mathcal{O}(T^{-1}).
\end{align*}
}Similarly, 
{\small\begin{align*}
&\hspace{-1cm}\frac{1}{T} \sum_{t=1}^{T-h}\sum_{t'=1}^{T-h'}\frac{1}{T} \|\cum(X_{t+h,T},X_{t',T}, X_{t'+h',T})\|_{2,3}\\
&\leq \frac{1}{T^2} \sum_{t_1,t_2,t_3=1}^{T}  \|\cum(X_{t_1,T},X_{t_2,T}, X_{t_3,T})\|_{2,3}\\
&\leq  \frac{1}{T^2}\sum_{t_1,t_2,t_3=1}^{T} \eta_3(t_2-t_1,t_3-t_1) = \mathcal{O}(T^{-1}).
\end{align*}
}and
{\small\begin{align*}
	&\hspace{-1cm}\frac{1}{T}\sum_{t=1}^{T-h}\sum_{t'=1}^{T-h'}\|\cum(Y_{t,T}, X_{t',T})\|_{2,2} \|\cum(X_{t+h,T}, X_{t'+h',T})\|_{2,2}\\
	&\leq \frac{1}{T}\sum_{t=1}^{T-h}\sum_{t'=1}^{T-h'}\|\ex[(Y_{t,T})^2]\|_2 \|\ex[X_{t',T}^2]\|_2 \|\cum(X_{t+h,T}, X_{t'+h',T})\|_{2,2}\\
	&\leq \frac{C}{T^2} \sum_{t=1}^{T-h}\sum_{t'=1}^{T-h'}\|\cum(X_{t+h,T}, X_{t'+h',T})\|_{2,2}\\
	&\leq \frac{C}{T^2} \sum_{t,t'=1}^{T} \eta_2(t-t')= \mathcal{O}(T^{-1}).
	\end{align*}
}The proof for the third and fourth summand relies on the summability assumption of $(1+|t_j|)\nu_k(t_1,\dots,t_{k-1})$ rather than $\nu_k(t_1,\dots,t_{k-1})$.
\end{proof}

\begin{lemma}\label{ThmTightness}
	Let Assumptions~\ref{cond:ls}--\ref{cond:cum} and \ref{eq:b1} and \ref{eq:b3}  be satisfied. Then,
	\[ \lim\limits_{n\to\infty}\limsup\limits_{T\to\infty} \ex\bigg[\sum_{\ell=n+1}^{\infty}\bigg(\langle \tilde{B}_T,\psi'_\ell\rangle^2 + \langle \tilde{B}_T^{(1)},\psi'_\ell\rangle^2 + \sum_{h=0}^{H} \langle \tilde{B}_{T,h},\psi_\ell\rangle^2 + \langle \tilde{B}_{T,h}^{(1)},\psi_\ell\rangle^2\bigg) \bigg]=0. \]
\end{lemma}

\begin{proof}[Proof of Lemma~\ref{ThmTightness}]
	By linearity of the expectation, we can prove the property for every process separately. We restrict our attention to the cases
	\begin{equation}\label{ThmTightnessPart1}
	\lim\limits_{n\to\infty}\limsup\limits_{T\to\infty} \ex\bigg[\sum_{\ell=n+1}^{\infty}\langle \tilde{B}_{T,h},\psi_\ell\rangle^2 \bigg]=0 
	\end{equation}
	and
	\begin{equation}\label{ThmTightnessPart2}
	\lim\limits_{n\to\infty}\limsup\limits_{T\to\infty} \ex\bigg[\sum_{\ell=n+1}^{\infty} \langle \tilde{B}_{T,h}^{(1)},\psi_\ell\rangle^2 \bigg]=0;
	\end{equation}
	the assertions regarding $\tilde{B}_T$ and $\tilde{B}_T^{\scs (1)}$ follow by similar arguments.
	
	First, by linearity of expectation,	
	\begin{align*}
	0&\leq \limsup\limits_{n\to\infty} \limsup\limits_{T\to\infty} \ex\bigg[\sum_{\ell=n+1}^{\infty} \langle \tilde{B}_{T,h},\psi_\ell\rangle^2\bigg]\\
	&=\limsup\limits_{n\to\infty} \limsup\limits_{T\to\infty} \ex\bigg[\sum_{\ell=1}^{\infty} \langle \tilde{B}_{T,h},\psi_\ell\rangle^2-\sum_{\ell=1}^{n} \langle \tilde{B}_{T,h},\psi_\ell\rangle^2\bigg]\\
	&\leq \limsup\limits_{T\to\infty} \ex\bigg[\sum_{\ell=1}^{\infty} \langle \tilde{B}_{T,h},\psi_\ell\rangle^2\bigg] - \liminf\limits_{n\to\infty} \liminf\limits_{T\to\infty} \ex\bigg[\sum_{\ell=1}^{n} \langle \tilde{B}_{T,h},\psi_\ell\rangle^2\bigg]\\
	&= \limsup\limits_{T\to\infty} \ex\bigg[\sum_{\ell=1}^{\infty} \langle \tilde{B}_{T,h},\psi_\ell\rangle^2\bigg] - \sum_{\ell=1}^{\infty} \liminf\limits_{T\to\infty} \ex\big[\langle \tilde{B}_{T,h},\psi_\ell\rangle^2\big]\\
	&= \limsup\limits_{T\to\infty} \ex\|\tilde{B}_{T,h}\|_{2,3}^2 - \ex\|B_h\|_{2,3}^2,
	\end{align*}
	where we used Equation (D3) from the proof of Proposition \ref{PropCovConv} in the last step. Thus, it is sufficient to prove $\limsup_{T\to\infty} \ex\|\tilde{B}_{T,h}\|_{2,3}^2 \leq \ex\|B_h\|_{2,3}^2$. 
By Fubini's theorem, we have 
	\begin{align*}
		&\, \ex\|\tilde{B}_{T,h}\|_{2,3}^2\\
		=&\, \frac{1}{T}\sum_{t,t'=1}^{T-h} \int_{[0,1]^3}\cov\big(X_{t,T}(\tau_1)X_{t+h,T}(\tau_2),X_{t',T}(\tau_1)X_{t'+h,T}(\tau_2)\big)\id(t,t'\leq\lfloor uT\rfloor) \diff (u,\tau_1,\tau_2).
		\end{align*}
	As in the proof of (D3) in the proof of Proposition \ref{PropCovConv}, we split the above sum into three sums $S_{T,1},S_{T,2},S_{T,3}$ according to $t=t',t<t'$ and $t>t'$, respectively.

For the convergence of the first sum, we obtain, by stationarity,
	\begin{align*}
	S_{T,1}&=\frac{1}{T}\sum_{t=1}^{T-h} \int_{[0,1]^3}\var\big(X_{t,T}(\tau_1)X_{t+h,T}(\tau_2)\big)\id(t\leq\lfloor uT\rfloor) \diff (u,\tau_1,\tau_2)\\
	&= \frac{1}{T}\sum_{t=1}^{T-h} \int_{[0,1]^3}\var\big(X_{t}^{(t/T)}(\tau_1)X_{t+h}^{(t/T)}(\tau_2)\big)\id(t\leq\lfloor uT\rfloor) \diff (u,\tau_1,\tau_2)+\mathcal{O}(T^{-1})\\
	&= \int_{[0,1]^3}\frac{1}{T}\sum_{t=1}^{T-h} \var\big(X_{0}^{(t/T)}(\tau_1)X_{h}^{(t/T)}(\tau_2)\big)\id(t\leq\lfloor uT\rfloor) \diff (u,\tau_1,\tau_2)+\mathcal{O}(T^{-1})\\
	&\xrightarrow{T\to\infty} \int_{[0,1]^3}\int_0^u \var\big(X_0^{(w)}(\tau_1) X_h^{(w)}(\tau_2)\big) \diff w(u,\tau_1,\tau_2).
	\end{align*}
	Next, the double sum involving $t<t'$ can be treated as follows:
	{\small\begin{align*}
	&\, \frac{1}{T}\sum_{t=1}^{(T-h) \wedge \lfloor uT\rfloor}\sum_{t'=t+1}^{(T-h) \wedge \lfloor uT\rfloor} \int_{[0,1]^3}\cov\big(X_{t,T}(\tau_1)X_{t+h,T}(\tau_2),X_{t',T}(\tau_1)X_{t'+h,T}(\tau_2)\big) \diff (u,\tau_1,\tau_2)\\
	=&\, 
	\frac{1}{T} \int_{[0,1]^3} \sum_{t=1}^{\lfloor uT\rfloor}\sum_{t'=t+1}^{\lfloor uT\rfloor}  \cov\big(X_{t}^{(t/T)}(\tau_1)X_{t+h}^{(t/T)}(\tau_2),X_{t'}^{(t/T)}(\tau_1)X_{t'+h}^{(t/T)}(\tau_2)\big) \diff (u,\tau_1,\tau_2)+\mathcal{O}(T^{-1})\\
	=&\, \frac{1}{T} \int_{[0,1]^3} \sum_{t=1}^{\lfloor uT\rfloor}\sum_{k=1}^{\lfloor uT\rfloor-t}\cov\big(X_{t}^{(t/T)}(\tau_1)X_{t+h}^{(t/T)}(\tau_2),X_{k+t}^{(t/T)}(\tau_1)X_{k+t+h}^{(t/T)}(\tau_2)\big)\diff (u,\tau_1,\tau_2)+\mathcal{O}(T^{-1})\\
	=&\, \frac{1}{T} \int_{[0,1]^3} \sum_{t=1}^{\lfloor uT\rfloor}\sum_{k=1}^{\lfloor uT\rfloor-t}\cov\big(X_0^{(t/T)}(\tau_1)X_h^{(t/T)}(\tau_2),X_{k}^{(t/T)}(\tau_1)X_{k+h}^{(t/T)}(\tau_2)\big)\diff (u,\tau_1,\tau_2)+\mathcal{O}(T^{-1}).
	\end{align*}
}By Lebesgue's dominated convergence theorem, the integral and the limit, as $T$ tends to infinity, are interchangeable in the last equality. Thus, the right-hand side converges according to Lemma \ref{LemmaLimitIntegral} to
\[ 
	\int_{[0,1]^3}\sum_{k=1}^{\infty} \int_{0}^{u}\cov\big(X_{0}^{(w)}(\tau_1)X_{h}^{(w)}(\tau_2),X_{k}^{(w)}(\tau_1)X_{k+h}^{(w)}(\tau_2)\big)\diff w \diff (u,\tau_1,\tau_2).  
\]
	
	A similar assertion holds for the double sum involving $t>t'$. Altogether, we obtain that $\ex\|\tilde{B}_{T,h}\|_{2,3}^2$ converges to 
	\begin{multline*}
	 \int_{[0,1]^3}\sum_{k=-\infty}^{\infty} \int_{0}^{u}\cov\big(X_{0}^{(w)}(\tau_1)X_{h}^{(w)}(\tau_2),X_{k}^{(w)}(\tau_1)X_{k+h}^{(w)}(\tau_2)\big)\diff w \diff (u,\tau_1,\tau_2)\\
	= \int_{[0,1]^3} \var\big(\tilde B_h(u,\tau_1,\tau_2)\big) \diff (u,\tau_1,\tau_2) = \ex\|\tilde B_h\|_{2,3}^2
	\end{multline*}
	by Fubini's theorem, which proves \eqref{ThmTightnessPart1}.

	For the proof of \eqref{ThmTightnessPart2} observe that 
	\[ 
	0\leq \lim\limits_{n\to\infty}\limsup\limits_{T\to\infty}\ex\bigg[\sum_{\ell=n+1}^{\infty}\langle \tilde{B}_{T,h}^{(1)},\psi_\ell\rangle^2\bigg]\leq \limsup\limits_{T\to\infty} \ex\|\tilde{B}_{T,h}^{(1)}\|_{2,3}^2-\ex\|\tilde B_h^{(1)}\|_{2,3}^2, \]
	as before, and we can conclude the statement by showing $\limsup_{T\to\infty}\ex\|\tilde{B}_{T,h}^{\scs (1)}\|_{2,3}^2\leq \ex\|\tilde B_h^{\scs (1)}\|_{2,3}^2$. Fubini's theorem and the independence of the family $(R_i)_{i\in\naturals}$ lead to
	\begin{align*} 
	&\, \ex\|\tilde{B}_{T,h}^{(1)}\|_{2,3}^2 \\
	=&\, \ex\bigg[\int_{[0,1]^3} \frac{1}{mT} \sum_{i,i'=1}^{\lfloor uT\rfloor\wedge(T-h)}R_iR_{i'} \sum_{t=i}^{(i+m-1)\wedge(T-h)}\sum_{t'=i'}^{(i'+m-1)\wedge(T-h)}\big\{X_{t,T}(\tau_1)X_{t+h,T}(\tau_2)\\
	&\hspace{2.6cm} -\mu_{t,T,h}(\tau_1,\tau_2) \big\}\big\{X_{t',T}(\tau_1)X_{t'+h,T}(\tau_2)-\mu_{t',T,h}(\tau_1,\tau_2)\big\} \diff (u,\tau_1,\tau_2)\bigg]\\
	=&\, S_{T,1}+S_{T,2}+S_{T,3},
	\end{align*}
	where
	\begin{align*}
	S_{T,1}=&\int_{[0,1]^3}\frac{1}{T}\sum_{i=1}^{\lfloor uT\rfloor\wedge(T-h)} \frac{1}{m}\sum_{t=i}^{(i+m-1)\wedge(T-h)}   \var\big(X_{t,T}(\tau_1)X_{t+h,T}(\tau_2)\big)\diff (u,\tau_1,\tau_2),\\
	S_{T,2}=&\int_{[0,1]^3}\frac{1}{T}\sum_{i=1}^{\lfloor uT\rfloor\wedge(T-h)} \frac{1}{m}\sum_{t=i}^{(i+m-2)\wedge(T-h)}\sum_{t'=t+1}^{(i+m-1)\wedge(T-h)}  \\
	&\hspace{2cm} \cov\big(X_{t,T}(\tau_1)X_{t+h,T}(\tau_2),X_{t',T}(\tau_1)X_{t'+h,T}(\tau_2)\big)
	 \diff (u,\tau_1,\tau_2), \\
	S_{T,3}=&\int_{[0,1]^3}\frac{1}{T}\sum_{i=1}^{\lfloor uT\rfloor\wedge(T-h)} \frac{1}{m}\sum_{t'=i}^{(i+m-2)\wedge(T-h)}\sum_{t=t'+1}^{(i+m-1)\wedge(T-h)}   \\
	&\hspace{2cm} \cov\big(X_{t,T}(\tau_1)X_{t+h,T}(\tau_2),X_{t',T}(\tau_1)X_{t'+h,T}(\tau_2)\big)
	 \diff (u,\tau_1,\tau_2).
	\end{align*}
	We investigate the three previous terms separately. By the same arguments as in the proof of Proposition~\ref{prop:cov} and the stationarity of $(X_t^{\scs (u)})_{t\in\Z}$, we have
	{\small\begin{align*} 
	\begin{split}
	S_{T,1}&=\int_{[0,1]^3}  \frac{1}{T}\sum_{i=1}^{\lfloor uT\rfloor} \frac{1}{m}\sum_{t=i}^{(i+m-1)\wedge(T-h)} \var\big(X_t^{(t/T)}(\tau_1)X_{t+h}^{(t/T)}(\tau_2)\big)\diff (u,\tau_1,\tau_2)+\mathcal{O}(mT^{-1})\\
	&= \int_{[0,1]^3} \frac{1}{T}\sum_{i=1}^{\lfloor uT\rfloor} \frac{1}{m}\sum_{t=i}^{i+m-1} \var\big(X_0^{(t/T)}(\tau_1)X_h^{(t/T)}(\tau_2)\big)\diff (u,\tau_1,\tau_2)+\mathcal{O}(mT^{-1}).
	\end{split}
	\end{align*}
	}For $u<1$, the previous integrand can be rewritten as
	{\small\begin{multline*}
	\frac{1}{T}\sum_{i=1}^{m-1}\frac{i}{m} \var\big(X_0^{(t/T)}(\tau_1)X_h^{(t/T)}(\tau_2)\big)
	+ \frac{1}{T}\sum_{i=m}^{\lfloor uT\rfloor} \var\big(X_0^{(t/T)}(\tau_1)X_h^{(t/T)}(\tau_2)\big)\\
	+ \frac{1}{T}\sum_{i=\lfloor uT\rfloor +1}^{\lfloor uT\rfloor +m-1}\frac{\lfloor uT\rfloor +m-i}{m} \var\big(X_0^{(t/T)}(\tau_1)X_h^{(t/T)}(\tau_2)\big),
	\end{multline*}
	}which implies that
	\[ \lim\limits_{T\to\infty}S_{T,1}=\int_{[0,1]^3}\int_0^u \var\big(X_0^{(w)}(\tau_1),X_h^{(w)}(\tau_2)\big)\diff w \diff (u,\tau_1,\tau_2), \]
	by Lebesgue's dominated convergence theorem. 
The sums $S_{T,2}$ and $S_{T,3}$ can be treated similarly, which finally implies that
	\begin{align*} 
	&\, \limsup_{T\to\infty}\ex\|\tilde{B}_T^{(1)}\|_{2,3}^2 \\
	=&\, \int_{[0,1]^3}\sum_{t=-\infty}^{\infty} \int_0^u \cov\big(X_0^{(w)}(\tau_1)X_h^{(w)}(\tau_2),X_t^{(w)}(\tau_1)X_{t+h}^{(w)}(\tau_2)\big)\diff w\diff (u,\tau_1,\tau_2)\\
	=&\, \int_{[0,1]^3} \var\big(\tilde B(u,\tau_1,\tau_2)\big) \diff (u,\tau_1,\tau_2)=\ex\|\tilde B^{(1)}\|_{2,3}^2. 
	\end{align*}
	Thus \eqref{ThmTightnessPart2} holds true, which proves the lemma.
\end{proof}

%
%
%

\begin{lemma}\label{LemmaLimitIntegral}
	Let $(f_k)_{k\in\naturals}$ be a sequence of integrable functions on the unit interval $[0,1]$, such that $f_k(x)\leq \nu(k)$, for all $x\in[0,1]$, with $\sum_{k=1}^{\infty}\nu(k)<\infty$ and let $(a_n)_{n\in\naturals}$ be a sequence of integers with $a_n\to\infty$ as $n$ tends to infinity. Then, 
	\[ 
	\lim_{n\to\infty} \sum_{k=1}^{a_n} \frac{1}{n} \sum_{\ell=1}^{\lfloor un \rfloor} f_k\big(\tfrac{\ell}{n}\big)
	= 
	\sum_{k=1}^{\infty} \int_0^u f_k(x)\diff x
	\]
	for any $u\in[0,1]$.
\end{lemma}
\begin{proof}
The statement is an immediate consequence of Lebesgue's dominated convergence theorem, applied to the sequence of functions $g_n(k,x) = \ind(k \le a_n) \sum_{\ell=1}^{\scs \lfloor un \rfloor} f_k(\ell/n) \ind(x \in ((\ell-1)/n, \ell/n]).$
\end{proof}

\bibliographystylesuppl{apalike}
\bibliographysuppl{bibliography}

\end{document}